\documentclass[12pt,pdftex]{amsart}

\selectfont

\usepackage{wrapfig}
\usepackage{epic,eepic,ecltree}

\usepackage{graphicx}

\usepackage{amssymb,color, euscript, enumerate}
\usepackage{amsthm}
\usepackage{amsmath}
\usepackage{braket}
\usepackage{amscd}
\usepackage{txfonts}
\usepackage{comment}
\usepackage{bm}
\usepackage{amscd}

\numberwithin{equation}{section}

\newcommand*\pFqskip{8mu}
\catcode`,\active
\newcommand*\pFq{\begingroup
        \catcode`\,\active
        \def ,{\mskip\pFqskip\relax}%
        \dopFq
}
\catcode`\,12
\def\dopFq#1#2#3#4#5{%
        {}_{#1}F_{#2}\biggl[\genfrac..{0pt}{}{#3}{#4};#5\biggr]%
        \endgroup
}

\newcommand{\F}{F_{0,1,\infty}}

\newcommand{\CP}{\C P^1}
\newcommand{\CPm}{\C P^1 \setminus \{0,1,\infty\}}

\newcommand{\N}{\mathbb{N}}
\newcommand{\Z}{\mathbb{Z}}
\newcommand{\R}{\mathbb{R}}
\newcommand{\C}{\mathbb{C}}
\newcommand{\Q}{\mathbb{Q}}

\newcommand{\td}{{\bar{t}}}

\newcommand{\Om}{{\Omega}}

\newcommand{\va}{\bm{1}}
\newcommand{\1}{\bm{1}}
\newcommand{\z}{{\bar{z}}}

\newcommand{\h}{{\bar{h}}}
\newcommand{\p}{{\bar{p}}}

\newcommand{\y}{{\bar{y}}}
\newcommand{\x}{{\bar{x}}}
\newcommand{\uz}{\underline{z}}

\newcommand{\s}{{\bar{s}}}
\newcommand{\n}{{\bar{n}}}

\newcommand{\pa}{{\partial}}
\newcommand{\reg}{{1^{l+r}}}
\newcommand{\regt}{{\text{min}}}

\newcommand{\eve}{{\mathrm{even}}}
\newcommand{\rreg}{{1^{r+r}}}
\newcommand{\wt}{{\mathrm{wt}}}
\newcommand{\dpe}{d_\perp}

\newcommand{\al}{\alpha}
\newcommand{\ep}{\epsilon}
\newcommand{\be}{\beta}
\newcommand{\ga}{\gamma}
\newcommand{\ze}{\zeta}
\newcommand{\de}{\delta}
\newcommand{\om}{\omega}
\newcommand{\la}{\lambda}

\newcommand{\omb}{{\bar{\omega}}}
\newcommand{\si}{\sigma}

\newcommand{\lef}{{\text{left}}}

\newcommand{\Log}{\mathrm{Log}}
\newcommand{\Arg}{\mathrm{Arg}}

\newcommand{\lat}{{\mathrm{lat}}}

\newcommand{\Is}{{\mathrm{IS}}}
\newcommand{\Isr}{{\mathrm{IS}}^{(l,r)}}
\newcommand{\Isrr}{{\mathrm{IS}}^{(r,r)}}

\newcommand{\ft}{\frac{1}{2}}
\newcommand{\fs}{\frac{1}{16}}

\newcommand{\Ld}{{\overline{L}}}
\newcommand{\D}{{\bar{D}}}

\newcommand{\GCor}{{\mathrm{GCor}}}

\newcommand{\End}{\mathrm{End}}

\newcommand{\Vir}{{\mathrm{Vir}}}

\newcommand{\FF}{\underline{\text{Framed full VOAs}}}

\newcommand{\FA}{\underline{\text{Framed algebras}}}

\newtheorem{thm}{Theorem}[section]

\newtheorem{lem}[thm]{Lemma}
\newtheorem{prop}[thm]{Proposition}
\newtheorem{cor}[thm]{Corollary}
\newtheorem{rem}[thm]{Remark}

\newtheorem{mainthm}{Main Theorem}

\begin{document}
\title[Framed full vertex algebras]{}
%
%

\begin{center}
{\LARGE \bf Code conformal field theory and framed algebra
} \par \bigskip

\renewcommand*{\thefootnote}{\fnsymbol{footnote}}
{\normalsize
Yuto Moriwaki \footnote{email: \texttt{moriwaki.yuto (at) gmail.com}}
}
\par \bigskip
{\footnotesize Research Institute for Mathematical Sciences, Kyoto University\\
Kyoto, Japan}

\par \bigskip
\end{center}

\vspace*{8mm}

\noindent
\textbf{Abstract.}

It is known that there are 48 Virasoro algebras acting on the monster conformal field theory. We call conformal field theories with such a property, which are not necessarily chiral, code conformal field theories.

In this paper, we introduce a notion of a framed algebra, which is a finite-dimensional non-associative algebra, and showed that the category of framed algebras and the category of code conformal field theories are equivalent.
We have also constructed a new family of integrable conformal field theories using this equivalence.
These conformal field theories are expected to be useful for the study of moduli spaces of conformal field theories.

\vspace*{8mm}

\begin{center}
{\large \bf Introduction
}
\end{center}
 \par \bigskip

%
%

In physics, the moduli space of two-dimensional conformal field theories has a geometric structure by deformations, and plays an important role in string theory.
Conformal field theories consisting of only holomorphic fields (chiral conformal field theories) were mathematically formulated by Borcherds and are called {\it vertex algebras} or {\it vertex operator algebras} (VOAs) \cite{B,FLM}. Since chiral conformal field theories are not deformable, their moduli space is discrete. For example, it has been predicted by physics that the moduli space of modular invariant chiral conformal field theories with central charge $(24,0)$ consists of $71$ VOAs \cite{Sch}, which has been studied by mathematicians (see for example \cite{DM,LS3}).

One of them is the monster VOA $V^\natural$ whose automorphism group is the monster group, which is the largest sporadic finite simple group.
The VOA $V^\natural$ contains the mutually commutative $48$ Virasoro algebras with central charges $\ft$, $\Vir_\ft^{\oplus 48}$ \cite{DMZ}.
Dong, Griess and H\"{o}hn introduced an algebra called a {\it framed VOA} and investigated such an algebra in general  \cite{DGH}.
A framed VOA is a VOA with central charge $\frac{l}{2}$ that contains the $l$-tensor product of Virasoro VOAs with central charge $\ft$ as a sub VOA.
From a framed VOA $V$, subgroups $C_V,D_V$ of $\Z_2^l$ called {\it codes} are constructed, which play an important role in the study of framed VOAs (see for example \cite{DGH,Mi,LY}).

In this paper,  we call a conformal field theory with central charge $(\frac{l}{2},\frac{r}{2})$
that contains the $l$-tensor produce of Virasoro VOA in the holomorphic part
and the $r$-tensor product of Virasoro VOA in the anti-holomorphic part
an {\it $(l,r)$-code conformal field theory}.
The monster group is an exceptional symmetry, but code conformal field theories are not exceptional, but are considered to be ubiquitous. In fact, of the $71$ chiral conformal field theories with central charge $(24,0)$ that are expected to exist, $56$ are known to be $(48,0)$ code conformal field theories (framed VOA) \cite{La,LS,LS2}.

\begin{wrapfigure}{c}{0.3\textwidth}
\centering
\includegraphics[scale=0.3,width=0.3 \textwidth]{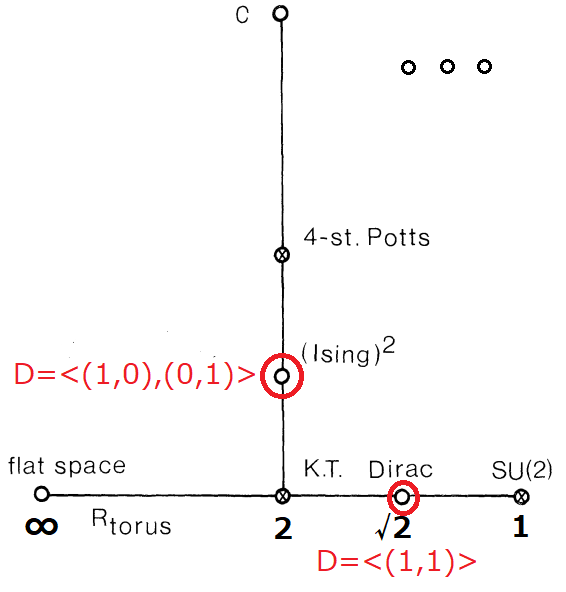}
\begin{minipage}[c]{0.3\textwidth}
\label{fig_moduli}
\caption{}
\end{minipage}
\end{wrapfigure}



Figure 1 shows the moduli space of conformal field theories with central charge $(1,1)$ as expected in physics \cite{Gi,DVV,DVV2}.
 Among them, {\it$(\text{Ising})^2$} and {\it $\text{Dirac}$} are $(2,2)$-code conformal field theories. The line in the figure corresponds to a deformation of conformal field theories. In the case of the central charge $(1,1)$, by deforming the code conformal field theory {\it $(\text{Ising})^2$} and {\it $\text{Dirac}$}, all conformal field theories can be constructed except for the three points in the upper right of the figure.

The purpose of this paper is to construct a family of modular invariant code conformal field theories, and also to construct new conformal field theories by considering their current-current deformation, which is mathematically
constructed in our previous paper \cite{M3}.
There, we introduced an algebra called {\it a full vertex operator algebra} and gave a mathematical formulation of two-dimensional conformal field theories.
In this paper, based on the full vertex operator algebra, we introduce the notion of an $(l,r)$-framed full VOA. An $(l,r)$-framed full VOA is a mathematical formulation of the $(l,r)$-code conformal field theory. We also introduced 
a  notion of {\it an $(l,r)$ framed algebra} which is a non-associative finite dimensional algebra.
The first main result is the following theorem (Theorem \ref{correspondence} and Theorem \ref{equivalence}):
\begin{mainthm}
The category of $(l,r)$-code conformal field theories ($(l,r)$-framed full VOA) and 
the category of $(l,r)$-framed algebras are equivalent as categories.
\end{mainthm}
Therefore, the construction and classification of code conformal field theories can be reduced to the construction and classification of framed algebras.

For a positive integer $r >0 $, a subgroup of $\Z_2^r$ is called {\it a code} (a linear code).
For $g = (g_1,\dots,g_r) \in \Z_2$, $|g|_\Delta = \#\{i\;|\;g_i =1 \} \in \Z_{\geq 0}$ is called {\it the code length}.
A subgroup of $\Z_2^r$ is determined only by the rank,
but by considering the structure of the code length, there are various possible codes for the same rank.
For any code $G \subset \Z_2^r$ satisfying $(1,\dots,1) \in G$,
we explicitly give a finite dimensional vector space $S_G$ and a product $\cdot:S_G\times S_G \rightarrow S_G$ on it. The second result is the following (Theorem \ref{construction}).
%
\begin{mainthm}
For any code $G \subset \Z_2^r$ satisfying $(1,\dots,1) \in G$,
$S_G$ is a simple $(r,r)$-framed algebra.
\end{mainthm}
Together with the first result, we can construct a code conformal field theory $F_{G}$ with central charge $(\frac{r}{2},\frac{r}{2})$.
We also showed that the torus partition function of this conformal field theory is a real analytic modular function
on the upper half plane and is modular invariant (Corollary \ref{cor_modular}). 
The modular invariance is a physically required condition for the theory to be well-defined at higher genus \cite{FMS}.

%
%

Let $\langle a_1,\dots,a_n\rangle$ be the subgroup of $\Z_2^r$ generated by $a_1,\dots,a_n\in \Z_2^r$.
For example, if $r=1$, $G=\langle 1\rangle=\Z_2^1$, then the corresponding conformal field theory is the critical Ising model with central charge $(\ft,\ft)$.

For $r=2$, the central charge of the corresponding code conformal field theory is $(\frac{2}{2},\frac{2}{2})$.
The $G=\langle (1,1) \rangle$ corresponds to {\it $\text{Dirac}$} 
and the $G=\langle (1,0),(0,1) \rangle$ to {\it $\text{Ising}^2$} in the figure.
Also, if $r = 3, G=\langle (1,1,1) \rangle$, then it is the level 2 $\mathrm{SO}(3)$-WZW model. 
In general, if $r \geq 4$, then $G=\langle (1,1,\dots,1)\rangle$ is the level 1 $\mathrm{SO}(r)$-WZW model.

Also, by taking more nontrivial codes $G$, we can construct many new families of conformal field theories.
For example, in Section \ref{sec_classify_code}, we classify the all nontrivial codes $G \subset \Z_2^r$ for $r \leq 6$ as in Table \ref{table_intro} and study the dimension of the frame algebra $S_G$
and the current of the conformal field theory $F_G$ (the Lie algebra of the chiral VOA).

\begin{table}[h]
\caption{all indecomposable code CFTs of rank $r \leq 6$}
\label{table_intro}
  \begin{tabular}{|l|c|c|c|c|} \hline
$r$ & code $G$ & current & $\dim S_G$ & name \\ \hline \hline
1 & $\langle 1 \rangle$ & $0$ & 3 & critical Ising model \\ \hline
2 & $\langle 11 \rangle$ & $\mathrm{SO}(2)$ & 10 & toroidal compactification $R={\sqrt{2}}$ \\ \hline
3 & $\langle 111 \rangle$ & $\mathrm{SO}(3)$ & 36 & $\mathrm{SO}(3)$-WZW model at level 2 \\ \hline
4 & $\langle 1111 \rangle$ & $\mathrm{SO}(4)$ & 136 & $\mathrm{SO}(4)$-WZW model at level 1 \\
 & $\langle1111\rangle^\perp $ & $0$ & $82$ & $G_4^\text{even}$ \\ \hline
5 & $\langle11111\rangle$ &$\mathrm{SO}(5)$& 528 & $\mathrm{SO}(5)$-WZW model at level 1\\
 & $\langle11000,00111,01100\rangle$ & $\mathrm{U}(1)$ & $276$ & $G_5^{2;1,1}$ \\ \hline
6 & $\langle11111\rangle$ &$\mathrm{SO}(6)$& 2080 & $\mathrm{SO}(6)$-WZW model at level 1 \\
& $\langle110000,001111,101000\rangle$& $\mathrm{SO}(3)$ & $1000$ & $G_6^{2;1,1}$\\
& $\langle110000,001111,101100\rangle$&$\mathrm{U}(1)^2$ & $936$ & $G_6^{2;1,2}$ \\ 
& $\langle110000,001100,000011,101010\rangle$ &$0$ & $756$ & $(E_6^{4;2})^\perp$\\
 & $\langle111111\rangle^\perp$& 0 & 730 & $G_6^\text{even}$ \\
\hline
\end{tabular}
\end{table}

An important feature of code conformal field theory is that important physical quantities such as correlation functions can be described combinatorially using codes. For example, an example of a four-point correlation function in code conformal field theory, excluding the scalar factor, is as follows (Proposition \ref{correlator2}):
%
\begin{align*}
&2^{-r}\Pi_{0 \leq i < j \leq 3} \left((z_i-z_j)(\z_i-\z_j)\right)^{-\frac{r}{8}}F(z_0,z_1,z_2,z_3)^{r-|\al^0\al^1+\al^0\al^2+\al^0\al^3+\al^1\al^2+\al^1\al^3+\al^2\al^3|_l}\\
&G_{01,23}(z_0,z_1,z_2,z_3)^{|\al^0\al^1+\al^2\al^3|_l}
G_{02,13}(z_0,z_1,z_2,z_3)^{|\al^0\al^2+\al^1\al^3|_l}
G_{03,12}(z_0,z_1,z_2,z_3)^{|\al^0\al^3+\al^1\al^2|_l}
\end{align*}
where $F(z_0,z_1,z_2,z_3)$ and $G_{ij,kl}(z_1,z_2,z_3,z_4)$ are defined by
\begin{align*}
F(z_0,z_1,z_2,z_3)
&= \left(|(z_0-z_1)(z_2-z_3)|^\ft
+|(z_0-z_2)(z_1-z_3)|^\ft
+|(z_0-z_3)(z_1-z_2)|^\ft \right)^\ft,\\
G_{01,23}(z_0,z_1,z_2,z_3)
&=\left(-|(z_0-z_1)(z_2-z_3)|^\ft
+|(z_0-z_2)(z_1-z_3)|^\ft
+|(z_0-z_3)(z_1-z_2)|^\ft \right)^\ft,\\
G_{02,13}(z_0,z_1,z_2,z_3)
&=\left( |(z_0-z_1)(z_2-z_3)|^\ft
-|(z_0-z_2)(z_1-z_3)|^\ft
+|(z_0-z_3)(z_1-z_2)|^\ft \right)^\ft,\\
G_{03,12}(z_0,z_1,z_2,z_3)
&=\left( |(z_0-z_1)(z_2-z_3)|^\ft
+|(z_0-z_2)(z_1-z_3)|^\ft
-|(z_0-z_3)(z_1-z_2)|^\ft \right)^\ft.
\end{align*}

Although $G_{ij,kl}(z_1,z_2,z_3,z_4)$ is not single-valued on the configuration space $X_4(\CP)=\{(z_0,z_1,z_2,z_3)\in (\CP)^4\;|\; z_i\neq z_j\}$ due to the existence of square roots, 
their monodromies cancel each other and the correlation function defines a single-valued real analytic function on $X_4(\CP)$.
Also, $\al^0,\al^1,\al^2,\al^3$ are elements in $\Z_2^{r+r}$, and ${|\al^0\al^1+\al^2\al^3|_l}$ etc. are integers determined by the code length. Thus, the code conformal field theories give a family of integrable (various physical quantities can be computed) conformal field theories.


It is believed in physics that conformal field theory can be deformed by a field of conformal weight $(1,1)$ which satisfies a good property called {\it exactly marginal}.
If a $(1,1)$ field is the product of a $(1,0)$ field and a $(0,1)$ field (called {\it current fields}), 
then the conditions for the field to be exactly marginal is known in physics \cite{CS}.
The deformation generated by such a $(1,1)$ field is called {\it a current-current deformation} in physics \cite{FR}.
In \cite{M3}, the current-current deformation of a conformal field theory is mathematically constructed in terms of a full VOA. The conditions for a code conformal field theory $F_G$ to admit a current-current deformation can be described combinatorially in terms of codes.
To be more precise, the dual code $G^\vee$ is defined by $G^\vee=\{\al \in \Z_2^r\;|\; |\al G|_\Delta \subset 2\Z\}$.
Then, we obtained the following result (Theorem \ref{deform_correlator}).
\begin{mainthm}
Assume that there are mutually orthogonal length two elements $\al_1,\dots,\al_N$ in the dual code $G^\vee$,
i.e., $|\al_i\al_j|_\Delta=0$ for $i \neq j$ and $|\al^i|=2$ for any $i =1,\dots,N$.
Then, the framed full VOA $F_G$ admits a current-current deformation parametrized by $\mathrm{O}(N,N)/\mathrm{O}(N)\times\mathrm{O}(N)$.
Here, $\mathrm{O}(a,b)$ is the orthogonal group with the signature $(a,b)$.
\end{mainthm}
For example, if $G=\langle (1,1) \rangle$, then $G^\vee=\langle (1,1) \rangle$
and thus $F_G$ admits a one-dimensional deformation family $\mathrm{O}(1,1)/\mathrm{O}(1)\times\mathrm{O}(1)\cong \R$. In fact, this family of deformations corresponds to the transversal line through the {\it Dirac} in the figure.

Combining this result with Main theorem 2, we can construct a continuous family of conformal field theories.
The four-point correlation function of the deformed CFT $F_G$ can also be calculated explicitly in terms of codes again.
An example of the four-point correlation function of the deformed code conformal field theory corresponding to $\si \in \mathrm{O}(N,N)$ is as follows (see Section \ref{sec_cc_def} for the precise definition):
\begin{align*}
2^{-r+3N}F(z_0,z_1,z_2,z_3)^{r-2N} 
\Pi_{0 \leq i < j \leq 3} \left((z_i-z_j)(\z_i-\z_j)\right)^{\frac{1}{4}\left(p \si^{-1} (s^i,s^i) ,p \si^{-1} (s^j,s^j)\right)_l + \frac{1}{4}N -\frac{r}{8}}.
\end{align*}

In the following, we will briefly explain the basic tools for study coded conformal field theories: finite dimensional algebra, the framed algebra, and Main theorem 1.

\noindent
\begin{center}
{0.1. \bf 
What is the framed algebra?
}
\end{center}

The fields of a two-dimensional conformal field theory are operator-valued real analytic functions and have an algebra structure.
The aforementioned chiral conformal field theory (vertex operator algebra) is a subalgebra of the conformal field theory consisting of holomorphic fields (see Proposition \ref{vertex_algebra}).
A conformal field theory is a module of the vertex operator algebras of holomorphic part $V_L$ and anti-holomorphic part $\overline{V_R}$.
Therefore, the construction of conformal field theories is done by first classifying the modules of the vertex operator algebras and then combining the modules to satisfy the axiom \cite{MS2}.
The combination of modules satisfying the axiom corresponds to the frobenius algebra object in the representation category of $V_L \otimes \overline{V_R}$ \cite{FRS}.

The above considerations in physics were considered mathematically by Huang and Lepowsky.
In particular, they showed that the representation category of a good vertex operator algebra has the structure of a braided tensor category \cite{HL1,HL2,HL3,H1}. Furthermore, in \cite{HK}, they constructed a conformal field theory called {\it the diagonal model}, which is the diagonal combination of all the irreducible modules of a VOA.


On the other hand, there are many important non-diagonal conformal field theories. Our code conformal field theory is one such example.
Another example is the case where the vertex operator algebra of the holomorphic and anti-holomorphic parts are the affine Heisenberg vertex operator algebra.
Then, the frobenius algebra object in the representation category 
can be classified by using a certain class of associative algebras that we called {\it AH-pairs} \cite{M1}.
AH pairs are classified using the cohomology of groups, and from an AH-pair we can construct a non-diagonal conformal field theory \cite{M2}.
These non-diagonal conformal field theories are geometrically meaningful conformal field theories arising from the toroidal compactification of string theory (flat target space) \cite{P}.
However, the braided tensor category structure of the representation category of the affine Heisenberg vertex operator algebra is almost trivial, and the resulting theory is also trivial.


%

What we investigate in this paper are conformal field theories with symmetry of the Virasoro algebras $\Vir_\ft^{\oplus l}\oplus \Vir_\ft^{\oplus r}$.
Let $L(\ft,0)$ be a simple Virasoro vertex operator algebra of central charge $\ft$ \cite{FZ}
and $L_{l,r}(0)$
be a full vertex operator algebra defined by $L_{l,r}(0)=L(\ft,0)^{\otimes l}\otimes \overline{L(\ft,0)^{\otimes r}}$,
the tensor product of VOA $L(\ft,0)^{\otimes l}$
and conjugate VOA $\overline{L(\ft,0)^{\otimes r}}$ (see Section \ref{sec_full_VOA}).
Following the framed vertex operator algebra, we call a full vertex operator algebra
which contains $L_{l,r}(0)$ as a sub full VOA {\it an $(l,r)$-framed full vertex operator algebra}.

The module of the Virasoro vertex operator algebra $L(\ft,0)$ is classified and is known to be the direct sum of three irreducible modules:
$L(\ft,0), L(\ft,\ft), and L(\ft,\fs)$. It is also known that its representation category has a nontrivial braided tensor category structure \cite{FF,DMZ,Wa}.
Thus $L_{l,r}(0)$ has $3^{l+r}$ irreducible modules and the representation category is nontrivial.

Now let us recall the example of conformal field theories constructed from the above affine Heisenberg vertex operator algebra.
It is important to note that in the algebra $F$ of the conformal field theory, the lowest weight space $\Om_{F,H}$ of the affine Heisenberg Lie algebra $\hat{H}$ has an algebra structure \cite{M3}.
In particular, if the affine Heisenberg vertex operator algebra is sufficiently large in $F$, then $\Om_{F,H}$ becomes an AH pair, i.e., an associative algebra \cite{M1,M3}.
%
For a code conformal field theory $F$, the lowest weight space $S_F$ of the Virasoro algebras is expected to inherit an algebra structure.
The axiomatization of this algebra structure is the framed algebra (see Section \ref{sec_framed_algebra}).

The framed algebra is a non-commutative and non-associative finite dimensional algebra, reflecting the fact that the representation category of $L(\ft,0)$ is nontrivial.
The Main Theorem 1 states that
the functor which takes the lowest weight space of the Virasoro algebras gives the equivalence between the category  of framed full VOAs and the category of framed algebras.

\noindent
\begin{center}
{0.2. \bf 
Definition of framed algebra.
}
\end{center}
In this section, we will give a more detailed definition of a framed algebra.
For simplicity, let $l=r=1$.
In this case, the code conformal field theory is $F=\bigoplus_{h,\h \in \{0,\ft,\fs\}}
\left(L(\ft,h)\otimes \overline{L(\ft,\h)}\right)^{n_{h,\h}}$, where $n_{h,\h} \in \Z_{\geq 0}$ is the multiplicity.
For convenience of explanation, we further assume that $n_{h,\h}=\delta_{h,\h}$ 
and set $L_{1,1}(h,\bar{h})=L(\ft,h)\otimes \overline{L(\ft,\h)}$ for $h \in \Is=\{0,\ft,\fs\}$.
That is, $F= \bigoplus_{h \in \Is}
L_{1,1}(h,\bar{h})$ (diagonal model).
A full vertex operator algebra is a vector space $F$ equipped with a linear operator, called a {\it vertex operator}
$$Y(-,\uz): F\rightarrow \End F[[z,\z,|z|^\R]],\;
a \mapsto Y(a,\uz)=\sum_{r,s\in \R} a(r,s)z^{-r-1}\z^{-s-1},
$$
which satisfies some axiom (Section \ref{sec_full_VOA}).

%


From the decomposition of $F$ as an $L_{1,1}(0)$-module,
the vertex operator is also decomposed into the sum of 
$$Y{\binom{(h_1,\h_1)}{(h_2,\h_2),(h_3,\h_3)}}(-,\uz): L_{1,1}(h_2,\h_2) \rightarrow \mathrm{Hom}\left(L_{1,1}(h_3,\h_3),L_{1,1}(h_1,\h_1)\right).$$
By \cite{DMZ}, $Y{\binom{(h_1,\h_1)}{(h_2,\h_2),(h_3,\h_3)}}(-,\uz)=Y{\binom{h_1}{h_2,h_3}}(-,z)\otimes Y{\binom{\h_1}{\h_2,\h_3}}(-,\z)$,
that is, the tensor product of a holomorphic vertex operator and an anti-holomorphic vertex operator (see Section \ref{subsec_lowest}).
From the axioms of the full vertex operator algebra, $Y{\binom{h_1}{h_2,h_3}}(-,\uz)$ is an intertwining operator of 
modules of a vertex operator algebra introduced in \cite{FHL} (Lemma \ref{chiral_intertwining}).
The intertwining operators among the $L(\ft,0)$-modules are classified (see for example \cite{DMZ}).
For example, there is no $Y{\binom{\ft}{\ft,\fs}}(-,z)$.
The allowed triples are called the {\it fusion rule}, which is a fundamental quantity in the representation theory of a vertex operator algebra.
For $\la_2,\la_3$, by setting $\la_2\star \la_3$ to be the set of all possible $\la_1$, the fusion rule forms the product $\star: \Is \times \Is \rightarrow P(\Is)$.

The most important condition of the axiom of a full vertex operator algebra is, roughly speaking, the following:
\begin{itemize}
\item[FL)]
For any $a_1,a_2,a_3 \in F$ and $u \in F^\vee$ (the dual vector space of $F$),
The formal power series $u(Y(a_1,\uz_1)Y(a_2,\uz_2)a_3)$ and $u(Y(a_2,\uz_2)Y(a_1,\uz_1)a_3)$ with formal variables $z_1,\bar{z}_1,z_2,\bar{z}_2$ converge to a single-valued real analytic function on $Y_2(\C^\times)$ in $\{(z_1,z_2) \in Y_2(\C^\times)\;|\; |z_1|>|z_2|\}$ and $\{(z_1,z_2) \in Y_2(\C^\times)\;|\; |z_2|>|z_1|\}$, respectively, and their analytic continuations are the same function.
Here, $Y_2(\C^\times)=\{(z_1,z_2)\in \C \;|\;z_1\neq z_2,z_1\neq 0,z_2\neq 0 \}$.
\end{itemize}
The above formal power series $u(Y(a,\uz_1)Y(b,\uz_2)c)$ is called a {\it correlation function}.
Correlation functions are the most important physical quantity of conformal field theory.
%
Following \cite{MS2} and \cite{H1,H2,HK}, the code conformal field theory can be constructed in the following two steps:
%
\begin{enumerate}
\item[S1)]
We first compute the correlation functions of all possible combinations of intertwining operators $Y\binom{h_1}{h_2,h_3}(-,z)$ from the representation theory of the Virasoro vertex operator algebra. These functions are called {\it conformal blocks} and has monodromies, i.e., is a multivalued holomorphic function on $Y_2(\C^\times)$.
\item[S2)]
By combining holomorphic and anti-holomorphic intertwining operators, we second achieve a monodromy invariant correlation function that satisfies [FL].
\end{enumerate}
The first step (S1) was basically done by \cite{BPZ}.
In this paper, the calculation will be done in Section \ref{sec_conformal_block}.
It is important to note that even if we determine the input states, $h_0,h_1,h_2,h_3 \in \Is$ and $u \in L(\ft,h_0)^\vee$ and $a_i \in L(\ft,h_i)$ ($i=1,2,3$), the conformal block is not determined.
In fact, the conformal block is given by $u(Y\binom{h_0}{h_1,h}(a_1,z_1)Y\binom{h}{h_2,h_3}(a_2,z_2)a_3)$, so it depends on the intermediate state $h \in \Is$.
The set of all possible intermediate states in the four states scattering is given by
$$.
A(h_0,h_1,h_2,h_3)=\{h_0 \in h_1\star h, h \in h_2 \star h_3\}.
$$
From the Ward identity (Lemma \ref{ward}), we may assume that $u,a_1,a_2,a_3$ are the lowest weight vector $\bra{h_0},\ket{h_1},\ket{h_2},\ket{h_3}$.
In this case, the nontrivial conformal blocks 
$$C_{h_0,h_1,h_2,h_3}^{h}(z_1,z_2)=
\bra{h_0}Y\binom{h_0}{h_1,h}(\ket{h_1},z_1)Y\binom{h}{h_2,h_3}(\ket{h_2},z_2)\ket{h_3}
$$
for $h \in A(h_0,h_1,h_2,h_3)$
are given in table below (for the precise statement, see Proposition \ref{list_block2}).
%
\begin{table}[h]
\caption{Conformal blocks}
\label{intro_table_block2}
  \begin{tabular}{|l|c||c|} \hline
$(h_0,h_1,h_2,h_3)$ & $h$ & $C_{h_0,h_1,h_2,h_3}^{h}(z_1,z_2)$ \\ \hline \hline
$(\frac{1}{16},\frac{1}{16},\frac{1}{16},\frac{1}{16}) $
& $0$ & $\frac{1}{2}\{z_1z_2(z_1-z_2)\}^{-\frac{1}{8}}
\Bigl((z_1^{\frac{1}{2}}+z_2^{\frac{1}{2}})^{\frac{1}{2}} +
(z_1^{\frac{1}{2}}-z_2^{\frac{1}{2}})^{\frac{1}{2}} \Bigr)$
 \\
    & $\frac{1}{2}$ & $\frac{1}{2} \{z_1z_2(z_1-z_2)\}^{-\frac{1}{8}}
\Bigl((z_1^{\frac{1}{2}}+z_2^{\frac{1}{2}})^{\frac{1}{2}} - 
(z_1^{\frac{1}{2}}-z_2^{\frac{1}{2}})^{\frac{1}{2}} \Bigr)$
 \\ \hline
$(\frac{1}{2},\frac{1}{2},\frac{1}{2},\frac{1}{2}) $
& $0$ & $\{z_1z_2(z_1-z_2)\}^{-1}(z_1^2-z_1z_2+z_2^2)$ \\ \hline

$(\frac{1}{2},\frac{1}{2},\frac{1}{16},\frac{1}{16}) $
& $0$ & $z_2^{-\frac{1}{8}}\{z_1(z_1-z_2)\}^{-\frac{1}{2}}
(z_1-\frac{z_2}{2})$ \\
$(\frac{1}{2},\frac{1}{16},\frac{1}{2},\frac{1}{16}) $
& $\frac{1}{16}$ &
$\frac{1}{2}z_1^{-\frac{1}{8}}\{z_2(z_1-z_2)\}^{-\frac{1}{2}}
(z_1-2z_2)$ \\
$(\frac{1}{16},\frac{1}{2},\frac{1}{2},\frac{1}{16}) $
& $\frac{1}{16}$ & $\frac{1}{2}(z_1z_2)^{-\frac{1}{2}} (z_1-z_2)^{-1}(z_1+z_2)$\\
$(\frac{1}{2},\frac{1}{16},\frac{1}{16},\frac{1}{2}) $
& $\frac{1}{16}$ & $\frac{1}{2}\{z_1z_2\}^{-\frac{1}{2}}(z_1-z_2)^{-\frac{1}{8}}(z_1+z_2)$ \\
$(\frac{1}{16},\frac{1}{2},\frac{1}{16},\frac{1}{2}) $
& $\frac{1}{16}$ & $\frac{1}{2}z_1^{-1}\{z_2(z_1-z_2)\}^{-\frac{1}{2}}(z_1-2z_2)$\\
$(\frac{1}{16},\frac{1}{16},\frac{1}{2},\frac{1}{2}) $
& $0$ & $z_2^{-1}\{z_1(z_1-z_2)\}^{-\frac{1}{2}}(z_1-\frac{z_2}{2})$ \\ \hline
\end{tabular}
\end{table}
As you can see from the table, the conformal block has branches along $\{z_1=0\}\cup\{z_2=0\}\cup\{z_1=z_2\}$. In particular, it is not single-valued on $Y_2(\C^\times)$.
Note that in general, $u(Y(a_1,z_1)Y(a_2,z_2)a_3)$ converges at $|z_1|>|z_2|$.
So we take a path $\ga:[0,1]\rightarrow Y_2(\C^\times)$ connecting a point in the region $|z_1|>|z_2|$ with a point in the region $|z_2|>|z_1|$.
Then, the analytic continuation of 
$C_{h_0,h_1,h_2,h_3}^{h}(z_1,z_2)$ along the path $\ga$ 
is a linear sum of $\{C_{h_0,h_2,h_1,h_3}^{h'}(z_2,z_1)\}_{h' \in A(h_0,h_2,h_1,h_3)}$.
Thus, we can define {\it a connection matrix} $B_{h_0,h_1,h_2,h_3}^{h,h'}(\ga) \in \C$
for $(h,h') \in A(h_0,h_1,h_2,h_3)\times A(h_0,h_2,h_1,h_3)$ by
\begin{align*}
A_\ga \Bigl(C_{h_0,h_1,h_2,h_3}^{h}(z_1,z_2)\Bigr)
= \sum_{h' \in A(h_0,h_2,h_1,h_3)} B_{h_0,h_1,h_2,h_3}^{h,h'}(\ga)
C_{h_0,h_2,h_1,h_3}^{h'}(z_2,z_1)
\end{align*}
(see Section \ref{sec_conformal_block} for more precise definition).
In Section \ref{sec_conn}, we determine $B_{h_0,h_1,h_2,h_3}^{h,h'}(\ga) \in \C$
for some path $\ga_0:[0,1]\rightarrow Y_2(\C^\times)$ (Proposition \ref{connection_one}).
It is easy to turn the above result at $l=r=1$ into a general $l,r \in \Z_{\geq 0}$.
Set $\Isr=\Is^l \times \Is^r$. 
This set parametrizes all the irreducible modules of $L_{l,r}(0)=L(\ft,0)^{\otimes l}\otimes L(\ft,0)^{\otimes r}$.
%
For $\la^i=(h_1^i,\dots,h_l^i, \bar{h}_1^i\dots,\bar{h}_r^i) \in 
\Isr$, $\la=(h_1,\dots,h_l,\h_1,\dots,h_r) \in A(\la^0,\la^1,\la^2,\la^3)$ and $\la'=(h'_1,\dots,h'_l,\h'_1,\dots,\h'_r) \in A(\la^0,\la^2,\la^1,\la^3)$,
{\it a multi-index connection matrix} is defined by
$$
B_{\la^0,\la^1,\la^2,\la^3}^{\la,\la'}
\equiv \Pi_{i=1}^l B_{h_i^0,h_i^1,h_i^2,h_i^3}^{h_i,h'_i} \Pi_{j=1}^r 
\overline{B_{\h_j^0,\h_j^1,\h_j^2,\h_j^3}^{\h_j,\bar{h}'_j}} \in \C,
$$
where $\overline{B_{\h_j^0,\h_j^1,\h_j^2,\h_j^3}^{\h_j,\bar{h}'_j}}$ is the complex conjugate of $B_{\h_j^0,\h_j^1,\h_j^2,\h_j^3}^{\h_j,\bar{h}'_j} \in \C$.
It is easy to show that
the multi-index connection matrix 
gives the connection matrix for 
the multi-index conformal blocks,
which is defined naturally for $L_{l,r}(0)$-modules and
their intertwining operators.

Based on the above preparation, the framed algebra can be defined as follows.
Let $S=\bigoplus_{\la \in \Isr} S_\la$ be a finite-dimensional $\Isr$-graded vector space 
equipped with a linear map $\cdot: S \otimes S \rightarrow S$
and a distinguished non-zero element $1 \in S_0$ such that:
\begin{enumerate}
\item[FA1)]
For any $\la=(h_1,\dots,h_l,\h_1,\dots,h_r) \in \Isr$, $S_\la =0$ unless $\sum_{i=1}^l h_i - \sum_{j=1}^r \h_j \in \Z$.
\item[FA2)]
$S_{0^{l+r}}=\C 1$ and 
for any $a\in S$,
$a\cdot 1=1\cdot a=a$;
\item[FA3)]
For any $\la^1,\la^2 \in \Isr $, $a \in S_{\la^1}$ and $b\in S_{\la^2}$,
$a \cdot b \in \bigoplus_{\la \in \la^1 \star \la^2}S_{\mu}.$
\item[FA4)]
For any $\la^i \in \Isr$, $a_i \in S_{\la^i}$
and $\la' \in A(\la^0,\la^2,\la^1,\la^3)$ (i=1,2,3),
$$
a_2\cdot_{\la_0} ( a_1 \cdot_{\la'} a_3 )
= \sum_{\la \in A(\la^0,\la^1,\la^2,\la^3)}
B_{\la^0,\la^1,\la^2,\la^3}^{\la,\la'}
a_1\cdot_{\la_0}(a_2 \cdot_\la a_3),
$$
where $\cdot_\la$ is the composition of
$\cdot :S\otimes S\rightarrow S$ and the projection
$S \rightarrow S_\la$.
\end{enumerate}
For a framed algebra $S$,
we can naturally define a vertex operator $Y_S(-,\uz)$ on the $L_{l,r}(0)$-module $F_S=\bigoplus_{\la \in \Isr} L_{l,r}(\la)\otimes S_\la$.
Condition (FA1) implies that all the correlation functions of $Y_S(-,\uz)$ are single-valued around $z_2=0$.
And (FA2) means that all the correlation functions of $Y_S(-,\uz)$ are single-valued around $z_1=z_2$.
The scale transformation and rotation $\C^\times$ acts on $Y_2(\C^\times)$ with $(x,y)\mapsto (\rho x, \rho y)$.
Hence, using conformal invariance of the vertex operator, the correlation functions can be viewed as a function on $Y_2/\C^\times =\CPm$.
Since the fundamental group $\pi_1(\CPm)$ is generated by two elements, (FA1) and (FA2) imply that all the correlation functions is single-valued real analytic functions on $Y_2(\C^\times)$ and satisfy (FL). 
Conversely, given a framed full vertex operator algebra $F$, we can show that its lowest weight space $S_F$ becomes a framed algebra (Theorem \ref{correspondence}).

The most important observation of this paper is that 
the multi-index connection matrices $B_{\la^0,\la^1,\la^2,\la^3}^{\la,\la'}$ have a very simple combinatorial representation as follows (Theorem \ref{connection}):
\begin{align*}
B_{(d^0,c^0),(d^1,c^1),(d^2,c^2),(d^3,c^3)}^{(d^2+d^3, c), (d^1+d^3,c')}
&=(-1)^{|c^1c^2|}
(-1)^{|d^1c^2(c^0+c^3)|+|d^2c^1(c^0+c^3)|}
i^{-|d^1c^2|-|d^2c^1|+
|d^1d^2(c^0+c^3)|}\\
& \exp(\frac{-\pi i}{8}|d^1d^2|)
(\frac{1+i}{2})^{|d^1d^2d^3|_l}
(\frac{1-i}{2})^{|d^1d^2d^3|_r}
(-i)^{|d^1d^2d^3(c+c')|}.
\end{align*}
The meaning of the symbols is not explained here, but $d^i,c^i,c,c'\in \Z_2^{l+r}$
and $\la^i=(d^i,c^i)$ and $|-|_l, |-|_r, |-|:\Z_2^{l+r}\rightarrow \Z$ are maps defined by the code length (see Section \ref{sec_multi_conn}).
By using this combinatorial representation, the structure of a framed algebra can be studied, and in particular Main Theorem 2 is proved.

Finally we give the simplest non-trivial example of a framed algebra here.
The following combination of conformal blocks define a single-valued real analytic function on $Y_2(\C^\times)$:
\begin{align}
\ft &(z_1\z_1 z_2\z_2(z_1-z_2)(\z_1-\z_2))^{-\frac{1}{8}} \left((z_1\z_1)^\ft+(z_2\z_2)^\ft+((z_1-z_2)(\z_1-\z_2))^\ft
\right)^\ft \nonumber \\
&=
C_{\frac{1}{16},\frac{1}{16},\frac{1}{16},\frac{1}{16}}^0(z_1,z_2)
\overline{C_{\frac{1}{16},\frac{1}{16},\frac{1}{16},\frac{1}{16}}^0(z_1,z_2)}+
C_{\frac{1}{16},\frac{1}{16},\frac{1}{16},\frac{1}{16}}^{\ft}(z_1,z_2)
\overline{C_{\frac{1}{16},\frac{1}{16},\frac{1}{16},\frac{1}{16}}^{\ft}(z_1,z_2)}. \label{intro_combination}
\end{align}
Therefore, it can be a correlation function. In fact, there is a framed algebra and a code conformal field theory that realize this correlation function.
Let $S_{\mathrm{Ising}}=\C 1 \oplus \C a \oplus \C d$
be a three-dimensional vector space
with the $\Is^{(1,1)}$-grading,
$S_{0,0}=\C 1$, $S_{\frac{1}{2},\frac{1}{2}}=\C a$ and
$S_{\frac{1}{16},\frac{1}{16}}=\C d$.
Define a product on $S_{\mathrm{Ising}}$ by 
$a \cdot a =1$, $a \cdot d=d \cdot a=1$,
and $d \cdot d=1+a$ with unit $1$.
It is easy to verify that this algebra is a framed algebra.
Also, the product $d\cdot d=1+a$ corresponds to \eqref{intro_combination}. The conformal field theory corresponding to $S_{\mathrm{Ising}}$ is the critical Ising model in physics.
In fact, \eqref{intro_combination} agrees with the known correlation function of the critical Ising model \cite{BPZ}.
As mentioned in the introduction, $S_{\mathrm{Ising}}$ is one of the family of framed algebras constructed in this paper with $r=1, G=\langle 1 \rangle$.
%
%
%
%
%
%
%
%
%

\vspace{4mm}

This paper is organized as follows.
In Section \ref{sec_preliminary},
we collect some fundamental results concerning the representation theory of the Virasoro algebra of central charge $\ft$ \cite{}.
In Section \ref{sec_full_VOA},
we recall the definition and some results of a full vertex operator algebra and introduce notions of framed full vertex operator algebra and Virasoro conformal blocks of central charge $\ft$ (see also \cite{BPZ}).
The Virasoro conformal blocks will be studied in Section \ref{sec_conformal_block}. We explicitly calculate all the Virasoro conformal blocks and the connection matrices.
In Section \ref{sec_framed_algebra},
a notion of a framed algebra will be introduced
and their general structure will be studied.
In particular, Main theorem 1 will be shown.
In Section \ref{sec_construction},
we will construct a $(r,r)$-framed algebra $S_G$
from a code $G \subset \Z_2^r$ (Main Theorem 2).
Many properties of the associated full framed VOA will be studied there, for example, the modular invariance,
correlation functions, the chiral symmetry.
Also, we give a classification of codes up to $r \leq 6$.
In Section \ref{sec_cc_def},
the current-current deformation of 
framed full vertex operator algebra will be studied (Main Theorem 3).

\noindent
\begin{center}
{\bf Acknowledgements}
\end{center}
I would like to offer my gratitude to my supervisor Professor Masahito Yamazaki
for his support, encouragement and valuable discussions.

This work was supported by the Research Institute for Mathematical Sciences,
an International Joint Usage/Research Center located in Kyoto University.
This work was also partially supported by World Premier International Research Center Initiative (WPI Initiative),
MEXT, Japan and the Program for Leading Graduate Schools, MEXT, Japan.
The author gratefully acknowledges the Kavli Institute for the Physics and Mathematics of the Universe,
the University of Tokyo, where a part of this paper was written and its hospitality during 2020.

\tableofcontents

\section{Notations and preliminary results}
\label{sec_preliminary}
In this section, we recall some notations from \cite{M3} and some basic results of vertex operator algebras.
\subsection{Notations I -- formal calculus}
Throughout this paper, we will use the following notations.
We assume that the base field is $\C$ unless otherwise stated. 
Let $z$ and $\z$ be independent formal variables.
We will use the notation $\underline{z}$ for the pair $(z,\z)$ and $|z|$ for $z\z$.

For a vector space $V$,
we denote by $V[[z^\R,\z^\R]]$
the set of formal sums 
$$\sum_{s,\s \in \R} a_{s,\s}
z^{s} \bar{z}^{\s}$$ such that $a_{s,\s} \in V$.
The space $V[[z^\R,\z^\R]]$ contains various useful subspaces:
\begin{align*}
V[[z^\R]]&=\{\sum_{s \in \R} a_{s}
z^{s}\;|\;a_s \in V \} \\
V[[z^\pm]]&=\{\sum_{n \in \Z} a_{n}
z^{n}\;|\;a_n \in V \} \\
V[[z,\z,|z|^\R]] &= \{\sum_{s,\s \in \R} a_{s,\s}
z^{s} \bar{z}^{\s}\;|\;a_{s,\s}=0 \text{ unless }s-\s \in \Z
\}\\
V[[z,\z]] &= \{\sum_{n,\n \in \R} a_{n,\n}
z^{n} \bar{z}^{\n}\;|\;a_{n,\n} \in V\}.
\end{align*}

We also denote by
$V((z,\z,|z|^\R))$ the subspace of $V[[z,\z,|z|^\R]]$
consisting of the series
$\sum_{s,\s \in \R} a_{s,\s}z^{s} \bar{z}^{\s} \in V[[z,\z,|z|^\R]]$  such that:
\begin{enumerate}
\item
For any $H \in \R$,
$\#\{(s,\s)\in \R^2\;|\; a_{s,\s}\neq 0 \text{ and }s+\s\leq H \}
$ is finite.
\item
There exists $N \in \R$ such that
$a_{s,\s}=0$ unless $s \geq N$ and $\s \geq N$
\end{enumerate}
and $V((z))$ the subspace of $V[[z^\pm]]$ consisting of the series
$\sum_{n \in \Z} a_{n}z^{n} \in V[[z^\pm]]$  such that:
\begin{enumerate}
\item
There exists $N \in \R$ such that
$a_{n}=0$ unless $n \geq N$.
\end{enumerate}
The space $V((z))$ is called the space of formal Laurent series.
Thus, $V((z,\z,|z|^\R))$ is a generalization of the Laurent series to two-variables.

Let $f(\uz) \in V((z,\z,|z|^\R))$.
By the assumption, 
there exists $r_0,r_1,r_2,\dots \in \R$
such that
\begin{enumerate}
\item
$r_0<r_1<r_2<\cdots$;
\item
$r_i \rightarrow \infty$;
\item
$f(\uz)$ could be written as
$$\sum_{i=0}^\infty \sum_{n,m=0}^\infty a_{n,m}^i z^{n}\z^{m}|z|^{r_i},
$$
where $a_{n,m}^i \in \C$.
\end{enumerate}

\begin{rem}
As seen above, $\C((z,\z,|z|^\R))$ is a Novikov ring with polynomial coefficients.
\end{rem}

We will also consider their combinations, e.g.,
$
V((y/x,\y/\x,|y/x|^\R))[x^\pm,\x^{\pm},|x|^\R],
$
which is spanned by
$$\sum_{i=1}^k \sum_{n,m=-l}^l\sum_{s,\s \in \R}
a_{n,m,r,s}^i x^{n+r_i} \x^{m+r_i} (y/x)^s(\y/\x)^\s$$
for some $k,l \in \Z_{>0}$ and $r_i \in \R$ and $a_{n,m,s,\s}^i \in V$
such that $a_{n,m,s,\s}^i=0$ unless $s-\s\in\Z$
and there exists $N$ such that $a_{n,m,s,\s}^i=0$ unless
$s\geq N$ and $\s \geq N$
and $\{(s,\s) \in \R\;|\; a_{n,m,s,\s}^i \neq 0 \text{ and }s+\s \leq H \}$ is finite for any $H\in \R$.

Let $\frac{d}{dz}$ and $\frac{d}{d\z}$ be formal differential operators
acting on $V[[z,\z,|z|^\R]]$ by
\begin{align*}
\frac{d}{dz}\sum_{s,\s \in \R} a_{s,\s}z^s \bar{z}^\s
&= \sum_{s,\s \in \R} s a_{s,\s}z^{s-1} \bar{z}^\s \\
\frac{d}{d\z}\sum_{s,\s \in \R} a_{s,\s}z^s \bar{z}^\s
&= \sum_{s,\s \in \R} \s a_{s,\s}z^{s} \bar{z}^{\s-1}.
\end{align*}
Since $\frac{d}{dz}|z|^{s}=s|z|^s z^{-1}$, the differential operators
$\frac{d}{dz}$ and $\frac{d}{d\z}$ acts on all the above vector spaces.

We note that if $f(\uz) \in V((z,\z,|z|^\R))$ satisfies $\frac{d}{d\z} f(\uz)=0$,
then $f(\uz) \in V((z))$.

For any $R \in \R_{>0}$, set $A_R=\{z\in \C \;|\; 0<|z|<R\}$,
an annulus. Let $f(\uz) \in \C((z,\z,|z|^\R)).$
Then, there exists $N \in \R$ such that
\begin{align}
|z|^N f(\uz)
= \sum_{\substack{s,\s \in \R
\\  s,\s \geq 0}} a_{s,\s}
z^{s} \bar{z}^{\s}.
\end{align}
We say the series $f(\uz)$ is absolutely convergent around $0$
if there exists $R \in \R_{>0}$ such that the sum
$\sum_{s,\s \in \R} |a_{s,\s}|R^{s+\s}$ is convergent.
In this case, $f(\uz)$ is compactly absolutely-convergent
to a continuous function defined on the annulus $A_R$. We note that the definition of the convergence is independent of the choice of $N$ (for more details see \cite[Section 1.2]{M3}).

\subsection{Notations II -- correlators}
The notion of {\it a conformal singularity} is introduced in \cite{M3}
in order to formulate two-dimensional conformal field theory.
In this section, we briefly recall the definition of a conformal singularity
and define a space of correlators in two-dimensional conformal field theory.

Let $\al_1,\dots,\al_n \in \C P^1$ and
$f$ be a $\C$-valued real analytic function on $\C P^1\setminus \{\al_1,\dots,\al_n\}$.
A chart $(\chi,\alpha)$ of $\C P^1$ at a point $\alpha \in \C P^1$ is a biholomorphism $\chi$
from an open subset $U$ of $\C P^1$ to an open subset of $\C$ such that $\alpha \in U$ and $\chi(\alpha)=0$.
We say that $f$ has a {\it conformal singularity} at $\alpha_i$
if for any chart $(\chi,\alpha_i)$ of $\C P^1$ at $\alpha_i$,
there exists a formal power series
\begin{align}
\sum_{s,\s \in \R} a_{s,\s} z^s \z^\s \in \C((z,\z,|z|^\R)) \label{eq_CS2}
\end{align}
such that it is compactly absolutely-convergent to $f\circ \chi^{-1}(z)$ on the annulus $A_R=\{z\in \C\;|\;R>|z|>0\}$ for some $R \in \R_{>0}$.
It is clear that the above condition is independent of a choice of a chart and 
 the coefficients of the series is uniquely determined by the chart.
Let $f$ have a conformal singularity at $\al_i$.
Denote by $j(\chi, f) \in \C((z,\z,|z|^\R))$ the formal power series
which is compactly absolutely-convergent to $f\circ \chi^{-1}(z)$,
and by $F_{0,1,\infty}$ the space of real analytic functions on $\C P^1 \setminus \{0,1,\infty \}$
with possible conformal singularities at $\{0,1,\infty\}$.

Examples of functions belonging to $\F$ are
$$|z|^r, |1-z|^r, z^n, (1-z)^n, (1-\z)^n \in \F,$$
where $r \in \R$ and $n \in \Z$.
For instance, the expansions of $|1-z|^r$ are
\begin{align*}
j(z, |1-z|^r) &= \sum_{n,m=0}^\infty \binom{r}{n}\binom{r}{m}z^n\z^m,\\
j(1-z, |1-z|^r) &= |z|^r,\\
j(z^{-1}, |1-z|^r) &= \sum_{n,m=0}^\infty \binom{r}{n}\binom{r}{m}z^{n-r} \z^{m-r},
\end{align*}
where $z,1-z,z^{-1}$ are charts of $0,1,\infty$, respectively.
In fact, $\F$ is a $\C[z^\pm,(1-z)^\pm,\z^\pm,(1-z)^\pm,
|1-z|^\R]$-module.

A non-trivial example of a function in $\F$ is 
\begin{align}
f_{\mathrm{Ising}}(z)=\frac{1}{2}\Bigl(|1-\sqrt{1-z}|^{1/2}+|1+\sqrt{1-z}|^{1/2}\Bigr),  \label{eq_Ising}
\end{align}
which appears in a four point function of the two-dimensional critical Ising model.
In this paper, we will show that \eqref{eq_Ising} is actually a correlator of a full vertex algebra.
The expansion of $f_{\mathrm{Ising}}(z)$ around $0$ with the chart $z$ is
\begin{align}
2+(z\z)^{1/2}/2-z/4-\z/4+(z\z)^{1/2}(z+\z)/16+ z\z/32-5z^2/64-5\z^2/64+\dots.
\end{align}
Since $f_{\mathrm{Ising}}(z)$ satisfies the equations $f_{\mathrm{Ising}}(z)=f_{\mathrm{Ising}}(1-z)=
(z\z)^{1/4}f_{\mathrm{Ising}}(1/z)$,
the expansions around $1$ and $\infty$ are
also of the form \eqref{eq_CS2}.
Thus, $f_{\mathrm{Ising}}(z) \in \F$.

Finally, we remark on the case that 
$f\in \F$ is a holomorphic function.
Recall that the ring of regular functions on the affine scheme
$\CPm$ is $\C[z^\pm,(1-z)^\pm]$.
It is easy to show that
a function in $\C[z^\pm,(1-z)^\pm]$ has conformal singularities at $\{0,1,\infty\}$.
 Thus, $\C[z^\pm,(1-z)^\pm]\subset \F$.
Conversely,
let $f \in \F$ satisfy $\frac{d}{d\z}f=\frac{1}{2}(\frac{d}{dx}-i\frac{d}{dy}) f=0$.
Then, the condition for the conformal singularity implies that
$\{0,1,\infty\}$ is a  pole of $f$, thus, $f$ is a meromorphic function on $\CP$.
Hence, $f \in \C[z^\pm,(1-z)^\pm]$.
\begin{prop}\label{holomorphic_F}
If $f \in \F$ is a holomorphic function on $\C P^1 \setminus \{0,1,\infty \}$, then $f \in \C[z^\pm,(1-z)^\pm]$.
\end{prop}

In the definition of a vertex algebra (chiral conformal field theory),
$\C[z^\pm,(1-z)^\pm]$ serves as a space of four point correlation functions
and plays an important role (see for example the introduction of \cite{M3}).
Based on this fact, we introduced a space of correlation functions
in non-chiral conformal field theory.

Set $$U(x,y)=
\C((y/x,\y/\x,|y/x|^\R))[x^\pm,\x^\pm,|x|^\R]
$$
and $$Y_2(\C^\times)=\{(z_1,z_2) \in \C^2\;|\; z_1\neq z_2, z_1 \neq 0, z_2 \neq 0 \}.$$
Let 
 $\eta(z_1,z_2):Y_2(\C^\times) \rightarrow \CPm$ be the real analytic function defined by
$\eta(z_1,z_2)=\frac{z_2}{z_1}$.
For $f \in \F$, $f \circ \eta$ is a real analytic function on $Y_2(\C^\times)$.
Denote by $\GCor_2$ the space of real analytic functions on $Y_2(\C^\times)$
spanned by
\begin{align}
z_1^{\al} \z_1^{\be} f\circ \eta(z_1,z_2), \label{eq_GCO}
\end{align}
where $f\in \F$ and $\al,\be \in \R$ satisfy
$\al-\be \in \Z$.

It is clear that $\GCor_2$ is closed under the product
and the derivations $\frac{d}{dz_1},\frac{d}{d\z_1},\frac{d}{dz_2},\frac{d}{d\z_2}$.
Since $(z_1\frac{d}{dz_1}+z_2\frac{d}{dz_2})  z_1^{\al} \z_1^{\be} f\circ \eta(z_1,z_2)
=\al  z_1^{\al} \z_1^{\be} f\circ \eta(z_1,z_2)$
and $(\z_1\frac{d}{d\z_1}+\z_2\frac{d}{d\z_2})  z_1^{\al} \z_1^{\be} f\circ \eta(z_1,z_2)
=\be  z_1^{\al} \z_1^{\be} f\circ \eta(z_1,z_2)$, by using a formal calculus,
we have:
\begin{lem}
\label{generalized_limit}
Let $\mu(z_1,z_2) \in \GCor_2$ satisfy $(z_1\frac{d}{dz_1}+z_2\frac{d}{dz_2})\mu(z_1,z_2)= \al \mu(z_1,z_2)$
and $(\z_1\frac{d}{d\z_1}+\z_2\frac{d}{d\z_2})\mu(z_1,z_2)= \be \mu(z_1,z_2)$ for some $\al,\be \in \R$.
Then, there exists unique $f \in \F$ such that
$\mu(z_1,z_2)=z_1^\al\z_1^\be f(\frac{z_1}{z_2})$.
\end{lem}

Let $\mu(z_1,z_2)=z_1^{\al} \bar{z}_1^{\be} f\circ \eta(z_1,z_2)$ in (\ref{eq_GCO}).
The expansions of $\mu(z_1,z_2)$ in $\{(z_1,z_2) \in Y_2(\C^\times) \;|\;|z_1|>|z_2|\}$
and $\{(z_1,z_2) \in Y_2(\C^\times) \;|\;|z_2|>|z_1|\}$ are respectively
given by
\begin{align*}
z_1^{\al}\bar{z}_1^{\be}&\lim_{z \to z_2/z_1} j(z,f)\\
z_1^{\al}\bar{z}_1^{\be}&\lim_{z \to z_1/z_2} j(z^{-1},f),
\end{align*}
which define maps
$$|_{|z_1|>|z_2|}:\GCor_2 \rightarrow U(z_1,z_2),\; \mu(z_1,z_2) \mapsto \mu(z_1,z_2)|_{|z_1|>|z_2|}$$
and
$$|_{|z_2|>|z_1|}:\GCor_2 \rightarrow U(z_2,z_1),\; \mu(z_1,z_2) \mapsto \mu(z_1,z_2)|_{|z_2|>|z_1|}.$$
Since
$f(\frac{z_2}{z_1})=f(\frac{z_2}{z_2+(z_1-z_2)})$,
the expansions of $\mu$ in $\{(z_1,z_2) \in Y_2(\C^\times) \;|\;|z_2|>|z_1-z_2|\}$
is given by
\begin{align*}
z_2^{\al} \bar{z}_2^{\be}
\sum_{i,j \geq 0}\binom{\al}{i}\binom{\be}{j}
(z_0/z_2)^i(\z_0/\z_2)^j &\lim_{z \to -z_0/z_2} j(1-z^{-1},f),
\end{align*}
where $z_0=z_1-z_2$.
We denote it by
$$|_{|z_2|>|z_0|}:\GCor_2 \rightarrow U(z_2,z_0),
\mu(z_1,z_2) \mapsto \mu(z_0+z_2,z_2)|_{|z_2|>|z_0|}.
$$
The following lemma connects a full vertex algebra (real analytic) and a vertex algebra (holomorphic):
\begin{lem}[\cite{M3}]
\label{hol_generalized}
Let $\mu(z_1,z_2) \in \GCor_2$ satisfies $\frac{d}{d\z_1} \mu(z_1,z_2)=0$,
 $(z_1\frac{d}{dz_1}+z_2\frac{d}{dz_2})\mu(z_1,z_2)=\al \mu(z_1,z_2)$
and  $(\z_1\frac{d}{d\z_1}+\z_2\frac{d}{d\z_2})\mu(z_1,z_2)=\be \mu(z_1,z_2)$ for some $\al,\be\in \R$.
Then, $\mu(z_1,z_2)\in \C[z_1^\pm,(z_1-z_2)^\pm,z_2^\pm,\z_2^\pm,|z_2|^\R]$.
Furthermore, if $\frac{d}{d\z_2} \mu(z_1,z_2)=0$, then $\mu(z_1,z_2) \in \C[z_1^\pm,z_2^\pm,(z_1-z_2)^\pm]$.
\end{lem}

The space of holomorphic generalized two-point correlation functions
is denoted by $\GCor_2^{\text{hol}}$,
that is,
$$
\GCor_2^{\text{hol}}=\C[z_1^\pm,z_2^\pm,(z_1-z_2)^\pm].
$$

\subsection{Vertex operator algebras, modules and intertwining operators}
In this section, we briefly recall the definitions of a vertex operator algebra, modules and intertwining operators.
For a $\Z$-graded vector space $V=\bigoplus_{n\in \Z} V_n$, 
set $V^\vee=\bigoplus_{n\in \Z} V_n^*$,
where $V_n^*$ is the dual vector space of $V_n$.

A $\Z$-graded vertex algebra is a $\Z$-graded $\C$-vector space $V=\bigoplus_{n\in \Z} V_n$ equipped with a linear map 
$$Y(-,z):V \rightarrow \End (V)[[z^\pm]],\; a\mapsto Y(a,z)=\sum_{n \in \Z}a(n)z^{-n-1}$$
and a non-zero element $\1 \in V_0$ satisfying the following conditions:
\begin{enumerate}
\item[V1)]
For any $a,b \in F$, $Y(a,z)b \in V((z))$;
\item[V2)]
For any $a \in V$, $Y(a,z)\1 \in V[[z,\z]]$ and $\lim_{z \to 0}Y(a,z)\1 = a(-1)\1=a$;
\item[V3)]
$Y(\1,z)=\mathrm{id} \in \End V$;
\item[V4)]
For any $a,b,c \in V$ and $u \in V^\vee$, there exists $\mu(z_1,z_2) \in \GCor_2^{\text{hol}}$ such that
\begin{align*}
u(Y(a,z_1)Y(b,z_2)c) &= \mu(z_1,z_2)|_{|z_1|>|z_2|}, \\
u(Y(Y(a,z_0)b,z_2)c) &= \mu(z_0+z_2,z_2)|_{|z_2|>|z_0|},\\
u(Y(b,z_2)Y(a,z_1)c)&=\mu(z_1,z_2)|_{|z_2|>|z_1|};
\end{align*}
\item[V5)]
$V_n(r)V_m \subset V_{n+m-r-1}$ for any $n,m,r \in \Z$.
\end{enumerate}
The Fourier modes $a(n)$ is called the {\it $n$-th product.}

A vertex operator algebra (of CFT type)
is a $\Z$-graded vertex algebra with a distinguished vector
$\om \in V$ such that
\begin{enumerate}
\item
There exists a scalar $c \in \C$ such that 
\begin{align}
[L(m),L(n)]=(m-n)L(m+n)+\frac{m^3-m}{12}\delta_{m+n,0}\,c
\label{eq_cc}
\end{align}
holds for any $n,m \in \Z$, where $L(n)=\om(n+1)$;
\item
$[L(-1),Y(a,z)]=\frac{d}{dz}Y(a,z)$ for any $a\in V$;
\item
$L(0)|_{V_n}=n$ for any $n \in \Z$;
\item
$V_n=0$ for any $n<0$;
\item
$V_n$ is a finite dimensional vector space for any $n \in \Z$;
\item
$V_0$ is spanned by $\1$.
\end{enumerate}
The scalar $c \in \C$ in \eqref{eq_cc} is called {\it the central charge} of the vertex operator algebra.

Let $V$ be a vertex operator algebra.
A $V$-module is defined similarly by
a linear map $Y(-,z):V \rightarrow \End M[[z^\pm]]$ (see for example \cite{LL}).
Throughout of this paper, we assume that a $V$-module $M$
satisfies the following conditions:
\begin{enumerate}
\item[M1)]
The action of $L(0)=\om(1)$ on $M$ is semisimple with real eigenvalues.
\item[M2)]
$M_r=\{m\in M_i\;|\; L(0)m=r m \}$ is a finite dimensional vector space for any $r\in \R$;
\item[M3)]
There exists $N \in \R$ such that
$M_r=0$ for any $r \leq N$.
\end{enumerate}
For non-zero $V$-module $M$, by (M3), there exists $h \in \R$ such that
$M=\bigoplus_{r \geq h}M_r$ and $M_h\neq 0$.
Such $h \in \R$ is called {\it a conformal weight} of $M$.
We note that if $M$ is a simple $V$-module, then $M=\bigoplus_{k\in \Z_{\geq 0}}M_{h+k}$.

Let $M$ be a $V$-module
and set 
$$M^\vee=\bigoplus_{r\in \R} M_r^*,
$$
a restricted dual space, where $M_r^*$ is the dual vector space of $M_r$.
Denote the canonical pairing $M^\vee \otimes M\rightarrow \C$
by $\langle - \rangle$.
A $V$-module structure on $M^\vee$ is defined by
\begin{align}
\langle Y(a,z)f, v \rangle 
= \langle f, Y\left(\exp(L(1)z)(-z^{-2})^{L(0)} a,z^{-1}\right)v \rangle \label{eq_dual}
\end{align}
for $a\in V$, $f\in M^\vee$ and $v\in M$ \cite{FHL}.
It is called a dual module.
If there exists a $V$-module isomorphism $\phi:M\cong M^\vee$, then
the canonical pairing together with $\phi$ defines a non-degenerate bilinear form on $M$
by 
\begin{align}
\langle \phi m_1,m_2\rangle = (m_1,m_2)_M\;\;\; \text{ for }m_1,m_2 \in M, \label{eq_form}
\end{align}
which satisfies the following condition:
\begin{align*}
(Y(a,z)m_1,m_2)=(m_1,Y\left(\exp(z L(1))(-z^{-2})^{L(0)}a,z^{-1}\right)m_2)
\end{align*}
for $a\in V$ and $m_1,m_2\in M$ (see \cite[Remark 5.3.3.]{FHL}).


%

Now, we recall the definition of an intertwining operator among
modules of a vertex operator algebra from \cite{FHL,HL1}.
Let $M_1,M_2,M_3$ be $V$-modules.
An intertwining operator of type $\binom{M_1}{M_2 M_3}$
is a linear map
$${I}(-,z):M_2 \rightarrow \text{Hom} (M_3,M_1)[[z^\R]],
\; v \mapsto I(v,z)=\sum_{r \in \R}v(r)z^{-r-1}$$
such that:
\begin{enumerate}
\item
For any $v_2 \in M_2$ and $v_3 \in M_3$,
there exists $N \in \R$ such that $v_2(r)v_3=0$ for any $r \geq N$;
\item
$[L(-1),I(v,z)]=\frac{d}{dz}I(v,z)$ for any $v \in M_2$;
\item
For any $v \in M_2$, $a \in V$ and $n\in \Z$,
\begin{align*}
[a(n), I(v,z)]&=\sum_{i\geq 0} \binom{n}{i} I(a(i)v,z)z^{n-i}\\
I(a(n)v,z)&=\sum_{i\geq 0}\binom{n}{i}(-1)^i 
\left( 
a(n-i)I(v,z)z^i
- (-1)^{n}I(v,z)a(i)z^{n-i}
\right).
\end{align*}
\end{enumerate}

The space of all intertwining operators of type $\binom{M_1}{M_2M_3}$ forms a vector space,
which is denoted by $I\binom{M_1}{M_2M_3}$.

Let $I(-,z) \in I\binom{M_1}{M_2M_3}$
and let $h_i \in \R$ be the conformal weight of $M_i$
and assume that $M_i=\bigoplus_{k\in \Z_{\geq 0}}(M_i)_{h_i+k}$ for any $i=1,2,3$.
Since $[L(0),I{\binom{M_1}{M_2M_3}}(v,z)]=$ for any
$r\in \R$ and $v \in (M_2)_r$,
we have:
\begin{lem}
\label{expansion_index}
For any $v_2\in M_2$ and $v_3 \in M_3$,
\begin{align*}
I(v_2,z)v_3 \in z^{h_1-h_2-h_3}(M_1)((z)).
\end{align*}
\end{lem}

We quote some lemmas from \cite{FHL}:
\begin{lem}[\cite{FHL}]
\label{vacuum_intertwining}
Let $M$ be a $V$-module.
\begin{enumerate}
\item
The vertex operator $Y_M(-,z)$ which defines the module structure on $M$ is 
an intertwining operator of type $\binom{M}{V,M}$;
\item
A vertex operator $Y\binom{M}{M,V}(-,z)$ defined by
\begin{align*}
Y\binom{M}{M,V}(b,z)a
= \exp(L(-1)z)Y_M(a,-z)b
\end{align*}
for $a \in V$ and $b \in M$,
is an intertwining operator of type $\binom{M}{M,V}$;
\item
If $V \cong V^\vee$ and $M \cong M^\vee$,
then a vertex operator $Y\binom{V}{M,M}(-,z)$ defined by
\begin{align*}
(Y\binom{V}{M,M}(m_1,z)m_2,v)_V
= (m_2, Y\binom{M}{M,V}(\exp(zL(1))(-z^{-2})^{L(0)}m_1,z^{-1})v)_M
\end{align*}
for $a \in V$ and $m_1,m_2 \in M$,
is an intertwining operator of type $\binom{V}{M,M}$,
where $(-,-)_V$ and $(-,-)_M$ are defined by \eqref{eq_form}.
\end{enumerate}
\end{lem}

\begin{lem}{\cite[Lemma 5.2.3]{FHL}}
\label{exp_L1}
Let $Y(-,z) \in I\binom{M_1}{M_2M_3}$.
For any $r \in \R$ and $m \in M_r$ with $L(1)m=0$,
$$\exp(L(1)x)Y(m,y)\exp(-L(1)x)=(1-xy)^{-2r} Y(m, y/(1-xy))
.$$
\end{lem}

Set $T^+(z)= \sum_{n \geq 0} L_{n-1} z^{-n-1}$
and $T^-(z)=\sum_{n \geq  0} L_{-n-2} z^{n}$.
The following lemma immediately follows
from the definition of an intertwining operator:
\begin{lem}\label{ward}
Let $I(-,z) \in I_{\binom{M_1}{M_2M_3}}$ 
and $a \in (M_2)_{\Delta_a}$ for $\Delta_a \in \R$.
Then, for any $n\in \Z$ and $m \geq 0$, the following equalities hold:
\begin{enumerate}
\item
$[L(n), I(a,z)] =(z^{n+1}\frac{d}{dz}+(n+1)z^n \Delta_a)I(a,w)+\sum_{k \geq 1}\binom{n+1}{k+1}
I(L(k)a,z)z^{n-k}$;
\item
$I(L(-1)a,z) = \frac{d}{dz}I(a,z)$;
\item
$I(L(-m-2)a,z) = \left(\frac{1}{m!}\frac{d}{dz}^m T^-(z)\right)I(a,z) + 
I(a,z)\left(\frac{1}{m!}\frac{d}{dz}^m T^+(z)\right)$;
\item
$[T^+(z), I(a,w)] = \sum_{k \geq 0} \frac{1}{(z-w)^{k+1}}|_{|z|>|w|}I(L(k-1)a,w)$;
\item
$[I(a,w), T^-(z)] = \sum_{k \geq 0} \frac{1}{(z-w)^{k+1}}|_{|w|>|z|}I(L(k-1)a,w).$
\end{enumerate}
\end{lem}
%

\subsection{Virasoro vertex operator algebra}

The Virasoro algebra is the Lie algebra
$\Vir =\bigoplus_{n\in \Z} \C L(n) \oplus \C c$
with commutation relations
\begin{align*}
[L(m),L(n)]&=(m-n)L(m+n)+\delta_{m+n,0}\frac{m^3-m}{12}c,\\
[L(m),c]&=0.
\end{align*}
Given complex numbers $h$ and $c$,
the Verma module $V(c,h)$ is a free module generated by a
vector $v = \ket{h}$
satisfying
$L(0) \cdot \ket{h} = h\ket{h}, c \cdot \ket{h} = c\ket{h},$ and $L(n) \cdot \ket{h} = 0$ for $n > 0$.
The Verma module $V(c, h)$ admits a unique maximal proper submodule $J(c,h)$ with the irreducible quotient
$$L(c,h) = V(c,h)/J(c,h)$$
and $L(c,0)$ is a simple vertex operator algebra \cite{FZ}.
In the case that the central charge $c$ takes values of the form
$$c_{p,q}=1-6\frac{(p-q)^2}{pq},$$
where $p,q$ is coprime integers such that $p,q \geq 2$,
$L(c_{p,q},0)$ is a regular vertex operator algebra,
i.e., it has only finitely many isomorphism classes of irreducible modules and every module is completely reducible \cite{Wa,DLM}.
The irreducible modules of $L(c_{p,q},0)$ is listed by $\{L(c_{p,q},h_{r,s})\}_{r=1,\dots, p-1, s=1,\dots,q-1}$ \cite{IK},
where $$
h_{r,s}=\frac{(qr-ps)^2-(q-p)^2}{4pq}.
$$
Note that there are symmetric properties
$$h_{r,s}=h_{p-r,q-s}=h_{r+p,s+q}.$$
By \cite{BPZ}, we have:
\begin{lem}\label{singular}
For any coprime integers $p,q \geq 2$,
\begin{align*}
L(-1)^2-\frac{q}{p} L(-2)&\ket{h_{2,1}}\in J_{c_{p,q}, h_{2,1}}\\
L(-1)^2-\frac{p}{q} L(-2)&\ket{h_{1,2}}\in J_{c_{p,q}, h_{1,2}}.
\end{align*}
\end{lem}

Hereafter, we only consider the case that $(p,q)$ is equal to $(3,4)$, the Ising model.
In this case, the central charge is $c_{3,4}=\frac{1}{2}$ and 
the conformal weights of irreducible modules are 
$h_{1,1}=0=h_{2,3}$ and $h_{2,1}=\frac{1}{2}=h_{1,3}$
and $h_{1,2}=\frac{1}{16}=h_{2,2}$.
Set $\Is=\{0,\frac{1}{2},\frac{1}{16}\}$.

Let $h\in \Is$.
By the construction, $L(0)$-grading of $L(\ft,h)$ start from $h$, that is,
$L(\ft,h)= \bigoplus_{k\geq 0} L(\ft,h)_{h+k}$.
Let $\bra{h} \in \mathrm{Hom}_{\C}(L(\ft,h), \C)$ be the unique linear map satisfying
\begin{enumerate}
\item
$\bra{h}\ket{h}=1$;
\item
$\bra{h}L(\ft,h)_{h+k}=0$ for any $k\geq 1$.
\end{enumerate}
Then, $\bra{h}$ is a vector in the dual module $L(\ft,h)^\vee$.

The $q$-character of an $L(\ft,0)$-module $L(\ft,h)$
is a $q$-series defined by
\begin{align*}
\chi_h(q)=\sum_{k\geq 0}\dim L(\ft,h)_{k+h}q^{k+h-\frac{1}{48}} \in q^{h-\frac{1}{48}}\Z[[q]].
\end{align*}
Then, we have (see for example \cite{IK}):
\begin{lem}
\label{character}
\begin{align*}
\chi_0(q)&=\frac{1}{\eta(q)}\sum_{k \in \Z} \left(q^{\frac{(24k+1)^2}{48}} - q^{\frac{(24k+7)^2}{48}} \right)
\\
\chi_\ft(q)&=\frac{1}{\eta(q)}\sum_{k \in \Z} \left(q^{\frac{(24k+2)^2}{48}} - q^{\frac{(24k+10)^2}{48}} \right)\\
\chi_\fs(q)&=\frac{1}{\eta(q)}\sum_{k \in \Z} \left(q^{\frac{(24k+5)^2}{48}} - q^{\frac{(24k+11)^2}{48}} \right),
\end{align*}
where $\eta(q)$ is the Dedekind eta function.
Furthermore, the matrix $S$ such that $\chi_i(-\frac{1}{\tau})=\sum_{j}S_{ij}\chi_j(\tau)$ is
\begin{align*}
S=\frac{1}{2}
\begin{pmatrix}
1 & 1 & \sqrt{2} \\
1 & 1 & -\sqrt{2} \\
\sqrt{2} &- \sqrt{2}& 0 
\end{pmatrix}
\end{align*}
where $q=\exp(2\pi i \tau)$ and the matrix is ordered as $0,\ft,\fs$.
\end{lem}

%
%

\subsection{Virasoro conformal block}
In this section, we recall the result of intertwining operators among $L(\ft,0)$-modules.
The fusion rule of $L(\ft,0)$-modules $\Is=\{0,\ft,\fs\}$
is a map $\star : \Is \times \Is \rightarrow P(\Is)$, where $P(\Is)$ is
a power set of $\Is$ and the map $\star$ is described in Table \ref{table_fusion}:
\begin{table}[h]
\label{table_fusion}
\caption{Fusion rule}
  \begin{tabular}{|c|ccc|} \hline
$\star$ & $0$ & $\ft$ & $\fs$ \\ \hline
$0$ & $0$ & $\ft$ & $\fs$ \\
$\ft$ & $\ft$ & $0$ & $\fs$ \\
$\fs$ & $\fs$ & $\fs$ & $\{0, \ft\}$ \\ \hline
\end{tabular}
\end{table}

Set $I\binom{h_1}{h_2h_3}=I\binom{L(\ft,h_1)}{L(\ft,h_2)L(\ft,h_3)}$, the space of intertwining operators between $L(\ft,0)$-modules.
Then, the following proposition is obtained in \cite{DMZ}:
\begin{prop}
\label{fusion_ising}
For $h_1,h_2,h_3 \in \Is$,
\begin{align*}
\dim I\binom{h_1}{h_2h_3}&=
\begin{cases}
1, & h_1 \in h_2 \star h_3,\\
0, & \text{otherwise}.
\end{cases}
\end{align*}
Furthermore,
if $h_1 \in h_2 \star h_3$ and $Y(-,z) \in  I\binom{h_1}{h_2h_3} \setminus \{0\}$,
then $\bra{h_1}Y(\ket{h_2},z)\ket{h_3}\neq 0$.
\end{prop}

For $h_1,h_2,h_3 \in \Is$ with $h_1 \in h_2\star h_3$,
let ${I'}_{h_2h_3}^{h_1}(-,z) \in I\binom{h_1}{h_2h_3}$ be an intertwining operator
satisfying
$$
\bra{h_1}{I'}_{h_2h_3}^{h_1}(\ket{h_2},z)\ket{h_3}=z^{h_1-h_2-h_3}.
$$
Such an intertwining operator uniquely exists by Proposition \ref{fusion_ising}.

The above normalization,
$$
{I'}_{h_2h_3}^{h_1}(\ket{h_2},z)\ket{h_3}=z^{h_1-h_2-h_3}\ket{h_1} +\text{higher terms} \in z^{h_1-h_2-h_3} L(\ft,h_1)[[z]]
$$
seems natural, however,
we will later see that this normalization is not a natural choice.
The correct normalization is given by setting
\begin{align}
{I}_{h_2h_3}^{h_1}(-,z)= 
\begin{cases}
\sqrt{2}^{-1}{I'}_{h_2h_3}^{h_1}(-,z), & (h_1,h_2,h_3) \text{ is a permutation of } (\frac{1}{2},\frac{1}{16},\frac{1}{16}), \\
{I'}_{h_2h_3}^{h_1}(-,z), & \text{otherwise}.
\end{cases}
\label{eq_normal}
\end{align}

%
%
%

%
%

%
%
For $h_0,h_1,h_2,h_3 \in \Is$,
set $A(h_0,h_1,h_2,h_3)=\{h \in \Is \;|\; h_0 \in h_1\star h \text{ and } h \in h_2\star h_3  \}$.
We note that since the fusion rule is commutative and associative,
$A(h_0,h_1,h_2,h_3) \neq \emptyset$ if and  only if
$h_0 \in h_1\star h_2 \star h_3$.

For $h_0,h_1,h_2,h_3 \in \Is$ with $h \in A(h_0,h_1,h_2,h_3)$,
let 
$${C}_{h_0,h_1,h_2,h_3}^{h}: L(\ft,h_0)^\vee \otimes L(\ft,h_1)\otimes L(\ft,h_2) \otimes L(\ft,h_3)
\rightarrow \C\{\{x\}\}\{\{y\}\}$$
be the linear map defined by
\begin{align}
{C}_{h_0,h_1,h_2,h_3}^{h}(a_0^*,a_1,a_2,a_3;x,y)
=
\langle a_0^*,
{I}_{h_1,h}^{h_0}(a_1,x)
{I}_{h_2,h_3}^h(a_2,y)a_3
\rangle,
\end{align}
for $a_0^* \in L(\ft,h_0)^\vee$
and $a_i \in L(\ft,h_i)$ ($i=1,2,3$).
In Section \ref{sec_conformal_block}, we will see that the formal power series is absolutely convergent to a multivalued holomorphic function on $Y_2(\C^\times)$.
This function is 
called a {\it Virasoro conformal block} in physics.

\begin{rem}
More precisely,
this function is a limit of four variables function,
$$C_{h_0,h_1,h_2,h_3}^{h}(z_1,z_2,z_3,z_4)
= \langle I_{h_1,h_1}^{h_0}(a_1,z_1)I_{h_2,h}^{h_1}(a_2,z_2)I_{h_3,h_4}^h(a_3,z_3)I_{h_4,0}^{h_4}(a_4,z_4)\1\rangle,$$
in $(z_1,z_2,z_3,z_4) \mapsto (\infty, x,y, 0)$.
Since this function can be recovered from
${C'}_{h_0,h_1,h_2,h_3}^{h}(x,y)$ (see for example \cite{M2}),
we only consider ${C}_{h_0,h_1,h_2,h_3}^{h}(x,y)$.
\end{rem}

We note that if $A(h_0,h_1,h_2,h_3) = \emptyset$ then the conformal block does not make sense.
Thus, throughout of this paper we assume that $h_0,h_1,h_2,h_3 \in \Is$ satisfy $A(h_0,h_1,h_2,h_3) \neq \emptyset$
when we consider the conformal block.
The set $A(h_0,h_1,h_2,h_3)$ is called {\it intermediate states}
(under the interaction of four states).

If we exchange the order of the interaction,
we obtain another intermediate state.
For example, $A(\ft,\ft,\fs,\fs)=\{0\}$ and $A(\ft,\fs,\ft,\fs)=\{\fs \}$.
However, the number of intermediate states is independent of the order of the interaction,
$$
\# A(h_0,h_1,h_2,h_3)= \# A(h_0,h_2,h_1,h_3).
$$
We remark that 
$\# A(h_0,h_1,h_2,h_3) \geq 2$ if and only if
$(h_0,h_1,h_2,h_3)=(\fs,\fs,\fs,\fs)$.
In this case, there are two different intermediate states, which
makes the CFT non-trivial.

\section{Framed full vertex operator algebra} \label{sec_full_VOA}
In this section, we briefly recall the definition and results of
a full vertex algebra and a full vertex operator algebra from \cite{M3}.
In Section \ref{subsec_framed_full}, we introduce a notion of an $(l,r)$-framed full vertex operator algebra,
which is a full vertex operator algebra with the Virasoro symmetry $\Vir_\ft^{\oplus l}\oplus \Vir_\ft^{\oplus r}$.
In Section \ref{subsec_lowest}, we define an algebra structure on a lowest weight space of a framed full VOA.
The axiom of this algebra will study in Section \ref{sec_framed_algebra}.

\subsection{Full vertex operator algebras}
In this section, we recall the definition of a full vertex algebra and a full vertex operator algebra.
For an $\R^2$-graded vector space $F=\bigoplus_{h,\h \in \R^2} F_{h,\h}$, 
set $F^\vee=\bigoplus_{h,\h \in \R^2} F_{h,\h}^*$,
where $F_{h,\h}^*$ is the dual vector space of $F_{h,\h}$.
A full vertex algebra is an $\R^2$-graded $\C$-vector space
$F=\bigoplus_{h,\h \in \R^2} F_{h,\h}$ equipped with a linear map 
$$Y(-,\uz):F \rightarrow \End (F)[[z^\pm,\z^\pm,|z|^\R]],\; a\mapsto Y(a,\uz)=\sum_{r,s \in \R}a(r,s)z^{-r-1}\z^{-s-1}$$
and an element $\va \in F_{0,0}$ 
satisfying the following conditions:
%
\begin{enumerate}
\item[FV1)]
For any $a,b \in F$, $Y(a,\uz)b \in F((z,\z,|z|^\R))$;
\item[FV2)]
$F_{h,\h}=0$ unless $h-\h \in \Z$;
\item[FV3)]
For any $a \in F$, $Y(a,\uz)\va \in F[[z,\z]]$ and $\lim_{z \to 0}Y(a,\uz)\va = a(-1,-1)\va=a$;
\item[FV4)]
$Y(\va,\uz)=\mathrm{id} \in \End F$;
\item[FV5)]
For any $a,b,c \in F$ and $u \in F^\vee$, there exists $\mu(z_1,z_2) \in \GCor_2$ such that
\begin{align*}
\langle u, Y(a,\uz_1)Y(b,\uz_2)c \rangle &= \mu(z_1,z_2)|_{|z_1|>|z_2|}, \\
\langle u, Y(Y(a,\uz_0)b,\uz_2)c \rangle &= \mu(z_0+z_2,z_2)|_{|z_2|>|z_0|},\\
\langle u, Y(b,\uz_2)Y(a,\uz_1)c \rangle&=\mu(z_1,z_2)|_{|z_2|>|z_1|};
\end{align*}
\item[FV6)]
$F_{h,\h}(r,s)F_{h',\h'} \subset F_{h+h'-r-1,\h+\h'-s-1}$ for any $r,s,h,h',\h,\h' \in \R$.
\end{enumerate}

Let $(F_1,Y_1,\1_1)$ and $(F_2,Y_2,\1_2)$ be full vertex algebras.
A full vertex algebra homomorphism from $F_1$ to $F_2$ is a grading preserving linear map 
$f:F_1\rightarrow F_2$ such that
\begin{enumerate}
\item
$f(\1_1)=\1_2$
\item
$f(Y_1(a,\uz)b)=Y_2(f(a),\uz)f(b)$ for any $a,b \in F_1$.
\end{enumerate}
The notions of a subalgebra and an ideal are defined in the usual way.

The difference between vertex algebras
and full vertex algebras is summarized
in the following table:
\begin{table}[h]
\begin{tabular}{|l||c|c|}\hline
 & vertex algebra & full vertex algebra \\ \hline \hline
vector space &  $V=\bigoplus_{n \in \Z} V_n$ & $F=\bigoplus_{h,\h \in \R} F_{h,\h}$ \\
vertex operator &$\End V[[z^\pm]]$ & $\End F[[z,\z,|z|^\R]]$ \\
singularity & $\C((z))$ & $\C((z,\z,|z|^\R))$ \\
correlation function &$\C[z^\pm,(1-z)^\pm]$ & $\F$\\ \hline
\end{tabular}
\end{table}

Let $(V,Y,\1)$ be a $\Z$-graded vertex algebra.
Then, by a standard result of the theory of a vertex algebra (see for example \cite{FLM}),
$u(Y(a_1,z_1)Y(a_2,z_2)a_3)$ is an expansion of a rational polynomial
in $\GCor_2^{\text{hol}}=\C[z_1^\pm,z_2^\pm,(z_1-z_2)^{-1}] \subset \GCor_2$ in $|z_1|>|z_2|$
for any $u\in V^\vee$ and $a_1,a_2,a_3\in V$.
Thus, we have:
\begin{prop}\label{graded_vertex}
A $\Z$-graded vertex algebra is a full vertex algebra.
\end{prop}

Let $(F,Y,\1)$ be a full vertex algebra.
Set $\bar{F}=F$ and
$\bar{F}_{h,\h}=F_{\h,h}$ for $h,\h\in \R$.
Define $\bar{Y}(-,\uz):\bar{F} \rightarrow \End (\bar{F})[[z,\z,|z|^\R]]$
by $\bar{Y}(a,\uz)=\sum_{r,s \in \R}a(r,s)z^{-s-1}\z^{-r-1}$.
Then, we have:
\begin{prop}\label{conjugate}
$(\bar{F},\bar{Y},\1)$ is a full vertex algebra.
\end{prop}
We call it a conjugate full vertex algebra of $(F,Y,\1)$.

Let $F$ be a full vertex algebra.
The set $\{(h,\h) \in \R^2\; |\;F_{h,\h}\neq 0\}$
is called a spectrum.
The spectrum of $F$ is said to be bounded below if there exists $N\in \R$ such that
$F_{h,\h}=0$ for any $h\leq N$ or $\h \leq N$
and discrete if
for any $H\in \R$,
$\sum_{h+\h\leq H} \dim F_{h,\h}$ is finite.
Then, we have:
\begin{prop}\label{full_tensor}
Let $(F_1,Y_1,\1_1)$ and $(F_2,Y_2,\1_2)$ be full vertex algebras.
If the spectrum of $F_1$ is discrete and $F_2$ is bounded below,
then $(F_1\otimes F_2,Y_1\otimes Y_2,\1_1\otimes \1_2)$ is a full vertex algebra.
\end{prop}

Let $F$ be a full vertex algebra
and $D$ and $\D$ denote the endomorphism of $F$
defined by $Da=a(-2,-1)\bm{1}$ and $\D a=a(-1,-2)$ for $a\in F$,
i.e., $$Y(a,z)\1=a+Daz+\D a\z+\dots.$$
Define $Y(a,-\uz)$ by
$Y(a,-\uz)=\sum_{r,s\in \R}(-1)^{r-s} a(r,s)z^r \z^s$,
where we used $a(r,s)=0$ for $r-s \notin \Z$,
which follows from (FV2) and (FV6).
Then, we have:
\begin{prop}
\label{translation}
For $a \in F$, the following properties hold:
\begin{enumerate}
\item
$Y(Da,\uz)=d/dz Y(a,\uz)$ and $Y(\D a,\uz)=d/d\z Y(a,\uz)$;
\item
$D\1=\D\1=0$;
\item
$[D,\D]=0$;
\item
$Y(a,\uz)b=\exp(zD+\z\D)Y(b,-\uz)a$;
\item
$Y(\D a,\uz)=[\D,Y(a,\uz)]$ and $Y(Da,\uz)=[D,Y(a,\uz)]$.
\end{enumerate}
\end{prop}

We introduced a notion of a full vertex operator algebra similarly to the definition of
a vertex operator algebra.
{\it An energy-momentum tensor}
of a full vertex algebra is a pair of vectors
$\omega \in F_{2,0}$ and $\bar{\omega}\in F_{0,2}$ such that
\begin{enumerate}
\item
$\D \om=0$ and $D \omb=0$;
\item
There exist scalars $c, \bar{c} \in \C$ such that
\begin{align}
[L(m),L(n)]=(m-n)L(m+n)+\frac{m^3-m}{12}\delta_{m+n,0}\,c
\label{eq_cc}\\
[\Ld(m),\Ld(n)]=(m-n)\Ld(m+n)+\frac{m^3-m}{12}\delta_{m+n,0}\,\bar{c} \nonumber
\end{align}
holds for any $n,m \in \Z$, where $L(n)=\omega(n+1,-1)$ and $\Ld(n)=\bar{\omega}(-1,n+1)$;
\item
$L(-1)=D$ and $\Ld(-1)=\D$;
\item
$L(0)|_{F_{h,\h}}=h$ and
$\Ld(0)|_{F_{h,\h}}=\h$ for any $h,\h \in \R$;
\item
The spectrum of $F$ is discrete and bounded below;
\item
$F_{0,0}=\C \1$.
\end{enumerate}
A full vertex operator algebra is a pair of a full vertex algebra and its energy momentum tensor.

Let $(F_1,Y_1,\1_1,\om_1,\omb_1)$ and $(F_2,Y_2,\1_2,\om_2,\omb_2)$ be full vertex operator algebras.
{\it A conformal embedding}
is a full vertex algebra homomorphism $i:F_1 \rightarrow F_2$ such that:
\begin{enumerate}
\item
$i(\om_1)=\om_2$ and $i(\omb_1)=\omb_2$;
\item
$i$ is injective.
\end{enumerate}

\subsection{Full vertex algebra and vertex algebra}
In this section, we briefly recall the relation between a vertex algebra and a full vertex algebra.
Let $F$ be a full vertex algebra.
A vector $a \in F$ is said to be a holomorphic vector (resp. an anti-holomorphic vector)
if $\D a=0$ (resp. $D a=0$).
Let $a \in \ker \D$.
Then, since $0=Y(\D a,\uz)=d/d\z Y(a,\uz)$,
we have $a(r,s)=0$ unless $s=-1$.
Hence, $Y(a,\uz)=\sum_{n \in \Z} a(n,-1) z^{-n-1}$.
Then, by Lemma \ref{hol_generalized},
we have (see \cite{M3} for more detail):
\begin{lem}
\label{hol_commutator}
Let $a,b\in F$.
If $\D a=0$,
then 
$$
Y(a,\uz)=\sum_{n \in \Z}a(n,-1)z^{-n-1}
$$
and for any $n\in \Z$,
\begin{align*}
[a(n,-1),Y(b,\uz)]&= \sum_{i \geq 0} \binom{n}{i} Y(a(i,-1)b,\uz)z^{n-i},\\
Y(a(n,-1)b,\uz)&= 
\sum_{i \geq 0} \binom{n}{i}(-1)^i a(n-i,-1)z^{i}Y(b,\uz)
-Y(b,\uz)\sum_{i \geq 0} \binom{n}{i}(-1)^{i+n} a(i,-1)z^{n-i}.
\end{align*}
\end{lem}

By Proposition \ref{translation}, $\D Y(a,\uz)b =Y(\D a,\uz)b+Y(a,\uz)\D b=0$.
Thus, the restriction of $Y$ on $\ker \D$ define a linear map $Y(-,z): \ker \D \rightarrow \End\ker \D[[z^\pm]]$.
By the above Lemma and Lemma \ref{hol_generalized}, we have:
\begin{prop}\label{vertex_algebra}
$\ker \D$ is a vertex algebra and $F$ is a $\ker \D$-module.
\end{prop}

\begin{lem}
\label{hol_commute}
For a holomorphic vector $a \in F$ and an anti-holomorphic vector $b\in F$, 
$[Y(a,z),Y(b,\bar{w})]=0$, that is,
$[a(n,-1),b(-1,m)]=0$ and $a(k,-1)b=0$ for any $n,m \in \Z$ and $k \in \Z_{\geq 0}$.
\end{lem}

By Proposition \ref{vertex_algebra}, $\ker D$ and $\ker \D$ is a subalgebra of $F$
and $\ker \D \otimes \ker D$ is a full vertex algebra.
Define a linear map $t:\ker \D \otimes \ker D \rightarrow F$ by
$t(a\otimes b)= a(-1,-1)b$ for $a\in \ker \D$ and $b\in \ker D$.
Then, we have:
\begin{prop}\label{ker_hom}
The map $t:\ker \D\otimes \ker D \rightarrow F$ is a full vertex algebra homomorphism.
\end{prop}
We note that if $\ker \D$ and $\ker D$ are simple vertex algebras,
then the map $t$ is injective.

Let $Y_c(-,z): F \rightarrow \End F[[z^\R]]$
be the vertex operator defined by 
\begin{align}
Y_c(a,z)=\lim_{\z \to z}Y(a,\uz)=
\sum_{r,s\in \Z}a(r,s)z^{-r-s-2} \label{eq_vertex}
\end{align}
for $a \in F$.
The identification $z=\z$ is well-defined since
it gives a well-defined linear map $z=\z: \C((z,\z,|z|)) \rightarrow \C[[z^\R]]$.

Let $V_F$ be the image of the map $t$ in Proposition \ref{ker_hom}, $V_F = t(\ker \D \otimes \ker D) \subset F$.
For any $a \in \ker \D$ and $b\in \ker D$, by Lemma \ref{hol_commute} and Lemma \ref{hol_commutator},
$$Y_c(a(-1,-1) b,z)=\sum_{n,m\in \Z}a(n,-1)b(-1,m)z^{-n-m-2}=Y_c(a,z)Y_c(b,z).
$$
In particular, $Y_c(a(-1,-1) b,z) \in \End F[[z^\pm]]$.
If $F$ is a full VOA, then by Lemma \ref{hol_commutator} $(V_F,Y_c)$ is a vertex operator algebra
and $F$ is a $V_F$-module.
Furthermore, by Lemma \ref{hol_commutator} and the definition of an intertwining operator, we have:
\begin{lem}
\label{chiral_intertwining}
Let $F$ be a full vertex operator algebra.
For the vertex operator $Y_c(-,z)$ defined by \eqref{eq_vertex},
$V_F$ is a vertex operator algebra, $F$ is a $V_F$-module
and $Y_c(-,z)$ is an intertwining operator of type $\binom{F}{FF}$.
\end{lem}

\begin{rem}
This lemma is also found in \cite{HK}.
\end{rem}

\subsection{Locality axiom}

In the definition of a full vertex algebra,
the terms
\begin{align}
\langle u, Y(a,\uz_1)Y(b,\uz_2)c \rangle &= \mu(z_1,z_2)|_{|z_1|>|z_2|},\nonumber \\
\langle u, Y(b,\uz_2)Y(a,\uz_1)c \rangle&=\mu(z_1,z_2)|_{|z_2|>|z_1|} \label{eq_locality}
\end{align}
are symmetric while
$$
\langle u, Y(Y(a,\uz_0)b,\uz_2)c \rangle = \mu(z_0+z_2,z_2)|_{|z_2|>|z_0|}
$$
is not.
So it is convenient to give a definition of a full vertex algebra
only by using \eqref{eq_locality}.
Such a result is obtained in \cite{M3} similarly to the case of vertex algebras.
We briefly recall this result in this section.

A full prevertex algebra is an $\R^2$-graded $\C$-vector space
$F=\bigoplus_{h,\h \in \R^2} F_{h,\h}$ equipped with a linear map 
$$Y(-,\uz):F \rightarrow \End F[[z^\pm,\z^\pm,|z|^\R]],\; a\mapsto Y(a,\uz)=\sum_{r,s \in \R}a(r,s)z^{-r-1}\z^{-s-1}$$
and a non-zero element $\1 \in F_{0,0}$ 
such that
\begin{enumerate}
\item[PV1)]
For any $a,b \in F$, $Y(a,\uz)b \in F((z,\z,|z|^\R))$;
\item[PV2)]
$F_{h,\h}=0$ unless $h-\h \in \Z$;
\item[PV3)]
For any $a \in F$, $Y(a,\uz)\1 \in F[[z,\z]]$ and $\lim_{z \to 0}Y(a,z)\1 = a(-1,-1)\1=a$;
\item[PV4)]
$Y(\1,\uz)=\mathrm{id} \in \End F$;
\item[PV5)]
$F_{h,\h}(r,s)F_{h',\h'} \subset F_{h+h'-r-1,\h+\h'-s-1}$ for any $r,s,h,\h,h',\h' \in \R$.
\end{enumerate}

A full prevertex algebra $(F,Y,\1)$ is said to be translation covariant
if there exist linear maps $D,\D \in \End \;F$ such that
\begin{enumerate}
\item[T1)]
$D\1=\D\1=0$;
\item[T2)]
For any $a\in F$,
$[D,Y(a,\uz)]=\frac{d}{dz}Y(a,\uz)$ and $[\D,Y(a,\uz)]=\frac{d}{d\z}Y(a,\uz)$;
\end{enumerate}
Then, we have (see \cite[Proposition 3.8]{M3}):
\begin{prop} \label{locality}
Assume that a translation covariant full prevertex algebra $(F,Y,\1,D,\D)$ satisfies the following conditions:
\begin{enumerate}
\item[L1)]
The spectrum of $F$ is bounded below;
\item[L2)]
For any $a_1,a_2,a_3 \in F$ and $u \in F^\vee$, there exists $\mu(z_1,z_2) \in \GCor_2$ such that
\begin{align*}
u(Y(a_2,\uz_2)Y(a_1,\uz_1)a_3) &= \mu(z_1,z_2)|_{|z_1|>|z_2|}\\
u(Y(a_1,\uz_1)Y(a_2,\uz_2)a_3) &=\mu(z_1,z_2)|_{|z_2|>|z_1|}
\end{align*}
\end{enumerate}
Then, $(F,Y,\1)$ is a full vertex algebra.
\end{prop}

\subsection{Definition of framed full vertex operator algebra}\label{subsec_framed_full}
Let $l,r \in \Z_{\geq 0}$.
By Lemma \ref{full_tensor},
the tensor product of the Virasoro vertex operator algebras
$L(\ft,0)^l\otimes \overline{L(\ft,0)}^r$
is a full vertex operator algebra,
where $\overline{L(\ft,0)}^r$ is the conjugate full vertex algebra in Proposition \ref{conjugate}.
We denote it by $L_{l,r}(0)$.
{\it An $(l,r)$-framed full vertex operator algebra}
is a full vertex operator algebra $F$ together with
a conformal embedding $i:L_{l,r}(0) \rightarrow F$.

Let $S_l, S_r$ be the symmetric group of degree $l,r$.
Then, by the permutation, $S_l\times S_r$ is naturally a subgroup of
the full vertex operator algebra automorphism group $\mathrm{Aut}\,L_{l,r}(0)$.
A morphism from an $(l,r)$-framed full VOA $i_1:L_{l,r}(0) \rightarrow F_1$
to an $(l,r)$-framed full VOA  $i_2:L_{l,r}(0) \rightarrow F_2$
is a pair of a full VOA homomorphism $\phi:F_1\rightarrow F_2$
and $(g,\bar{g}) \in S_l\times S_r$
such that
$i_2 \circ (g,\bar{g})= \phi \circ i_1: L_{l,r}(0) \rightarrow F_2$.
We denote the category of framed full vertex operator algebras
by $\FF$.
By definition,
a framed full vertex operator algebra is a module of copies of Virasoro algebras $\Vir_\ft^{\oplus l} \oplus \Vir_\ft^{\oplus r}$.
We denote the action of $i$-th component of the left Virasoro algebras $\Vir_\ft^{\oplus l}$
by $\{L_i(n)\}_{n\in\Z}$
and the action of $j$-th component of the right Virasoro algebras $\Vir_\ft^{\oplus r}$
by $\{\Ld_j(n)\}_{n\in\Z}$ for $l \geq i \geq 1$ and $r \geq j \geq 1$.

\subsection{Multi-index Virasoro conformal block}
\label{sec_multi_def}
In this section, we prepare some notations for $L_{l,r}(0)$-modules.
Set $\Isr = \Is^{l} \times \Is^r$.
For $\la=(h_1,h_2,\dots,h_l,\h_1,\dots,\h_r) \in \Isr$,
define an $L_{r,s}(0)$-module by
$L_{l,r}(\la)= L(\frac{1}{2},h_1) \otimes \dots \otimes 
L(\frac{1}{2},h_l) \otimes \overline{L(\frac{1}{2},\h_1)}\otimes \dots \otimes \overline{L(\frac{1}{2},\h_r)}
$
and set
$\ket{\la}=\ket{h_1}\otimes \cdots \ket{h_l} \otimes \ket{\h_1}\otimes \cdot \otimes \ket{\h_r} \in L_{l,r}(\la)$,
which is a lowest weight vector.
By the similar argument in Lemma \ref{chiral_intertwining},
it is easy to show that
this gives a bijection between irreducible $L_{l,r}(0)$-modules
and $\Isr$.

By extending the fusion rule of $\Is$,
define a fusion rule $\star:\Isr \times \Isr \rightarrow
P(\Isr)$ by
\begin{align*}
\la^1 \star \la^2 =(h_1^1\star h_1^2) \times \cdots \times (h_l^1 \star h_l^2) \times (\h_1^1\star \h_1^2) \times \cdots \times (\h_r^1 \star \h_r^2)
\end{align*}
for $\la^i=(h_1^i,\dots,h_l^i, \bar{h}_1^i\dots,\bar{h}_r^i) \in \Isr$
and $i=1,2$.

For $\la^1,\la^2 \in \Isr$ and $\la \in \la^1\star \la^2$,
let $I_{\la^1,\la^2}^\la(-,\uz): L_{l,r}(\la^1) \rightarrow  \mathrm{Hom}\left(L_{l,r}(\la^2),L_{l,r}(\la)\right)[[z^\R,\z^\R]]$
be a linear map defined by the tensor product of intertwining operators:
\begin{align*}
I_{\la^1,\la^2}^\la(-,\uz)=\bigotimes_{i=1}^l I_{h_i^1,h_i^2}^{h_i}(-,z) \otimes \bigotimes_{j=1}^r \bar{I}_{h_i^1,h_i^2}^{h_i}(-,\z).
\end{align*}
Let $l,r,s:\Isr \rightarrow \mathbb{Q}$ be maps defined by
$l(\la)=\sum_{i=1}^l h_i$ and $r(\la)=\sum_{j=1}^r \h_j$
for $\la=(h_1,h_2,\dots,h_l,\h_1,\dots,\h_r) \in \Isr$
and $s = l - r$.
\begin{rem}
The map $l$ (resp. $r$) measures the conformal weights of holomorphic part (resp. anti-holomorphic part)
and $s$ corresponds to the spin of the state,
which must be an integer for any bosonic state.
\end{rem}
By Lemma \ref{expansion_index}, we have:
\begin{lem}
Let $\la^1,\la^2 \in \Isr$, $\la\in \la^1\star \la^2$
and $a \in L_{l,r}(\la^1)$, $b \in L_{l,r}(\la^2)$.
Then,
$$I_{\la^1,\la^2}^\la(a,\uz)b \in 
z^{l(\la)-l(\la^1)-l(\la^2)}\z^{r(\la)-r(\la^1)-r(\la^2)} L_{l,r}(\la)((z,\z)).
$$
In particular, if $s(\la),s(\la^1), s(\la^2)\in \Z$,
then $I_{\la^1,\la^2}^\la(a,\uz)b \in L_{l,r}(\la)((z,\z,|z|^\R))$.
\end{lem}

For $\la^0,\la^1,\la^2,\la^3 \in \Isr$,
let $A(\la^0,\la^1,\la^2,\la^3)$ be the set of all
$\la \in \Isr$ such that $\la \in \la^2 \star \la^3$ 
and $\la^0 \in \la^1\star \la$.
We note that since the fusion rule is commutative and associative again,
$A(\la^0,\la^1,\la^2,\la^3) \neq \emptyset$ if and  only if
$\la^0 \in \la^1\star \la^2 \star \la^3$.

For $\la^0,\la^1,\la^2,\la^3 \in \Isr$ with $\la \in A(\la^0,\la^1,\la^2,\la^3)$,
let 
$${C}_{\la^0,\la^1,\la^2,\la^3}^{\la}: L_{l,r}(\la^0)^\vee \otimes L_{l,r}(\la^1)\otimes L_{l,r}(\la^2) \otimes L_{l,r}(\la^3)
\rightarrow \C\{\{x\}\}\{\{y\}\}$$
be the linear map defined by
\begin{align}
{C}_{\la^0,\la^1,\la^2,\la^3}^{\la}(a_0^*,a_1,a_2,a_3;x,y)
=
\langle a_0^*,
{I}_{\la^1,\la}^{\la^0}(a_1,x)
{I}_{\la^2,\la^3}^h(a_2,y)a_3
\rangle,
\end{align}
for $a_0^* \in L_{l,r}(\la^0)^\vee$
and $a_i \in L_{l,r}(\la^i)$ ($i=1,2,3$).
In this paper we call this function a {\it multi-index Virasoro conformal block}.
It is clear from the definition that a multi-index Virasoro conformal block is
decomposed into the product of Virasoro conformal blocks.
Properties of multi-index Virasor conformal blocks will be studied in Section \ref{sec_conformal_block}.

\subsection{Lowest weight space and conformal block}\label{subsec_lowest}
Let $F$ be a framed full vertex operator algebra.
Then, $F$ is a module of the vertex operator algebra $L(\ft,0)^{l+r}$,
where the module structure is given by $Y_c(-,z)$ (see Lemma \ref{chiral_intertwining}).
Since any module of $L(\ft,0)^{l+r}$ is completely reducible \cite{Wa,DLM},
%
%
$F$ is decomposed into a direct sum of irreducible $L_{l,r}(0)$-modules,
\begin{align*}
F = \bigoplus_{\lambda \in \Isr} L_{l,r}(\lambda) \otimes (S_F)_\lambda,
\end{align*}
where $(S_F)_\lambda$, for $\la=(h_1,h_2,\dots,h_l,\h_1,\dots,\h_r) \in \Isr$, consists of all vectors $v \in F$ such that:
for any $l \geq i \geq 1, r \geq j \geq 1$
\begin{enumerate}
\item[LW1)]
$L_i(0)v=h_i v, \Ld_j(0)v=\h_j v$;
\item[LW2)]
$L_i(n)v=\Ld_j(n)v=0$ for any $n \geq 1$.
\end{enumerate}
By the definition of a full VOA, $(S_F)_\lambda$ is a finite dimensional vector space.
Set $S_{F} =\bigoplus_{\la \in \Isr} (S_F)_{\la} \subset F$,
which is the lowest weight space of $\Vir_\ft^{\oplus l}\oplus \Vir_\ft^{\oplus r}$.

By Lemma \ref{chiral_intertwining} again,
the vertex operator $Y_c(-,z):F\rightarrow \End F[[z^\R]]$ is
an $L(\ft,0)^{l+r}$-module intertwining operator of type $\binom{F}{FF}$.
We note that
\begin{align}
I\binom{F}{FF}&=
\bigoplus_{\la^1,\la^2,\la^3 \in \Isr} I\binom{L_{l,r}(\lambda^1)\otimes S_{\la^1}}{L_{l,r}(\lambda^2)\otimes S_{\la^2},L_{l,r}(\lambda^3)\otimes S_{\la^3}} \label{eq_intertwining} \\
&=
\bigoplus_{\la^1,\la^2,\la^3 \in \Isr}
\mathrm{Hom}_\C(S_{\la^2},\mathrm{Hom}_\C (S_{\la^3},S_{\la^1}))\otimes I\binom{L_{l,r}(\lambda^1)}{L_{l,r}(\lambda^2),L_{l,r}(\lambda^3)}\nonumber \\
&=\bigoplus_{\la^1,\la^2,\la^3 \in \Isr}
\mathrm{Hom}_\C(S_{\la^2} \otimes S_{\la^3},S_{\la^1})\otimes
\bigotimes_{i=1}^l I\binom{h_i^1}{h_i^2,h_i^3} \otimes
 \bigotimes_{j=1}^r  I\binom{\h_j^1}{\h_j^2,\h_j^3}, \nonumber
\end{align}
where $\la^i=(h_1^i,\dots,h_l^i,\h_1^i,\dots,\h_r^i)$ for $i=1,2,3$ and the last equality follows from \cite{DMZ}.

By \eqref{eq_intertwining}, the vertex operator $Y(-,\uz)$ of the framed full vertex operator algebra $F$
is decomposed into the tensor product of holomorphic and anti-holomorphic
 Virasoro intertwining operators $I_{h_2,h_3}^{h_1}(-,z), I_{\h_2,\h_3}^{\h_1}(-,\z)$.
Furthermore, it is uniquely determined by
the linear map $\bigoplus_{\la^1,\la^2,\la^3 \in \Isr} \mathrm{Hom}(S_{\la^2}\otimes S_{\la^3},S_{\la^1})$,
which gives an algebra structure on $S_F$. We denote this product $S_F \times S_F \rightarrow S_F$ by $\cdot$.

For $\la \in \Isr$,
let $\pi_\la:S_F \rightarrow S_F$ be the projection of the graded vector space $S_F$ onto $(S_F)_\la$
and 
$\cdot_\la : S\otimes S \rightarrow S_\la$ given by the composition of
the product $\cdot :S\otimes S \rightarrow S$ and the projection $\pi_\la$.
Then, the product $\cdot$ is a unique linear map $S_F\otimes S_F \rightarrow S_F$
such that for any $\la^1,\la^2 \in \Isr$ and $a_1 \in (S_F)_{\la^1}$, $a_2 \in (S_F)_{\la^2}$,
$$
Y(a_1,\uz)a_2 =
\sum_{\la \in \la^1\star \la^2 } I_{\la^1,\la^2}^{\la}(\ket{\la^1},z)\ket{\la^2} \otimes (a_1 \cdot_\la a_2) \in 
F=\bigoplus_{\la \in \Isr} L_{l,r}(\la) \otimes (S_F)_\la.
$$
%

The following lemma easily follows from \eqref{eq_intertwining} and the definition of a full vertex operator algebra:
\begin{lem}
\label{preframed}
For any framed full VOA $F$,
$S_F$ satisfies the following conditions:
\item[FA1)]
For any $\la \in \Isr$, $(S_F)_\la =0$ unless $s(\la) \in \Z$.
\item[FA2)]
$(S_F)_0=\C \1$ and for any $a\in S$,
$a \cdot \1=\1 \cdot a=a$;
\item[FA3)]
For any $\la^1,\la^2 \in \Isr $, $a_1 \in S_{\la^1}$ and $a_2 \in S_{\la^2}$,
$a_1 \cdot a_2 \in \bigoplus_{\la \in \la^1 \star \la^2}S_{\la}.$
\end{lem}

As seen above, we can obtain an algebra structure on $S_F$ from the vertex operator $Y(-,\uz)$.
In the next section, we consider this construction in the reverse direction.

\subsection{Induced vertex operator}
\label{sec_ind_low}
Let $S$ be a finite-dimensional $\Isr$-graded vector space,
$$
S=\bigoplus_{\la \in \Isr} S_\la
$$
with a linear map $\cdot: S \otimes S \rightarrow S$
and a distinguished non-zero element $1 \in S_0$ satisfies the conditions (FA1),(FA2) and (FA3) in Lemma \ref{preframed}.
For $\la \in \Isr$,
let $\pi_\la:S\rightarrow S$ be the projection of the graded vector space $S$ onto $S_\la$
and denote the composition of the product and the projection
by $\cdot_\la : S\otimes S \rightarrow S_\la$.

Set $F_S= \bigoplus_{\la \in \Isr} L_{r,s}(\la) \otimes S_\la$,
an $L_{l,r}(0)$-module.
Let $i:S \rightarrow F_S$ be the linear map defined by 
$i(a)=\ket{\la} \otimes a \in L_{r,s}(\la)\otimes S_\la$ for any $h\in \Isr$ and $a \in S_\la$.

Define $Y(-,z,\z): F_S \rightarrow \End F_S$
by 
$$Y(u_1 \otimes a_1, z,\z) u_2\otimes a_2
= \sum_{\la \in \la^1\star\la^2} I_{\la^1\la^2}^{\la}(u_1,z)u_2 \otimes (a_1 \cdot_{\la} a_2).
$$
for $\la^1,\la^2\in \Isr$, $u_1 \in L_{l,r}(\la^1)$, $u_2 \in L_{l,r}(\la^2)$
and $a_1 \in L_{l,r}(\la^1)$ and $a_2 \in L_{l,r}(\la^2)$.

Let $D,\D \in \End\,F_S$
defined by 
\begin{align*}
D=\sum_{i=1}^l L_i(-1),\;\;
\D=\sum_{j=1}^l \Ld_j(-1),
\end{align*}
where $L_i(-1), \Ld_j(-1)$ are given by the action of $\Vir^{\oplus l}\oplus\Vir^{\oplus r}$.
Then, we have:
\begin{prop}
\label{prevertex}
$(F_S,Y(-,\uz),\1,D,\D)$ is a translation covariant full prevertex algebra
and its spectrum is bounded below.
\end{prop}
\begin{proof}
(PV1), (PV5), (T1) and (T2) follow from the definition
of the intertwining operators.
(PV2) follows from (FA1). (PV3) and (PV4) follows
from (FA2) and Lemma \ref{vacuum_intertwining}.
The spectrum is clearly bounded below.
\end{proof}

\begin{rem}
In general, there is a linear isomorphism between the following spaces:
\begin{enumerate}
\item
The space of intertwining operators $I\binom{F_S}{F_S F_S}$;
\item
The space of linear maps $\cdot \in \mathrm{Hom}_\C(S\otimes S,S)$
such that $a \cdot b \in \bigoplus_{\la \in \la^1 \star \la^2}S_{\la}$ for any $\la^1,\la^2 \in \Isr $, $a \in S_{\la^1}$ and $b\in S_{\la^2}$.
\end{enumerate}
\end{rem}

One of the main result of this paper is to obtain the precise
condition when $F_S$ becomes a framed full vertex operator algebra.
By Proposition \ref{locality},
$F_S$ is a full vertex algebra if and only if
$F_S$ satisfies  the locality axiom (L2).
Let $\la^0,\la^1,\la^2,\la^3 \in \Isr$
and $u_i \otimes a_i \in L_{l,r}(\la^i) \otimes S_{\la^i}$
and $u_0^*\otimes a_0^* \in L_{l,r}(\la^0)^\vee \otimes (S_{\la^0})^*$.
Then,
\begin{align}
\langle &u_0^* \otimes a_0^*,
Y(u_1\otimes a_1,z_1,\z_1)Y(u_2\otimes a_2,z_2,\z_2)u_3\otimes a_3 \rangle \nonumber \\
&=
\sum_{\la \in \la^2 \star \la^3}
\langle u_0^*, I_{\la^1 \la}^{\la^0}(u_1,\uz_1)I_{\la^2 \la^3}^{\la}(u_2,\uz_2)u_3 \rangle
\langle a_0^*, a_1 \cdot_{\la_0} (a_2\cdot_\la a_3) \rangle \nonumber \\
&=\sum_{\la \in \la^2 \star \la^3}
C_{\la^0, \la^1,\la^2,\la^3}^{\la}(u_0,u_1,u_2,u_3; z_1,z_2)
\langle a_0^\vee, a_1 \cdot_{\la_0} (a_2\cdot_\la a_3)\rangle. \label{eq_correlator}
\end{align}
Thus, to clarify when $F_S$ is a full vertex algebra,
it is important to study multi-index Virasoro conformal blocks.
The next section is devoted to studying them.

%
%
%

\section{Conformal block of Virasoro algebras} \label{sec_conformal_block}

In this section, we consider the Virasoro conformal blocks.
We first study some specialization of single Virasoro conformal blocks.
For $(h_0,h_1,h_2,h_3) \in \Is$ and $h \in A(h_0,h_1,h_2,h_3)$,
set 
\begin{align*}
{C}_{h_0,h_1,h_2,h_3}^{h}(x,y)
&= {C}_{h_0,h_1,h_2,h_3}^{h}(\bra{h_0},\ket{h_1},\ket{h_2},\ket{h_3}; x,y)\\
&= \bra{h_0} I_{h_1, h}^{h_0}(\ket{h_1}, x)
I_{h_2,h_3}^h (\ket{h_2}, y) \ket{h_3}.
\end{align*}
We will later see that all informations of the Virasoro conformal blocks
are obtained from these specialized ones.

It is useful to introduce non-normalized conformal blocks,
\begin{align*}
{C'}_{h_0,h_1,h_2,h_3}^{h}(x,y)
&= \bra{h_0} {I'}_{h_1, h}^{h_0}(\ket{h_1}, x)
{I'}_{h_2,h_3}^h (\ket{h_2}, y) \ket{h_3}.
\end{align*}

In section \ref{sec_block_vac}, we study the specialized Virasoro conformal blocks
in the case that one of $h_i$ is equal to $0$.
General cases are studied in Section \ref{sec_block_gen}.
Analytic continuations of the Virasoro conformal blocks are studied in Section \ref{sec_conn}
and their multi-index generalizations are given in Section \ref{sec_multi_conn},
which is the most important part of this section.

\subsection{Conformal block involving vacuum state}\label{sec_block_vac}
In this section, we study ${C'}_{h_0,h_1,h_2,h_3}^{h}(x,y)$
in the case that one of $h_i$ is equal to $0$.
In this case, the conformal blocks are
very easy to calculate.
We note that results obtained in this section is
always true for any vertex operator algebra.

First, we assume that $h_1=0$. 
Then, $A(h_0,0,h_2,h_3)$ is not empty if and only if $h_0 \in h_2 \star h_3$.
Note that in this case $A(h_0,0,h_2,h_3)= \{h_0\}$.

For $h_0 \in h_2 \star h_3$,
by Lemma \ref{vacuum_intertwining}
and the definition of the intertwining operator ${I'}_{h_2,h_3}^{h_0}(-,z)$,
\begin{align*}
{C'}_{h_0,0,h_2,h_3}^{h_0}(x,y)&=
\bra{h_0} {I'}_{0, h_0}^{h_0}(\ket{0}, x)
{I'}_{h_2,h_3}^{h_0} (\ket{h_2}, y) \ket{h_3}\\
&=
\bra{h_0} Y(\1,x)
{I'}_{h_2,h_3}^{h_0} (\ket{h_2}, y) \ket{h_3}\\
&=
\bra{h_0} \mathrm{id}_{h_0}{I'}_{h_2,h_3}^{h_0} (\ket{h_2}, y) \ket{h_3}\\
&= y^{h_0-h_2-h_3}.
\end{align*}
Similarly, in the case of $h_2=0$,
for any $h_0 \in h_1 \star h_3$,
\begin{align*}
{C'}_{h_0,h_1,0,h_3}^{h_3}(x,y)
= x^{h_0-h_1-h_3}.
\end{align*}

Second, we assume that $h_3=0$.
Then, $A(h_0,h_1,h_2,0)$ is not empty if and only if $h_0 \in h_1 \star h_2$,
and in this case $A(h_0,h_1,h_2,0)=\{h_2\}$.

By Lemma \ref{vacuum_intertwining}
and $\bra{h_0}(L(-1)-)=0$,
\begin{align*}
{C'}_{h_0,h_1,h_2,0}^{h_2}(x,y)&=
\bra{h_0} {I'}_{h_1, h_2}^{h_0}(\ket{h_1}, x)
{I'}_{h_2,0}^{h_2} (\ket{h_2}, y) \ket{0}\\
&=
\bra{h_0} {I'}_{h_1, h_2}^{h_0}(\ket{h_1}, x)
\exp(L(-1)y)Y(\ket{0},-y)\ket{h_2}\\
&=
\bra{h_0} {I'}_{h_1, h_2}^{h_0}(\ket{h_1}, x)
\exp(L(-1)y)\ket{h_2}\\
&=
\bra{h_0} {I'}_{h_1, h_2}^{h_0}(\ket{h_1}, x-y)\ket{h_2}\\
&= \sum_{k=0}^\infty \binom{h_0-h_1-h_2}{k} x^{h_0-h_1-h_2-k}y^{k} = (x-y)^{h_0-h_1-h_2}|_{|x|>|y|}.
\end{align*}

Finally,  we assume that $h_0=0$.
Then, $A(0,h_1,h_2,h_3)$ is not empty if and only if $h_1 \in h_2 \star h_3$,
and in this case $A(0,h_1,h_2,h_3)=\{h_1\}$.

By Lemma \ref{vacuum_intertwining}, it is easy to show the following lemma:
\begin{lem}
For any $h \in \Is$ and $v \in L(\ft,h)$,
$
\bra{0}I_{h,h}^{0}(\ket{h},z)v\rangle
= \bra{h}\exp(L(1)z^{-1})v \rangle \cdot z^{-2h}.
$
\end{lem}

By the above lemma and Lemma \ref{exp_L1},
\begin{align*}
{C'}_{0,h_1,h_2,h_3}^{h_1}(x,y)&=
\bra{0} {I'}_{h_1, h_1}^{0}(\ket{h_1}, x)
{I'}_{h_2,h_3}^{h_1} (\ket{h_2}, y) \ket{h_3}\\
&= \bra{h_1}\exp(L(1)x^{-1}){I'}_{h_2,h_3}^{h_1} (\ket{h_2}, y) \ket{h_3}\cdot x^{-2h_1}\\
&= \bra{h_1}\exp(L(1)x^{-1}){I'}_{h_2,h_3}^{h_1} (\ket{h_2}, y) \ket{h_3}\cdot x^{-2h_1}\\
&= x^{-2h_1}(1-x^{-1}y)^{-2h_2}
\bra{h_1}{I'}_{h_2,h_3}^{h_1} (\ket{h_2}, y/(1-yx^{-1})) \ket{h_3}\\
&= x^{-2h_1}(1-x^{-1}y)^{-2h_2}(y/(1-yx^{-1}))^{h_1-h_2-h_3}\\
&=x^{h_2-h_1-h_3}y^{h_1-h_2-h_3}(x-y)^{h_3-h_1-h_2}|_{|x|>|y|}
\end{align*}

Thus, we have:
\begin{prop}
Let $h_0,h_1,h_2,h_3 \in \Is$ with $h\in A(h_0,h_1,h_2,h_3)$.
In the case that one of $h_i$ is equal to $0$,
the conformal blocks are given in the following table:
\end{prop}

\begin{table}[h]
\caption{Conformal blocks with vacuum sector}
\label{table_vacuum}
  \begin{tabular}{|c|c|} \hline
 & ${C'}_{h_0,h_1,h_2,h_3}^{h}(x,y)$ \\ \hline
$h_0=0$ & $x^{h_2-h_1-h_3}y^{h_1-h_2-h_3}(x-y)^{h_3-h_1-h_2}$\\
$h_1=0$ & $y^{h_0-h_2-h_3}$\\
$h_2=0$ & $x^{h_0-h_1-h_3}$\\
$h_3=0$ & $(x-y)^{h_0-h_1-h_2}$\\ \hline
\end{tabular}
\end{table}

\subsection{Conformal block -- general cases}\label{sec_block_gen}
In this section, we study conformal blocks in general cases,
that is, none of $\{h_i\}_{i=0,1,2,3}$ is not $0$.
We explain the method to calculate the conformal blocks introduced in \cite{BPZ}.

Since
\begin{align*}
h_0 &\bra{h_0} I_{h_1, h}^{h_0}(\ket{h_1}, x)
I_{h_2,h_3}^h (\ket{h_2}, y) \ket{h_3}\\
&=  \bra{h_0} L(0) I_{h_1, h}^{h_0}(\ket{h_1}, x)
I_{h_2,h_3}^h (\ket{h_2}, y) \ket{h_3} \\
&=\bra{h_0} [L(0), I_{h_1, h}^{h_0}(\ket{h_1}, x)]
I_{h_2,h_3}^h (\ket{h_2}, y) \ket{h_3}
+\bra{h_0} I_{h_1, h}^{h_0}(\ket{h_1}, x)
[L(0), I_{h_2,h_3}^h (\ket{h_2}, y) ]\ket{h_3}\\
+&\bra{h_0} I_{h_1, h}^{h_0}(\ket{h_1}, x)
I_{h_2,h_3}^h (\ket{h_2}, y) L(0) \ket{h_3} \\
&=(x\pa_x + y \pa_y +h_1+h_2+h_3)C_{(h_1,h_2,h_3)}^{h_0,h}(x,y),
\end{align*}
we have
$$
(x\pa_x + y \pa_y +h_1+h_2+h_3-h_0)C_{h_0,h_1,h_2,h_3}^{h}(x,y)=0.
$$
Thus, by Lemma \ref{generalized_limit}, 
$$x^{-h_0+h_1+h_2+h_3}C_{h_0,h_1,h_2,h_3}^{h}(x,y) \in 
\C[[(\frac{y}{x})^\R]].$$
Thus, the evaluation of $x=1$ is well-defined.
Set
\begin{align}
{C'}_{h_0,h_1,h_2,h_3}^{h}(z)=
\lim_{x \to 1} {C'}_{h_0,h_1,h_2,h_3}^{h}(x,z)
= \bra{h_0} {I'}_{h_1, h}^{h_0}(\ket{h_1}, 1)
{I'}_{h_2,h_3}^h (\ket{h_2}, z) \ket{h_3},
\end{align}
which is a formal power series in $z^{h-h_2-h_3}\C((z))$.
Then, we have:
\begin{lem}\label{block_limit}
For any $h_0,h_1,h_2,h_3 \in \Is$ with $h\in A(h_0,h_1,h_2,h_3)$, the following conditions hold:
\begin{enumerate}
\item
${C'}_{h_0,h_1,h_2,h_3}^{h}(x,y)
= x^{h_0-h_1-h_2-h_3}{C'}_{h_0,h_1,h_2,h_3}^{h}(\frac{y}{x})$;
\item
${C'}_{h_0,h_1,h_2,h_3}^{h}(z)=1\cdot z^{h-h_2-h_3}+ \mathrm{O}(z^{h-h_2-h_3+1}) \in
z^{h-h_2-h_3}\C[[z]]$.
\end{enumerate}
\end{lem}
By the above lemma, 
in order to determine $C_{h_0,h_1,h_2,h_3}^{h}(x,y)$
it suffices to calculate ${C'}_{h_0,h_1,h_2,h_3}^{h}(z)$,
which is also called a {\it conformal block}.

It is convenient to set
\begin{align*}
h_a &= \frac{3}{4}a - \frac{1}{2},
\end{align*}
for $a \in \R$.
Then, 
$h_{\frac{4}{3}}=\frac{1}{2}$ and $h_{\frac{3}{4}}=\frac{1}{16}$.
The following lemma is shown by \cite{BPZ}:
\begin{lem}\label{block_differential}
For $h_0,h_1,h_2 \in \Is$ and $a \in \{\frac{4}{3},\frac{3}{4} \}$
and $h\in A(h_0,h_1,h_2,h_a)$,
the conformal block ${C'}_{h_0,h_1,h_2,h_a}^{h}(z)$
 satisfies the following differential equation:
%
%
\begin{align*}
\Bigl(
(1-z)z \pa_z^2 +((2\tilde{h}-2+a)z+ a)\pa_z
+\frac{1}{(1-z)z}((\tilde{h}(a+\tilde{h}-1) -ah_1)z^2- ah_2)
\Bigr)
{C'}_{h_0,h_1,h_2,h_a}^{h}(z),
\end{align*}
where $\tilde{h}=h_0-h_1-h_2-h_a$.
\end{lem}
\begin{proof}
By Lemma \ref{ward} and Lemma \ref{singular},
\begin{align*}
0&= \bra{h_0} {I'}_{h_1, h}^{h_0}(\ket{h_1}, x){I'}_{h_2,h_a}^h (\ket{h_2}, y) \bigl(L(-1)^2-a L(-2)\bigr) \ket{h_a} \\
&= (\pa_x+\pa_y)^2+a (x^{-1}\pa_x+y^{-1}\pa_y -  \frac{h_1}{x^2}- \frac{h_2}{y^2}) {C'}_{h_0,h_1,h_2,h_a}^{h}(x,y).
\end{align*}
By Lemma \ref{block_limit},
\begin{align*}
0&=\Bigl(
(\pa_y+\pa_z)^2+a (y^{-1}\pa_y+z^{-1}\pa_z -  \frac{h_1}{y^2}- \frac{h_2}{z^2})
\Bigr)
y^{\tilde{h}}{C'}_{h_0,h_1,h_2,h_3}^{h}(\frac{z}{y}).
\end{align*}
Thus, the assertion holds.
\end{proof}

Let $h_0,h_1,h_2,h_3 \in \Is$ with $h_3 =h_a = \frac{3}{4}a - \frac{1}{2}$ for $a\in \{\frac{4}{3},\frac{3}{4}\}$
and set
$$D_{h_0,h_1,h_2,h_3}=(1-z)z \pa_z^2 +((2\tilde{h}-2+a)z+ a)\pa_z
+\frac{1}{(1-z)z}((\tilde{h}(a+\tilde{h}-1) -ah_1)z^2- ah_2),$$
the differential operator in Lemma \ref{block_differential}.

The explicit descriptions of the differential equations are given in Table \ref{table_differential}:
\begin{table}[h]
\label{table_differential}
\caption{Differential equations}
  \begin{tabular}{|l|l|} \hline
$(h_0,h_1,h_2,h_3)$ & $D_{h_0,h_1,h_2,h_3}$ \\ \hline \hline
$(\frac{1}{16},\frac{1}{16},\frac{1}{16},\frac{1}{16}) $
& 
$z(1 - z)\pa_z^2 + (3/4 - 3/2 z) \pa_z - 3/(64 z (1 - z))$
 \\ \hline
$(\frac{1}{2},\frac{1}{2},\frac{1}{2},\frac{1}{2}) $
& $z (1 - z)\pa_z^2+ (4/3 - 8/3 z)\pa_z- 2/(3 z (1 - z))$ \\ \hline
$(\frac{1}{2},\frac{1}{2},\frac{1}{16},\frac{1}{16}) $
&$z (1 - z) \pa_z^2 + (3/4 - 3/2 z) \pa_z -
(3 + 21z^2)/(64z (1 -z)) $  \\
$(\frac{1}{2},\frac{1}{16},\frac{1}{2},\frac{1}{16}) $
& $z (1 - z) \pa_z^2 + (3/4 - 3/2 x) \pa_z - 3/(8z (1 - z))$ \\
$(\frac{1}{16},\frac{1}{2},\frac{1}{2},\frac{1}{16}) $
& $z (1 - z) \pa_z^2 + (3/4 - 13/4 z) \pa_z
- (3 -7 z^2)/(8z (1 - z)) $ \\
$(\frac{1}{2},\frac{1}{16},\frac{1}{16},\frac{1}{2}) $
& $z (1 - z) \pa_z^2 + (4/3 - 11/12 z) \pa_z
-(16 +21 z^2)/(192 z (1 - z)) $ \\
$(\frac{1}{16},\frac{1}{2},\frac{1}{16},\frac{1}{2}) $ &
$z (1 - z) \pa_z^2+ (4/3 - 8/3 z) \pa_z - 1/(12z (1 - z))$ \\
$(\frac{1}{16},\frac{1}{16},\frac{1}{2},\frac{1}{2}) $
& $z (1 - z)\pa_z^2 + (4/3 - 8/3 z) \pa_z -
(8 -7 z^2)/(12z (1 - z))$ \\ \hline
\end{tabular}
\end{table}

Those differential equations have possible singular points at $\{0,1,\infty\}$,
but all singular points are regular.
Thus, their formal solutions ${C'}_{h_0,h_1,h_2,h_3}^{h}(z)$ are absolutely convergent
in $\{z\in \C\;|\;|z|<1\}$ and have analytic continuation to multivalued holomorphic functions on $\CPm$.
Since $z^r$ is not single-valued around $z=0$, 
we need to consider a branch cut.

Set $\R_-=\{r \in \R\;|\; r\leq 0\}$ and define the holomorphic function
$\mathrm{Log}:\C \setminus \R_- \rightarrow \C$ by
$\mathrm{Log}(z)=log|z|+\mathrm{Arg}(z)$ with $\mathrm{Arg}(z) \in (-\pi,\pi)$,
called {\it the principle value}.
The following lemma is clear:
\begin{lem}
\label{log}
Let $\mu,\nu \in \C \setminus \R_-$ such that $\mu\nu \notin \R_-$.
Then,
\begin{align*}
\Log(\mu\nu)=\Log(\mu)+\Log(\nu)+
\begin{cases}
2\pi i & \Arg(\al)+\Arg(\be) > \pi \\
0 & -\pi < \Arg(\al)+\Arg(\be) < \pi \\
-2\pi i & \Arg(\al)+\Arg(\be) <- \pi .
\end{cases}
\end{align*}
\end{lem}

Then, for $\mu \in \C \setminus \R_-$ and $r\in \R$,
$\mu^r$ is defined by $\mu^r=\exp(r\Log(\mu))$.
By Lemma \ref{block_differential}, the formal power series $C_{h_0,h_1,h_2,h_3}^{h}(z)$ is absolutely convergent to a single-valued holomorphic function
in this branch $V_+ = \C \setminus \R_- \cap \{|z|<1\}$.
Let denote this holomorphic function by $C_{h_0,h_1,h_2,h_3}^{h}(z)|_{V_+}$
Since $D_{h_0,h_1,h_2,h_3}$ is a second order differential equation with regular singularities, by the standard method,
we can determine the conformal blocks explicitly.
\begin{prop}
\label{list_block}
The conformal blocks are given in Table \ref{table_block}:
\end{prop}
\begin{table}[h]
\caption{Conformal blocks} \label{table_block}
  \begin{tabular}{|l|c||c|} \hline
$(h_0,h_1,h_2,h_3)$ & $h$ & ${C'}_{h_0,h_1,h_2,h_3}^{h}(z)|_{V_+}$ \\ \hline
$(\frac{1}{16},\frac{1}{16},\frac{1}{16},\frac{1}{16}) $
& $0$ & $\frac{1}{2}z^{-\frac{1}{8}}(1-z)^{-\frac{1}{8}}\Bigl(\sqrt{1-\sqrt{z}}
+\sqrt{1+\sqrt{z}} \Bigr)$
 \\
    & $\frac{1}{2}$ & $z^{-\frac{1}{8}}(1-z)^{-\frac{1}{8}}\Bigl(\sqrt{1+\sqrt{z}}
 - \sqrt{1-\sqrt{z}} \Bigr)$ \\ \hline
$(\frac{1}{2},\frac{1}{2},\frac{1}{2},\frac{1}{2}) $
& $0$ & $z^{-1}(1-z)^{-1}(1-z+z^2)$ \\ \hline

$(\frac{1}{2},\frac{1}{2},\frac{1}{16},\frac{1}{16}) $
& $0$ & $z^{-\frac{1}{8}}(1-z)^{-\frac{1}{2}}(1-\frac{z}{2})
$ \\
$(\frac{1}{2},\frac{1}{16},\frac{1}{2},\frac{1}{16}) $
& $\frac{1}{16}$ & $z^{-\frac{1}{2}}(1-z)^{-\frac{1}{2}}
(1-2z)$ \\
$(\frac{1}{16},\frac{1}{2},\frac{1}{2},\frac{1}{16}) $
& $\frac{1}{16}$ & $z^{-\frac{1}{2}}(1-z)^{-1}(1+z)$\\
$(\frac{1}{2},\frac{1}{16},\frac{1}{16},\frac{1}{2}) $
& $\frac{1}{16}$ & $z^{-\frac{1}{2}}(1-z)^{-\frac{1}{8}}(1+z)$ \\
$(\frac{1}{16},\frac{1}{2},\frac{1}{16},\frac{1}{2}) $
& $\frac{1}{16}$ & $z^{-\frac{1}{2}}(1-z)^{-\frac{1}{2}}(1-2z)$
 \\
$(\frac{1}{16},\frac{1}{16},\frac{1}{2},\frac{1}{2}) $
& $0$ & $z^{-1}(1-z)^{-\frac{1}{2}}(1-\frac{z}{2})$ \\ \hline
\end{tabular}
\end{table}

Here, we need to fix the branches of functions in the table.
Since the ranges $\{1-\mu \;|\; \mu \in V_+\}$
and $\{1 \pm \sqrt{\mu}\;|\; \mu \in V_+\}$ are contained in $\C \setminus \R_-$,
$(1-z)^r$ (resp. $\sqrt{1 \pm \sqrt{z}}$) on $V_+$ is defined by $\exp(r\Log(1-z))$
(resp. $\exp(\ft \Log(1\pm \sqrt{z}))$),
which fix the branches.

\begin{rem}
It is noteworthy that the differential equation is independent of the choice of an intermediate state $h \in A(h_0,h_1,h_2,h_3)$. By Lemma \ref{block_limit}, for $h \in  A(h_0,h_1,h_2,h_3)$,
the asymptotic behavior of $C_{h_0,h_1,h_2,h_3}^{h}(z)$ is $z^{h-h_2-h_3}$.
Thus, different choices of $h \in  A(\fs,\fs,\fs,\fs)=\{0,\ft\}$ correspond to 
linearly independent solution of $D_{\fs,\fs,\fs,\fs}$.
\end{rem}

Looking at Table \ref{table_block}, we find that the conformal block has some symmetry:
\begin{align*}
(h_0,h_1,h_2,h_3) &\leftrightarrow (h_0,h_2,h_1,h_3)\\
z &\leftrightarrow z^{-1}.
\end{align*}

In order to see this symmetry precisely,
we have to consider the conformal block of two variables ${C'}_{h_0,h_1,h_2,h_3}^{h}(y,z)$
and the true normalization of intertwining operators $I(-,z)$ \eqref{eq_normal}.
%
%
Since $C_{h_0,h_1,h_2,h_3}^{h}(x,y)$ is a scalar multiple of
${C'}_{h_0,h_1,h_2,h_3}^{h}(x,y)$,
the convergence of the formal power series $C_{h_0,h_1,h_2,h_3}^{h}(x,y)$
follows from the convergence of ${C'}_{h_0,h_1,h_2,h_3}^{h}(z)$.

To be more precise,
we fix the branch of the $C_{h_0,h_1,h_2,h_3}^{h}(x,y)$.

The series $C_{h_0,h_1,h_2,h_3}^{h}(x,y)$ is absolutely convergent
to a single-valued holomorphic function
in 
\begin{align}
\{(x,y) \in (\C \setminus \R_-) \times (\C \setminus \R_-)
\;|\; |x|>|y|\}, \label{eq_region}
\end{align}
which follows from the convergence of the series ${C'}_{h_0,h_1,h_2,h_3}^{h}(z)$
and Lemma \ref{block_limit}.
However, in \eqref{eq_region}, $(x^{-1}y)^r=\exp(r \Log(x^{-1}y))$ is
ill-defined if $x^{-1}y \in \R_-$ and is not always equal to
$x^{-r}y^r=\exp(-r\Log(x))\exp(r\Log(y))$ (see Lemma \ref{log}).
Thus, it is convenient to
consider smaller region, like $\{(x,y)\in \C^2\;|\; \Arg(x),\Arg(y) \text{ are small}\}$,
where $\Log$ is additive.

Thus, 
set $(x_0, y_0) = (4+i, 2)$.
For $\epsilon > 0$,
set 
$$U_+(\epsilon)=\left\{(x,y) \in \C^2\;|\;
|x-x_0|<\epsilon, |y-y_0|<\epsilon \right\},$$
which is a neighborhood of $(x_0,y_0)$.
If $\epsilon$ is sufficiently small, then
$U_+(\epsilon)$ is contained in $Y_2(\C^\times)$ and
$\{y/x \in \C \;|\; (x,y)\in U_+ \}$
is contained in $V_+$ since $y/x \sim y_0/x_0 \in V_+$.
Take such $\epsilon >0$ and set $U_+=U_+(\epsilon)$.
The precise values of $x_0,y_0,\epsilon$ do not matter.

In this region $U_+$, $C_{h_0,h_1,h_2,h_3}^{h}(x,y)$ define a single-valued holomorphic function.
Let denote this holomorphic function by $C_{h_0,h_1,h_2,h_3}^{h}(x,y)|_{U_+}$.

Then, by Proposition \ref{list_block} and Lemma \ref{block_limit},
we have:
\begin{prop}
\label{list_block2}
The conformal blocks $C_{h_0,h_1,h_2,h_3}^{h}(x,y)$ are given in Table \ref{table_block2}:
\end{prop}
%
%
\begin{table}[h]
\caption{Conformal blocks}
\label{table_block2}
  \begin{tabular}{|l|c||c|} \hline
$(h_0,h_1,h_2,h_3)$ & $h$ & $C_{h_0,h_1,h_2,h_3}^{h}(x,y)|_{U_+}$ \\ \hline \hline
$(\frac{1}{16},\frac{1}{16},\frac{1}{16},\frac{1}{16}) $
& $0$ & $\frac{1}{2}\{xy(x-y)\}^{-\frac{1}{8}}
\Bigl((x^{\frac{1}{2}}+y^{\frac{1}{2}})^{\frac{1}{2}} +
(x^{\frac{1}{2}}-y^{\frac{1}{2}})^{\frac{1}{2}} \Bigr)$
 \\
    & $\frac{1}{2}$ & $\frac{1}{2} \{xy(x-y)\}^{-\frac{1}{8}}
\Bigl((x^{\frac{1}{2}}+y^{\frac{1}{2}})^{\frac{1}{2}} - 
(x^{\frac{1}{2}}-y^{\frac{1}{2}})^{\frac{1}{2}} \Bigr)$
 \\ \hline
$(\frac{1}{2},\frac{1}{2},\frac{1}{2},\frac{1}{2}) $
& $0$ & $\{xy(x-y)\}^{-1}(x^2-xy+y^2)$ \\ \hline

$(\frac{1}{2},\frac{1}{2},\frac{1}{16},\frac{1}{16}) $
& $0$ & $y^{-\frac{1}{8}}\{x(x-y)\}^{-\frac{1}{2}}
(x-\frac{y}{2})$ \\
$(\frac{1}{2},\frac{1}{16},\frac{1}{2},\frac{1}{16}) $
& $\frac{1}{16}$ &
$\frac{1}{2}x^{-\frac{1}{8}}\{y(x-y)\}^{-\frac{1}{2}}
(x-2y)$ \\
$(\frac{1}{16},\frac{1}{2},\frac{1}{2},\frac{1}{16}) $
& $\frac{1}{16}$ & $\frac{1}{2}(xy)^{-\frac{1}{2}} (x-y)^{-1}(x+y)$\\
$(\frac{1}{2},\frac{1}{16},\frac{1}{16},\frac{1}{2}) $
& $\frac{1}{16}$ & $\frac{1}{2}\{xy\}^{-\frac{1}{2}}(x-y)^{-\frac{1}{8}}(x+y)$ \\
$(\frac{1}{16},\frac{1}{2},\frac{1}{16},\frac{1}{2}) $
& $\frac{1}{16}$ & $\frac{1}{2}x^{-1}\{y(x-y)\}^{-\frac{1}{2}}(x-2y)$\\
$(\frac{1}{16},\frac{1}{16},\frac{1}{2},\frac{1}{2}) $
& $0$ & $y^{-1}\{x(x-y)\}^{-\frac{1}{2}}(x-\frac{y}{2})$ \\ \hline
\end{tabular}
\end{table}

Here, the branches of functions in the table is defined as follows:
Let $(x,y) \in U_+$.
Since $\mathrm{Re}(x-y) > 0$ (resp. $\mathrm{Re}(x^\ft-y^\ft) > 0$), $(x-y)^r$ (resp. $\left( x^\ft-y^\ft \right)^\ft$) is well-defined for the principle value.
Thus, all the functions in the table are specified.
We note that in $U_+$ $\Log$ is additive, e.g.,
$(x^{-1}y)^r=x^{-r}y^r$ by Lemma \ref{log}. Hence, there is no ambiguity in the expressions.

Now, we can see the following symmetry from the table \ref{table_block2}:
\begin{align*}
(h_0,h_1,h_2,h_3) &\leftrightarrow (h_0,h_2,h_1,h_3)\\
(x,y) &\leftrightarrow (y,x).
\end{align*}
The precise statement of this symmetry is given in the next section.

\begin{rem}
\label{rem_four_point}
If we consider four point correlation functions
$$C_{h_0,h_1,h_2,h_3}^{h}(z_0,z_1,z_2,z_3)
= \langle I_{h_0,h_0}^{0}(a_0,z_0)I_{h_1,h}^{h_0}(a_1,z_1)I_{h_2,h_3}^h(a_2,z_2)I_{h_3,0}^{h_3}(a_3,z_3)\1\rangle,$$
then the conformal block admits $S_4$-symmetry or more precisely
the symmetry of the pure braided group $P_4$.
We note that this four point correlation function can be recovered from
the most degenerate one $C_{h_0,h_1,h_2,h_3}^{h}(z)$ (see \cite{M2}).
\end{rem}

\subsection{Connection matrix}\label{sec_conn}

We continue to use the notations and the definitions in the previous section.
Let $\si: Y_2(\C^\times) \rightarrow Y_2(\C^\times)$ be an involution defined by
$\si(x,y)=(y, x)$.
Let $\ga:[0,1] \rightarrow Y_2(\C^\times)$ be a continuous map such that $\ga(0)=(x_0,y_0)$ and $\ga(1)=(y_0,x_0)$.

For any $h_0,h_1,h_2,h_3 \in \Is$ and $h\in A(h_0,h_1,h_2,h_3)$,
the holomorphic function $C_{h_0,h_1,h_2,h_3}^{h}(x,y)|_{U_+}$
has an analytic continuation to the multivalued holomorphic function on $Y_2(\C^\times)$.
Let $A_\ga \left(C_{h_0,h_1,h_2,h_3}^{h}(x,y)\right)$ denote the function obtained by the analytic continuation of $C_{h_0,h_1,h_2,h_3}^{h}(x,y)|_{U_+}$ along the path $\ga$.
The function $A_\ga\left(C_{h_0,h_1,h_2,h_3}^{h}(x,y)\right)$ is defined 
on $\si(U_+)$, a neighborhood of $(y_0,x_0) \in Y_2(\C^\times)$.
The pull back of $A_\ga\left(C_{h_0,h_1,h_2,h_3}^{h}(x,y)\right)$ by
$\si:Y_2(\C^\times)\rightarrow Y_2(\C^\times)$ defines a holomorphic function on $U_+$.
We denote it by $\si^*A_\ga\left(C_{h_0,h_1,h_2,h_3}^{h}(x,y)\right)$.

Then, $\si^*A_\ga\left(C_{h_0,h_1,h_2,h_3}^{h}(x,y)\right)$
is a linear sum of functions
$\left\{ C_{h_0,h_2,h_1,h_3}^{h'}(x,y)|_{U_+}  \right\}_{h' \in A(h_0,h_2,h_1,h_3)}$,
which can be checked case by case by Table \ref{table_vacuum} and Table \ref{table_block2}.
 
More precisely, there exists a unique matrix $\{ B_{h_0,h_1,h_2,h_3}^{h,h'}(\ga)\}_{h' \in A(h_0,h_2,h_1,h_3)}$ such that
\begin{align*}
\si^* A_\ga \Bigl(C_{h_0,h_1,h_2,h_3}^{h}(x,y)\Bigr)
= \sum_{h' \in A(h_0,h_2,h_1,h_3)} B_{h_0,h_1,h_2,h_3}^{h,h'}(\ga)
C_{h_0,h_2,h_1,h_3}^{h'}(x,y)|_{U_+}.
\end{align*}
It is noteworthy that the matrix $\{ B_{h_0,h_1,h_2,h_3}^{h,h'}(\ga)\}_{h' \in A(h_0,h_2,h_1,h_3)}$ depends only on the homotopy class of the path $\ga$.

Let $\ga_0:[0,1]\rightarrow Y_2(\C^\times)$ be the path given by
$$
\ga_0(t)=(x_t,y_t)  = 
\left(\frac{x_0+y_0}{2} + \exp(\pi it) \frac{x_0-y_0}{2},
\frac{x_0+y_0}{2} - \exp(\pi it) \frac{x_0-y_0}{2}\right).
$$
Then, $\ga_0(0)=(x_0,y_0)$ and $\ga_0(1)=(y_0,x_0)$.

Hereafter, we calculate the connection matrices $\{ B_{h_0,h_1,h_2,h_3}^{h,h'}(\ga_0)\}$ for the path $\ga_0$.
Set $$B_{h_0,h_1,h_2,h_3}^{h,h'} = B_{h_0,h_1,h_2,h_3}^{h,h'}(\ga_0).$$

Since $x_t-y_t= \exp(\pi it)(x_0-y_0)$,
$x_t-y_t:[0,1] \rightarrow \C$ crosses the negative real line $\R_-$ from the above once,
and $x_t,y_t:[0,1] \rightarrow \C$ do not cross $\R_-$.
By using those facts, we can calculate the connection matrix of a function $x^py^q(x-y)^rP(x,y) = P(x,y) x^{p+r}y^q\sum_{k=0}^\infty \binom{r}{k}(-1)^k (y/x)^k|_{U_+}$,
where $p,q,r \in \R$ and $P(x,y) \in \C[x,y]$, a polynomial.
The branch of this function is chosen in the same way as in the previous section. Then, we have:
\begin{lem}
\label{connection_formula}
For any $p,q,r \in \R$ and $P(x,y) \in \C[x,y]$,
$$\si^* A_{\ga_0}(x^{p}y^q(x-y)^r  P(x,y))
= \exp(\pi ir) y^p x^q (x-y)^r P(y,x).
$$
\end{lem}
All connection matrices except for the case of $(h_0,h_1,h_2,h_3)=(\fs,\fs,\fs,\fs)$ can be calculated
by the above lemma. We note that except for this case $(h_0,h_1,h_2,h_3)=(\fs,\fs,\fs,\fs)$,
there is only one possible intermediate state $A(h_0,h_1,h_2,h_3)$,
and the connection matrices $B_{h_0,h_1,h_2,h_3}^{h,h'}$ are not matrices but scalars. Thus we sometimes omit the index of $B_{h_0,h_1,h_2,h_3}^{h,h'}$ as $B_{h_0,h_1,h_2,h_3}$ in these cases.

By Table \ref{table_vacuum}, we have:
\begin{lem}\label{connection_vacuum}
For any $h_0,h_1,h_2,h_3 \in \Is$
and $h \in h_2\star h_3$,
\begin{align*}
B_{0,h_1,h_2,h_3}^{h_1,h_2} &=\exp(\pi i(h_3-h_1-h_2)),\\
B_{h_0,0,h_2,h_3}^{h_0,h_3} &=1,\\
B_{h_0,h_1,0,h_3}^{h_3,h_0} &=1,\\
B_{h_0,h_1,h_2,0}^{h_2,h_1} &=\exp(\pi i(h_0-h_1-h_2)).
\end{align*}
\end{lem}

Similarly, by Table \ref{table_block2}, we have:
\begin{lem}\label{connection_general}
For any $h_0,h_1,h_2,h_3 \in \Is$ and $h \in A(h_0,h_1,h_2,h_3)$ such that none of $\{h_i\}_{i=0,1,2,3}$
is $0$,
the connection matrices are given in Table \ref{table_braid}:
\end{lem}
\begin{table}[h]
\caption{Connection matrices}
\label{table_braid}
  \begin{tabular}{|l||c|} \hline
$(h_0,h_1,h_2,h_3)$ & $B_{h_0,h_1,h_2,h_3}^{h,h'}$ \\ \hline \hline
$(\frac{1}{16},\frac{1}{16},\frac{1}{16},\frac{1}{16}) $
& $
 \exp(-\frac{1}{8}\pi i)\begin{pmatrix}
\frac{1+i}{2} & \frac{1-i}{2} \\
\frac{1-i}{2} & \frac{1+i}{2} \\
\end{pmatrix}$
 \\ \hline
$(\frac{1}{2},\frac{1}{2},\frac{1}{2},\frac{1}{2}) $
& 
$-1$ \\ \hline
$(\frac{1}{2},\frac{1}{2},\frac{1}{16},\frac{1}{16}) $
&
$i$  \\
$(\frac{1}{2},\frac{1}{16},\frac{1}{2},\frac{1}{16}) $
& $i$ \\
$(\frac{1}{16},\frac{1}{2},\frac{1}{2},\frac{1}{16}) $
& 
$-1$ \\
$(\frac{1}{2},\frac{1}{16},\frac{1}{16},\frac{1}{2}) $
& $\exp(-\frac{1}{8}\pi i)$ \\
$(\frac{1}{16},\frac{1}{2},\frac{1}{16},\frac{1}{2}) $ &
$i$ \\
$(\frac{1}{16},\frac{1}{16},\frac{1}{2},\frac{1}{2}) $
& $i$ \\ \hline
\end{tabular}
\end{table}
\begin{proof}
By Lemma \ref{connection_formula},
it suffices to consider the case of $(h_0,h_1,h_2,h_3)=(\fs,\fs,\fs,\fs)$.
Along the path $\ga_0$, $x^{\ft}$, $y^{\ft}$ and $x^{\ft}+y^{\ft}$ remain in
the region $\{z \in \C\;|\; \mathrm{Re} z >0\}$.
Thus, the branches of $x^{\ft},y^{\ft}$ and $(x^{\ft}+y^{\ft})^{\ft}$ do not change.
Since $x^{\ft}-y^{\ft}=(x^{\ft}+y^{\ft})^{-1}(x-y)$ cross
$\R_-$ once,
\begin{align*}
\si^* A_{\ga_0} \left((x^{\frac{1}{2}}+y^{\frac{1}{2}})^{\frac{1}{2}} +
(x^{\frac{1}{2}}-y^{\frac{1}{2}})^{\frac{1}{2}}\right)
=(x^{\frac{1}{2}}+y^{\frac{1}{2}})^{\frac{1}{2}} +i
(x^{\frac{1}{2}}-y^{\frac{1}{2}})^{\frac{1}{2}}
\end{align*}
and
\begin{align*}
\si^* A_{\ga_0} \left((x^{\frac{1}{2}}+y^{\frac{1}{2}})^{\frac{1}{2}} -
(x^{\frac{1}{2}}-y^{\frac{1}{2}})^{\frac{1}{2}}\right)
=(x^{\frac{1}{2}}+y^{\frac{1}{2}})^{\frac{1}{2}} -i
(x^{\frac{1}{2}}-y^{\frac{1}{2}})^{\frac{1}{2}},
\end{align*}
which implies the assertion.
\end{proof}

Combining Lemma \ref{connection_vacuum} and
Lemma \ref{connection_general}, we obtain:
\begin{prop}\label{connection_one}
For any $*$,
\begin{align*}
B_{*,0,*,*}=B_{*,*,0,*}&=1\\
B_{*,\frac{1}{2},\frac{1}{2},*}&=-1\\
B_{a,\frac{1}{2},\frac{1}{16},a'}=
B_{a,\frac{1}{16},\frac{1}{2},a'}&=\begin{cases}
i & (a \text{ or }a' = \frac{1}{2})\\
-i & otherwise,
\end{cases}
\\
B_{a,\frac{1}{16},\frac{1}{16},a'}^{(b,b')}&=
\exp(-\frac{1}{8}\pi i) \times
\begin{cases}
1& (a,a' \neq \frac{1}{16}, a=a') \\
i& (a,a' \neq \frac{1}{16}, a \neq a') \\
\frac{1+ i}{2} & (a=a'=\frac{1}{16}, b=b')\\
\frac{1-i}{2} & (a=a'=\frac{1}{16}, b \neq b').
\end{cases}
\end{align*}
\end{prop}

We end this section by generalizing the
above results to the whole conformal blocks.
Recall that for $(h_0,h_1,h_2,h_3) \in \Is$ and $h\in A(h_0,h_1,h_2,h_3)$,
the conformal block is a linear map
$${C}_{h_0,h_1,h_2,h_3}^{h}: L(\ft,h_0)^\vee\otimes L(\ft,h_1)\otimes L(\ft,h_2)\otimes L(\ft,h_3)
\rightarrow \C((x))((y)),
$$
defined by
\begin{align*}
{C}_{h_0,h_1,h_2,h_3}^{h}(a_0^*,a_1,a_2,a_3;x,y)=
\langle 
a_0^*,  {I}_{h_1, h}^{h_0}(a_1, x)
{I}_{a_2,h_3}^h (a_2, y) a_3\rangle,
\end{align*}
for $a_0^* \in L(\ft,h_0)^\vee$ and $a_i \in L(\ft,h_i)$ ($i=1,2,3$).
As we will see in the following lemma,
the whole conformal blocks ${C}_{h_0,h_1,h_2,h_3}^{h}(a_0^*,a_1,a_2,a_3;x,y)$
 can be recovered from the specialized
conformal blocks ${C}_{h_0,h_1,h_2,h_3}^{h}(x,y)$
 \cite{BPZ}:
\begin{lem}
\label{ward_induction}
For any $\Delta_i \in \R$ and $a_0^* \in L(\ft,h_0)_{\Delta_0}^\vee$ and $a_i \in L(\ft,h_i)_{\Delta_i}$ ($i=0,1,2,3$),
the following equalities hold:
\begin{enumerate}
\item For any $n\in \Z$,
\begin{align*}
&{C}_{h_0,h_1,h_2,h_3}^{h}(L(-n)a_0^*,a_1,a_2,a_3;x,y)\\
&=
\left(x^{n+1}\frac{d}{dx}+y^{n+1}\frac{d}{dy}+
(n+1)x^n \Delta_{a_1}+(n+1)y^n \Delta_{a_2}\right)
{C}_{h_0,h_1,h_2,h_3}^{h}(a_0^*,a_1,a_2,a_3;x,y)\\
+&\sum_{k=1}^\infty\binom{n+1}{k+1} \left(
{C}_{h_0,h_1,h_2,h_3}^{h}(a_0^*,L(k)a_1,a_2,a_3;x,y)x^{n-k}+
{C}_{h_0,h_1,h_2,h_3}^{h}(a_0^*,a_1,L(k)a_2,a_3;x,y)y^{n-k}
\right)\\
+&{C}_{h_0,h_1,h_2,h_3}^{h}(a_0^*,a_1,a_2,L(n)a_3;x,y);
\end{align*}
\item For any $n\in \Z$,
\begin{align*}
&{C}_{h_0,h_1,h_2,h_3}^{h}(a_0^*,a_1,a_2,L(-n)a_3;x,y)\\
&= -\left(x^{-n+1}\frac{d}{dx}+y^{-n+1}\frac{d}{dy}+
(-n+1)x^{-n} \Delta_{a_1}+(-n+1)y^{-n} \Delta_{a_2}\right)
{C}_{h_0,h_1,h_2,h_3}^{h}(a_0^*,a_1,a_2,a_3;x,y)\\
-&\sum_{k=1}^\infty\binom{-n+1}{k+1} \left(
{C}_{h_0,h_1,h_2,h_3}^{h}(a_0^*,L(k)a_1,a_2,a_3;x,y)x^{-n-k}+
{C}_{h_0,h_1,h_2,h_3}^{h}(a_0^*,a_1,L(k)a_2,a_3;x,y)y^{-n-k}
\right)\\
+&{C}_{h_0,h_1,h_2,h_3}^{h}(L(n)a_0^*,a_1,a_2,a_3;x,y).
\end{align*}
\item
\begin{align*}
{C}_{h_0,h_1,h_2,h_3}^{h}(a_0^*,L(-1)a_1,a_2,a_3;x,y)&=
 \frac{d}{dx}{C}_{h_0,h_1,h_2,h_3}^{h}(a_0^*,a_1,a_2,a_3;x,y)\\
{C}_{h_0,h_1,h_2,h_3}^{h}(a_0^*,a_1,L(-1)a_2,a_3;x,y)&=
 \frac{d}{dy}{C}_{h_0,h_1,h_2,h_3}^{h}(a_0^*,a_1,a_2,a_3;x,y).
\end{align*}
\item For $m \geq 0$,
\begin{align*}
{C}_{h_0,h_1,h_2,h_3}^{h}&(a_0^*,L(-m-2)a_1,a_2,a_3;x,y)\\
&=
\sum_{k \geq 0} \Bigl(\binom{k}{m} x^{k-m}
{C}_{h_0,h_1,h_2,h_3}^{h}(L(2+k)a_0^*,a_1,a_2,a_3;x,y)\\
&+\binom{-k-1}{m} \frac{1}{(x-y)^{-k-m-1}}|_{|x|>|y|}
{C}_{h_0,h_1,h_2,h_3}^{h}(a_0^*,a_1,L(k-1)a_2,a_3;x,y)\\
&+
\binom{-k-1}{m}
{C}_{h_0,h_1,h_2,h_3}^{h}(a_0^*,a_1,a_2,L(k-1)a_3;x,y)
x^{-1-k-m}
\Bigr).
\end{align*}
\item
For $m\geq 0$,
\begin{align*}
{C}_{h_0,h_1,h_2,h_3}^{h}&(a_0^*,a_1,L(-m-2)a_2,a_3;x,y)\\
&=
\sum_{k \geq 0} \Bigl(\binom{k}{m} y^{k-m}
{C}_{h_0,h_1,h_2,h_3}^{h}(L(2+k)a_0^*,a_1,a_2,a_3;x,y)\\
&+(-1)^{k+m+1}\binom{-k-1}{m} \frac{1}{(x-y)^{-k-m-1}}|_{|x|>|y|}
{C}_{h_0,h_1,h_2,h_3}^{h}(a_0^*,L(k-1)a_1,a_2,a_3;x,y)\\
&+
\binom{-k-1}{m}
{C}_{h_0,h_1,h_2,h_3}^{h}(a_0^*,a_1,a_2,L(k-1)a_3;x,y)
y^{-1-k-m}
\Bigr).
\end{align*}
\end{enumerate}
\end{lem}
\begin{proof}
By Lemma \ref{ward},
\begin{align*}
&{C}_{h_0,h_1,h_2,h_3}^{h}(a_0^*,L(-m-2)a_1,a_2,a_3;x,y)\\
&=a_0^*\left(\left(\frac{1}{m!}\frac{d}{dx}^m T^-(x)\right) {I}_{h_1, h}^{h_0}(a_1, x)
{I}_{a_2,h_3}^h (a_2, y) a_3\right)
+a_0^*\left({I}_{h_1, h}^{h_0}(a_1, x) \left(\frac{1}{m!}\frac{d}{dx}^m T^+(x) \right)
{I}_{h_2,h_3}^h (a_2, y) a_3\right)\\
&=
a_0^*\left(\left(\frac{1}{m!}\frac{d}{dx}^m T^-(x)\right) {I}_{h_1, h}^{h_0}(a_1, x)
{I}_{h_2,h_3}^h (a_2, y) a_3\right)\\
&+
a_0^*\left({I}_{h_1, h}^{h_0}(a_1, x) 
\left(\sum_{k \geq 0}
\binom{-k-1}{m} \frac{1}{(x-y)^{-k-m-1}}|_{|x|>|y|}{I}_{h_2,h_3}^h (L(k-1)a_2, y) \right)a_3\right)\\
&+ 
a_0^*\left({I}_{h_1, h}^{h_0}(a_1, x) {I}_{a_2,h_3}^h (a_2, y)  \left(\frac{1}{m!}\frac{d}{dx}^m T^+(x) \right) a_3\right)\\
&=
\sum_{k \geq 0} \Bigl(\binom{k}{m} x^{k-m}
{C}_{h_0,h_1,h_2,h_3}^{h}(L(2+k)a_0^*,a_1,a_2,a_3;x,y)\\
&+\binom{-k-1}{m} \frac{1}{(x-y)^{-k-m-1}}|_{|x|>|y|}
{C}_{h_0,h_1,h_2,h_3}^{h}(a_0^*,a_1,L(k-1)a_2,a_3;x,y)\\
&+
\binom{-k-1}{m}
{C}_{h_0,h_1,h_2,h_3}^{h}(a_0^*,a_1,a_2,L(k-1)a_3;x,y)
x^{-1-k-m}
\Bigr).
\end{align*}

The other cases are obtained in the same way.
\end{proof}

The following proposition says that 
for any $a_i$ the connection matrix of
${C}_{h_0,h_1,h_2,h_3}^{h}(a_0^*,a_1,a_2,a_3;x,y)$
 is the same as ${C}_{h_0,h_1,h_2,h_3}^{h}(x,y).$
\begin{prop}
\label{single_monodromy}
For any $a_0^* \in L(\ft,h_0)^\vee$ and $a_i \in L(\ft,h_i)$ ($i=0,1,2,3$),
the formal power series ${C}_{h_0,h_1,h_2,h_3}^{h}(a_0^*,a_1,a_2,a_3;x,y)$
is absolutely convergent in $U_+$
and has the analytic continuation to a multivalued holmorphic function on $Y_2(\C^\times)$.
Furthermore, 
\begin{align*}
\si^* A_{\ga_0}\left({C}_{h_0,h_1,h_2,h_3}^{h}(a_0^*,a_1,a_2,a_3;x,y)\right)
=\sum_{h' \in A(h_0,h_2,h_1,h_3)} B_{h_0,h_1,h_2,h_3}^{h,h'}
{C}_{h_0,h_2,h_1,h_3}^{h'}(a_0^*,a_2,a_1,a_3;x,y)|_{U_+}.
\end{align*}
\end{prop}
\begin{proof}
We may assume that
$a_0^* \in L(\ft,h_0)_{h_0+k_0}^\vee$ and $a_i \in L(\ft,h_i)_{h_i+k_i}$ ($i=0,1,2,3$)
for some $k_i \in \Z_{\geq 0}$.
We prove the proposition by the induction on $K=k_0+k_1+k_2+k_3$.
For $K=0$, since $L(\ft,h)_h$ is spanned by $\ket{h}$,
the equation follows from the definition of the connection matrix.
For $K>0$, then one of $a_i$ can be written as $a_i= L(-n){a'}_i$ with $n > 0$.

We only consider the case of $a_1=L(-m-2){a'}_1$
with $m \geq 0$. All the other cases can be shown in the same way.
We note that for any (single-valued) polynomial $p(x,y) \in \C[x^\pm,y^\pm,(x-y)^\pm]$ and multivalued function $f(x,y)$,
$$\si^* A_{\ga_0}\left(p(x,y)f(x,y)\right)=
\si^* A_{\ga_0}\left(p(x,y)\right) \si^* A_{\ga_0}\left(f(x,y)\right)
=p(y,x)\si^* A_{\ga_0}\left(f(x,y)\right).$$
Thus, by Lemma \ref{ward_induction},
\begin{align*}
\si^* A_{\ga_0}&\left({C}_{h_0,h_1,h_2,h_3}^{h}(a_0^*,L(-m-2){a'}_1,a_2,a_3;x,y)\right)\\
&= 
\sum_{k \geq 0} \Bigl(\binom{k}{m} \si^* A_{\ga_0}\left(x^{k-m}
{C}_{h_0,h_1,h_2,h_3}^{h}(L(2+k)a_0^*,{a'}_1,a_2,a_3;x,y)\right)\\
&+\binom{-k-1}{m} \si^* A_{\ga_0}\left(\frac{1}{(x-y)^{-k-m-1}}|_{|x|>|y|}
{C}_{h_0,h_1,h_2,h_3}^{h}(a_0^*,a'_1,L(k-1)a_2,a_3;x,y)\right)\\
&+
\binom{-k-1}{m}\si^* A_{\ga_0}\left({C}_{h_0,h_1,h_2,h_3}^{h}(a_0^*,a'_1,a_2,L(k-1)a_3;x,y)\right)
x^{-1-k-m}
\Bigr)\\
&= 
\sum_{k \geq 0} \Bigl(\binom{k}{m}y^{k-m} \si^* A_{\ga_0}
\left({C}_{h_0,h_1,h_2,h_3}^{h}(L(2+k)a_0^*,{a'}_1,a_2,a_3;x,y)\right)\\
&+\binom{-k-1}{m} (-1)^{k+m+1}\frac{1}{(x-y)^{k+m+1}}|_{|x|>|y|}
\si^* A_{\ga_0}\left({C}_{h_0,h_1,h_2,h_3}^{h}(a_0^*,a'_1,L(k-1)a_2,a_3;x,y)\right)\\
&+
\binom{-k-1}{m} y^{-1-k-m}  \si^* A_{\ga_0}\left({C}_{h_0,h_1,h_2,h_3}^{h}(a_0^*,a'_1,a_2,L(k-1)a_3;x,y)\right)
\Bigr).
\end{align*}
By the induction hypothesis and Lemma \ref{ward_induction} (3),
\begin{align*}
\text{RHS}&=
\sum_{h' \in A(h_0,h_2,h_1,h_3)} B_{h_0,h_1,h_2,h_3}^{h,h'}
\Biggl(
\sum_{k \geq 0} \Bigl(\binom{k}{m}y^{k-m} {C}_{h_0,h_2,h_1,h_3}^{h'}(L(2+k)a_0^*,a_2,a'_1,a_3;x,y)\\
&+\binom{-k-1}{m} (-1)^{k+m+1}\frac{1}{(x-y)^{k+m+1}}|_{|x|>|y|}
{C}_{h_0,h_1,h_2,h_3}^{h'}(a_0^*,L(k-1)a_2,a'_1,a_3;x,y)\\
&+
\binom{-k-1}{m} y^{-1-k-m} {C}_{h_0,h_2,h_1,h_3}^{h'}(a_0^*,a_2,a'_1,L(k-1)a_3;x,y)
\Biggr),
\end{align*}
which is equal to 
$$
\sum_{h' \in A(h_0,h_2,h_1,h_3)} B_{h_0,h_1,h_2,h_3}^{h,h'}
{C}_{h_0,h_2,h_1,h_3}^{h'}(a_0^*,a_2,L(-m-2)a'_1,a_3;x,y)
$$
by Lemma \ref{ward_induction} (5).
Thus, the assertion holds.
\end{proof}

\subsection{Multi-index connections}\label{sec_multi_conn}
Let $l,r \in \Z_{>}$.
%
%
For $\la^i=(h_1^i,\dots,h_l^i, \bar{h}_1^i\dots,\bar{h}_r^i) \in 
\Isr$, $\la=(h_1,\dots,h_l,\h_1,\dots,h_r) \in A(\la^0,\la^1,\la^2,\la^3)$ and $\la'=(h'_1,\dots,h'_l,\h'_1,\dots,\h'_r) \in A(\la^0,\la^2,\la^1,\la^3)$,
a multi-index connection matrix is defined by
$$
B_{\la^0,\la^1,\la^2,\la^3}^{\la,\la'}
\equiv \Pi_{i=1}^l B_{h_i^0,h_i^1,h_i^2,h_i^3}^{h_i,h'_i} \Pi_{j=1}^r 
\overline{B_{\h_j^0,\h_j^1,\h_j^2,\h_j^3}^{\h_j,\bar{h}'_j}} \in \C,
$$
where $\overline{B_{\h_j^0,\h_j^1,\h_j^2,\h_j^3}^{\h_j,\bar{h}'_j}}$ is the complex conjugate of $B_{\h_j^0,\h_j^1,\h_j^2,\h_j^3}^{\h_j,\bar{h}'_j} \in \C$.
It is easy to show that
the multi-index connection matrix 
gives the connection matrix for 
the multi-index conformal blocks
introduced in Section \ref{sec_multi_def}.
More precisely, by Proposition \ref{single_monodromy} we have:
\begin{prop}
\label{multi_monodromy}
Let $\la^i \in \Isr$ ($i=0,1,2,3$) and $\la \in A(\la^0,\la^1,\la^2,\la^3)$.
For any $u_0^* \in L(\ft,h_0)^\vee$ and $u_i \in L(\ft,h_i)$ ($i=1,2,3$),
the formal power series ${C}_{\la^0,\la^1,\la^2,\la^3}^{\la}(u_0^*,u_1,u_2,u_3;x,y)$
is absolutely convergent in $U_+$
and has the analytic continuation to a multivalued holmorphic function on $Y_2(\C^\times)$.
Furthermore, 
\begin{align*}
\si^* A_{\ga_0}\left({C}_{\la^0,\la^1,\la^2,\la^3}^{\la}(u_0^*,u_1,u_2,u_3;x,y)\right)
=\sum_{\la' \in A(\la^0,\la^2,\la^1,\la^3)} B_{\la^0,\la^1,\la^2,\la^3}^{\la,\la'}
{C}_{\la^0,\la^1,\la^2,\la^3}^{\la'}(u_0^*,u_1,u_2,u_3;x,y)|_{U_+}.
\end{align*}
\end{prop}

The purpose of this section is
to give a convenient combinatorial description
of these connection matrices.
We will first identify $\Isr$ as a subset of $\Z_2^{l+r} \times \Z_2^{l+r}$.

We start from the case of $(l,r)=(1,0)$.
Define a map $\wt: \Z_2^2 \rightarrow \Q$ by
$$\wt(d,c)\equiv \frac{1}{16} d+ \frac{1}{2}c$$
for $(d,c)\in \Z_2^2$.
Then, we can identify $\Is$
as the subset $\{(d,c) \in \Z_2^2\;|\; dc=0 \}$,
i.e.,
\begin{align*}
0 \leftrightarrow (0,0), \;\;\;\;\;\;\;\;\;\frac{1}{2}\leftrightarrow(0,1),\;\;\;\;\;\;\;\; \frac{1}{16}\leftrightarrow(1,0).
\end{align*}
Then, the fusion rule can be written as
$$
(d,c)\star (d',c')=(d+d',(1+d)c'+(1+d')c+dd'\{0,1\})
$$
for $(d,c), (d',c') \in \Is$,
where we used the ring structure on $\Z_2$.

We generalize the above identification to
$\Isr$ and $\Z_2^{l+r}\times \Z_2^{l+r}$.
For $c=(c_1,\dots,c_l,\bar{c}_1,\dots,\bar{c}_r)$
and $c'=(c'_1,\dots,c'_l,\bar{c}'_1,\dots,\bar{c}'_r) \in \Z_2^{l+r}$,
we define 
$cc'=(c_1c'_1,\dots,c_lc'_l,\bar{c}_1\bar{c}'_1,\dots,\bar{c}_r\bar{c}'_r) \in \Z_2^{l+r}$,
that is, the usual ring structure on $\Z_2^{l+r}$.
We also define maps $|-|_l, |-|_r, |-|: \Z_2^{l+r} \rightarrow \Z$ by
$|c|_l=\sum_{i=1}^l c_i$, $|c|_r=\sum_{i=1}^r \bar{c}_i$
and $|c|=|c|_l - |c|_r$.
\begin{rem}
\label{rem_formula}
For $c^1,c^2 \in \Z_2^{l+r}$,
$|c^1+c^2|=|c^1|+|c^2|-2|c^1c^2|$.
In particular, $(-1)^{|c^1+c^2|}=(-1)^{|c^1|+|c^2|}$.
We will  frequently use this formula.
\end{rem}

An element of $\Z_2^{l+r}$ is called a codeword.
For a codeword $d=(d_1,\dots,d_l,\bar{d}_1,\dots,\bar{d}_r)$,
we define 
$$\Z_2^d \equiv 
\{c \in \Z_2^{l+r}\;|\; dc =c \},
$$
which is the subset of $\Z_2^{l+r}$ consisting of all codewords
supported by $d$. The all-one vector is a codeword $\reg=(1,1,\dots,1) \in \Z_2^{l+r}$ and set $d_\perp= \reg+d$ for $d\in \Z_2^{l+r}$. Then, $\Z_2^d$ is the kernel of the left multiplication by $\dpe$.

We identify $\Isr$ as a subset of 
$\Z_2^{l+r}\times \Z_2^{l+r}$
consisting of elements $(d,c) \in \Z_2^{l+r}\times \Z_2^{l+r}$ satisfying $dc=(0,\dots,0)$.
For $\mu=(d,c) \in \Isr$, $d$ is called {\it d-part} of $\mu$
and $c$ is called {\it c-part} of $\mu$.

An explicit bijection is given
by $\{(d,c) \in \Z_2^{l+r}\times \Z_2^{l+r}\;|\;dc=(0,\dots,0) \} \rightarrow \Isr$,
\begin{align*}
(d_1,\dots,d_l,\bar{d}_1,\dots,\bar{d_r},c_1,\dots,c_l,\bar{c}_1,\dots,\bar{c_r}) \mapsto (\wt(d_1,c_1),\dots,\wt(d_l,c_l),\wt(\bar{d}_1,\bar{c}_1)\dots,\wt(\bar{d}_r,\bar{c}_r)).
\end{align*}
Then, $\Isr = \coprod_{d \in \Z_2^{l+r}}\{d\} \times \Z_2^{\dpe}$
and we have:
\begin{lem}\label{fusion_formula}
The fusion product on $\Isr$ 
can be written as
$$(d^1,c^1) \star (d^2,c^2) 
= \{(d^1+d^2,(d^1+d^2)_\perp(c^1+c^2) + \ga )\}_{\ga \in \Z_2^{d^1d^2}}.
$$ for $(d^1,c^1),(d^2,c^2) \in \Isr$.
In particular,
$\left(\{d^1\} \times \Z_2^{\dpe^1}\right) \star \left( \{d^2\} \times \Z_2^{\dpe^2}\right)
\subset \{d^1+d^2\}\times \Z_2^{(d^1+d^2)_\perp}$.
\end{lem}
By the above lemma, the fusion rule can be described by using
the product on $\Z_2^{l+r}$.

The intermediate states can also be described by the product on $\Z_2^{l+r}$:
\begin{lem}\label{fusion}
Let $(d^0,c^0),(d^1,c^1),(d^2,c^2), (d^3,c^3) \in \Isr$.
\begin{enumerate}
\item
For $\al \in \Z_2^{l+r}$, $(0,\al)$ is a simple current,
i.e., $\#(0,\al)\star (d,c)=1$ for any $(d,c)\in \Isr$;
\item
$(0,0) \in (d^1,c^1)\star (d^2,c^2)$ if and only if
$d^1=d^2$ and $c^1=c^2$;
\item
$A\left((d^0,c^0),(d^1,c^1),(d^2,c^2),(d^3,c^3)\right) \neq \emptyset$
if and only if there exits $\ga^1 \in \Z_2^{d^2d^3\dpe^1}$,$\ga^2 \in \Z_2^{d^3d^1\dpe^2}$ and $\ga^3 \in \Z_2^{d^1d^2\dpe^3}$ such that
\begin{align*}
d^0&=d^1+d^2+d^3\\
c^0&=(d^1+d^2+d^3)_\perp (c^1+c^2+c^3)+\ga^1+\ga^2+\ga^3;
\end{align*}
\item
For any $\ga^1 \in \Z_2^{d^2d^3\dpe^1}$,$\ga^2 \in \Z_2^{d^3d^1\dpe^2}$ and $\ga^3 \in \Z_2^{d^1d^2\dpe^3}$,
\begin{align*}
A&\left((d^1+d^2+d^3, (d^1+d^2+d^3)_\perp (c^1+c^2+c^3)+
\ga^1+\ga^2+\ga^3),(d^1,c^1),(d^2,c^2),(d^3,c^3)\right)\\
&=\{(d^2+d^3, (d^2+d^3)_\perp(c^2+c^3)+\ga^1+\ga^s)
\}_{\ga^s \in \Z_2^{d^1d^2d^3}}
\end{align*}
and
\begin{align*}
A&\left((d^1+d^2+d^3, (d^1+d^2+d^3)_\perp (c^1+c^2+c^3)+
\ga^1+\ga^2+\ga^3),(d^2,c^2),(d^1,c^1),(d^3,c^3)\right)\\
&=
\{(d^1+d^3, (d^1+d^3)_\perp(c^1+c^3)+\ga^2+\ga^u)) 
\}_{\ga^u \in \Z_2^{d^1d^2d^3}}.
\end{align*}
In particular, $\# A\left((d^0,c^0),(d^1,c^1),(d^2,c^2),(d^3,c^3)\right)\leq 1$ if $d^1d^2d^3=0$.
\end{enumerate}
\end{lem}
\begin{proof}
(1), (2) clearly follows from Lemma \ref{fusion_formula}.
By Lemma \ref{fusion_formula},
\begin{align}
&(d^1,c^1)\star \Bigl((d^2,c^2)\star (d^3,c^3)\Bigr)\label{eq_fusion} \\
&=(d^1,c^1)\star (d^2+d^3, (d^2+d^3)_\perp (c^2+c^3) +\Z_2^{d^2d^3})\nonumber \\
&=(d^1+d^2+d^3, (d^1+d^2+d^3)_\perp(c^1+ (d^2+d^3)_\perp (c^2+c^3))+\dpe^1 \Z_2^{d^2d^3} +\Z_2^{d^1(d^2+d^3)}). \nonumber
\end{align}
We note that 
$\dpe^1 \Z_2^{d^2d^3} +\Z_2^{d^1(d^2+d^3)}
=\Z_2^{d^2d^3\dpe^1}+\Z_2^{d^3d^1\dpe^2}+\Z_2^{d^1d^2\dpe^3}$.
Since
$d^2c^2=0$, we have:
\begin{align*}
(d^1+d^2+d^3)_\perp(d^2+d^3)c^2 = (d^1+d^2+d^3)_\perp d^3\dpe^2 c^2=d^1d^3\dpe^2 c^2 \in \Z_2^{d^3d^1\dpe^2}.
\end{align*}
Hence, (3) follows from
$(d^1+d^2+d^3)_\perp(d^2+d^3)_\perp (c^2+c^3)
=(d^1+d^2+d^3)_\perp(c^2+c^3)+(d^1+d^2+d^3)_\perp(d^2+d^3)(c^2+c^3).$
%
%
(4) easily follows from \eqref{eq_fusion}.
\end{proof}

\begin{rem}
More than two intermediate states exist
if and only if there is a component which is equal to $(\fs,\fs,\fs,\fs)$.
Since the d-part corresponds to $\fs$,
all the possible intermediate states are parameterized by $\Z_2^{d^1d^2d^3}$.
We note that
since $d^0=d^1+d^2+d^3$,
$\Z_2^{d^0d^1d^2d^3}=\Z_2^{(d^1+d^2+d^3)d^1d^2d^3}
=\Z_2^{d^1d^2d^3}$ by $d^id^i=d^i$ for $d^i \in \Z_2^{l+r}$.
Thus, the intermediate states are symmetric for $d^0,d^1,d^2,d^3$.
\end{rem}

%
%
If $d^1d^2d^3=0$,
then there is only one intermediate state
and the connection matrix is a scalar.
So we set
$$
B_{(d^0,c^0),(d^1,c^1),(d^2,c^2),(d^3,c^3)}
=
B_{(d^0,c^0),(d^1,c^1),(d^2,c^2),(d^3,c^3)}^{(d^2+d^3,(d^2+d^3)_\perp (c^2+c^3)), (d^1+d^3,(d^1+d^3)_\perp (c^1+c^3))}
$$
for short.

Now, we can state the main result on this section,
which follows from Proposition \ref{connection_one}.
\begin{thm}\label{connection}
Let $(d^1,c^1),(d^2,c^2), (d^3,c^3) \in \Isr$
and $\ga^1 \in \Z_2^{d^2d^3\dpe^1}$,$\ga^2 \in \Z_2^{d^3d^1\dpe^2}$ and $\ga^3 \in \Z_2^{d^1d^2\dpe^3}$.
Set $d^0=d^1+d^2+d^3$ and $c^0=(d^1+d^2+d^3)_\perp(c^1+c^2+c^3)+\ga^1+\ga^2+\ga^3$.
Then, for any $c,c' \in \Z_2^{l+r}$ with
$(d^2+d^3, c) \in A((d^0,c^0),(d^1,c^1),(d^2,c^2),(d^3,c^3))$
and $(d^1+d^3,c') \in A((d^0,c^0),(d^2,c^2),(d^1,c^1),(d^3,c^3))$,
\begin{align*}
B_{(d^0,c^0),(d^1,c^1),(d^2,c^2),(d^3,c^3)}^{(d^2+d^3, c), (d^1+d^3,c')}
&=(-1)^{|c^1c^2|}
(-1)^{|d^1c^2(c^0+c^3)|+|d^2c^1(c^0+c^3)|}
i^{-|d^1c^2|-|d^2c^1|+
|d^1d^2(c^0+c^3)|}\\
& \exp(\frac{-\pi i}{8}|d^1d^2|)
(\frac{1+i}{2})^{|d^1d^2d^3|_l}
(\frac{1-i}{2})^{|d^1d^2d^3|_r}
(-i)^{|d^1d^2d^3(c+c')|}.
\end{align*}
Furthermore,
if $d^1=(0,\dots,0)$, then
 \begin{align*}
B_{(d^0,c^0),(0,c^1),(d^2,c^2),(d^3,c^3)}
=(-1)^{|c^1c^2|}(-i)^{|c^1d^2|}(-1)^{|d^2c^1(c^0+c^3)|}.
\end{align*}
\end{thm}

We end this section by showing that
the multi-index conformal blocks are linearly independent.
We will use the following lemma:
\begin{lem}
\label{lem_independent}
Let $f_1(y),f_2(y),g_1(y),g_2(y) \in \C((y,\bar{y}))$
such that
\begin{enumerate}
\item
$f_1(y)$ and $f_2(y)$ are linearly independent over $\C$;
\item
$f_1(y)g_1(y)+f_2(y)g_2(y)=0$;
\item
$(\pa_y f_1(y))g_1(y)+(\pa_y f_2(y))g_2(y)=0$
and 
$(\pa_{\bar{y}} f_1(y))g_1(y)+(\pa_{\bar{y}}f_2(y))g_2(y)=0$.
\end{enumerate}
Then, $g_1(y)=g_2(y)=0$.
\end{lem}
\begin{proof}
We assume that $g_1(y)\neq 0$.
By (2) and (3),
$\pa_y \left(g_1(y)g_2(y)^{-1}\right)=
\pa_{\bar{y}} \left(g_1(y)g_2(y)^{-1}\right)=0$,
which contradicts (1).
\end{proof}

\begin{prop}
\label{block_independent}
For any $l,r \in \Z_{\geq 0}$ and 
 $\la^0,\la^1,\la^2,\la^3 \in \Isr$,
$\{C_{\la^0,\la^1,\la^2,\la^3}^{\la}(-,-,-,-;x,y)\}_{\la \in A(\la^0,\la^1,\la^2,\la^3)}$ are linearly independent
as vectors in \\
$\mathrm{Hom}_\C\left(
L_{l,r}(\la^0)^\vee \otimes L_{l,r}(\la^1)\otimes L_{l,r}(\la^2)\otimes L_{l,r}(\la^3), 
\C[[x^\R,\x^\R,y^\R,\y^\R]] \right).
$
over $\C$.
\end{prop}
\begin{proof}
It suffices to show that
$\{C_{\la^0,\la^1,\la^2,\la^3}^{\la}(-,-,-,-;1,y)\}_{\la \in A(\la^0,\la^1,\la^2,\la^3)}$ are linearly independent
as vectors in
$\mathrm{Hom}_\C\left(
L_{l,r}(\la^0)^\vee \otimes L_{l,r}(\la^1)\otimes L_{l,r}(\la^2)\otimes L_{l,r}(\la^3), 
\C((y^\fs,\bar{y}^{\fs})) \right)$,
where $\C((y^\fs,\bar{y}^{\fs}))$
is the field of formal Laurent series with the formal valuables $y^\fs$ and $\bar{y}^\fs$.
We may assume that all $\la^i$ are equal to
$\fs^{l,r} = (\fs,\fs,\dots,\fs) \in \Isr$.
In this case, $A(\fs^{l,r},\fs^{l,r},\fs^{l,r},\fs^{l,r})
=\{0,\ft\}^{l+r}$.
We prove the assertion by the induction of $l+r$.
If $l+r=1$, then the assertion clearly follows.
Let $l+r >1$ and 
assume that
there exists $\{k_\la \in \C\}_{\la \in \{0,\ft\}\times \dots \times \{0,\ft\}}$ such that
\begin{align*}
\sum_{\la \in \{0,\ft\}^{l+r}}
k_\la C_{\fs^{l,r},\fs^{l,r},\fs^{l,r},\fs^{l,r}}^{\la}(-,-,-,-;1,y)=0.
\end{align*}
Then,
\begin{align}
0&=
C_{\fs,\fs,\fs,\fs}^{0}(-,-,-,-;1,y)
\left(\sum_{\la' \in \{0,\ft\}^{l+r-1}}
k_{0,\la'} C_{\fs^{l-1,r},\fs^{l-1,r},\fs^{l-1,r},\fs^{l-1,r}}^{\la'}(-,-,-,-;1,y)
\right) \nonumber \\
&+
C_{\fs,\fs,\fs,\fs}^{\ft}(-,-,-,-;1,y)
\left(
\sum_{\la' \in \{0,\ft\}^{l+r-1}}
k_{\ft,\la'} C_{\fs^{l-1,r},\fs^{l-1,r},\fs^{l-1,r},\fs^{l-1,r}}^{\la'}(-,-,-,-;1,y)
\right). \label{eq_fs_split}
\end{align}
By Lemma \ref{lem_independent},
 substituting suitable vectors $\ket{\fs}, L(-1)\ket{\fs} \in L(\ft,\fs)$ into the entries of \eqref{eq_fs_split},
we can show
that both
 $\sum_{\la' \in \{0,\ft\}^{l+r-1}}
k_{0,\la'} C_{\fs^{l-1,r},\fs^{l-1,r},\fs^{l-1,r},\fs^{l-1,r}}^{\la'}(-,-,-,-;1,y)$
and 
$\sum_{\la' \in \{0,\ft\}^{l+r-1}}
k_{\ft,\la'} C_{\fs^{l-1,r},\fs^{l-1,r},\fs^{l-1,r},\fs^{l-1,r}}^{\la'}(-,-,-,-;1,y)$ are equal to $0$.
Thus, by the induction assumption,
all $k_{\la}$ are equal to $0$.
Hence, the assertion holds.
\end{proof}

\section{Framed algebra}\label{sec_framed_algebra}
Motivated by the study in last section,
we introduce a notion of a famed algebra in this section.
In Section \ref{subsec_def_framed}, we define a notion of an $(l,r)$-framed algebra
and show that for any framed algebra we can construct a framed full VOA and vice versa.
In Section \ref{subsec_associative}- \ref{subsec_products}, we develop a general theory of a framed algebra.
In Section \ref{sec_code_vertex}, we study VOAs constructed from trivial framed algebras.

\subsection{Definition of framed algebra}
\label{subsec_def_framed}
Let $l,r \in \Z_{\geq 0}$ and
 $S=\bigoplus_{\la \in \Isr} S_\la$ be a finite-dimensional $\Isr$-graded vector space 
equipped with a linear map $\cdot: S \otimes S \rightarrow S$
and a distinguished non-zero element $1 \in S_0$ such that:
\begin{enumerate}
\item[FA1)]
For any $\la \in \Isr$, $S_\la =0$ unless $s(\la) \in \Z$.
\item[FA2)]
$S_0=\C 1$ and 
for any $a\in S$,
$a\cdot 1=1\cdot a=a$;
\item[FA3)]
For any $\la^1,\la^2 \in \Isr $, $a_1 \in S_{\la^1}$ and $a_2 \in S_{\la^2}$,
$a_1 \cdot a_2 \in \bigoplus_{\la \in \la^1 \star \la^2}S_{\la}.$
\end{enumerate}
Recall that for $\la \in \Isr$
$\pi_\la:S\rightarrow S$ is the projection of the graded vector space $S$ onto $S_\la$
and the composition of the product $\cdot$ and the projection $\pi_\la$ is denoted 
by $\cdot_\la : S\otimes S \rightarrow S_\la$.
By Proposition \ref{prevertex},
$F_S= \bigoplus_{\la \in \Isr} L_{r,s}(\la) \otimes S_\la$ is a full prevertex algebra.
\begin{prop}
\label{vertex}
The following conditions are equivalent:
\begin{enumerate}
\item
The full prevertex algebra $F_S$ is a full vertex algebra;
\item
For any $\la^i \in \Isr$, $a_i \in S_{\la^i}$  (i=1,2,3)
and $\la' \in A(\la^0,\la^2,\la^1,\la^3)$,
$$
a_2\cdot_{\la^0} ( a_1 \cdot_{\la'} a_3 )
= \sum_{\la \in A(\la^0,\la^1,\la^2,\la^3)}
B_{\la^0,\la^1,\la^2,\la^3}^{\la,\la'}
a_1\cdot_{\la^0}(a_2 \cdot_\la a_3).
$$
\end{enumerate}
\end{prop}
\begin{proof}
Let $\la^0,\la^1,\la^2,\la^3 \in \Isr$
and $u_i \otimes a_i \in L_{l,r}(\la^i) \otimes S_{\la^i}$
and $u_0^*\otimes a_0^* \in L_{l,r}(\la^0) \otimes (S_{\la^0})^*$.
First, we assume that $F_S$ is a full vertex algebra.
Then, by (FV5), there exists $\mu(z_1,z_2) \in \GCor_2$ such that
\begin{align*}
\langle u_0^* \otimes a_0^*,
Y(u_1\otimes a_1,\uz_1)Y(u_2\otimes a_2,\uz_2)u_3\otimes a_3 \rangle &=\mu(z_1,z_2)|_{|z_1|>|z_2|}\\
\langle u_0^* \otimes a_0^*,
Y(u_2\otimes a_2,\uz_2)Y(u_1\otimes a_1,\uz_1)u_3\otimes a_3 \rangle &=\mu(z_1,z_2)|_{|z_2|>|z_1|}.
\end{align*}
Let us consider the path $\ga_0$ in Section \ref{sec_conn}.
By \eqref{eq_correlator} and Proposition \ref{multi_monodromy},
\begin{align*}
&\si^* A_{\ga_0} \left( \langle u_0^* \otimes a_0^*,
Y(u_1\otimes a_1,\uz_1)Y(u_2\otimes a_2,\uz_2)u_3\otimes a_3 \rangle \right)\\
&=\si^* A_{\ga_0} \left(  \sum_{\la \in A(\la^0,\la^1,\la^2,\la^3)}
C_{\la^0, \la^1,\la^2,\la^3}^{\la}(u_0,u_1,u_2,u_3; z_1,z_2)
\langle a_0^\vee, a_1 \cdot_{\la_0} (a_2\cdot_\la a_3)\rangle \right).\\
&=  \sum_{\la \in A(\la^0,\la^1,\la^2,\la^3)}\langle a_0^\vee, a_1 \cdot_{\la_0} (a_2\cdot_\la a_3)\rangle
 \si^* A_{\ga_0} \left( C_{\la^0, \la^1,\la^2,\la^3}^{\la}(u_0,u_1,u_2,u_3; z_1,z_2)  \right) \tag{$\clubsuit$} \label{eq_club} \\
&=  \sum_{\la \in A(\la^0,\la^1,\la^2,\la^3)}\langle a_0^\vee, a_1 \cdot_{\la_0} (a_2\cdot_\la a_3)\rangle
\left( \sum_{\la' \in A(\la^0,\la^2,\la^1,\la^3)} B_{\la^0, \la^1,\la^2,\la^3}^{\la,\la'}
C_{\la^0, \la^2,\la^1,\la^3}^{\la'}(u_0,u_2,u_1,u_3; z_1,z_2)  \right)\\
&=
\sum_{\la' \in A(\la^0,\la^2,\la^1,\la^3)} 
C_{\la^0, \la^2,\la^1,\la^3}^{\la'}(u_0,u_2,u_1,u_3; z_1,z_2) 
\left(   \sum_{\la \in A(\la^0,\la^1,\la^2,\la^3)} 
B_{\la^0, \la^1,\la^2,\la^3}^{\la,\la'}
\langle a_0^\vee, a_1 \cdot_{\la_0} (a_2\cdot_\la a_3)\rangle
 \right).
\end{align*}
Since $\langle u_0^* \otimes a_0^*,
Y(u_1\otimes a_1,\uz_1)Y(u_2\otimes a_2,\uz_2)u_3\otimes a_3 \rangle$
is absolutely convergent to the single-valued real analytic function $\mu(z_1,z_2)$,
we have
\begin{align*}
\si^* A_{\ga_0}&\left(\langle u_0^* \otimes a_0^*,
Y(u_1\otimes a_1,\uz_1)Y(u_2\otimes a_2,\uz_2)u_3\otimes a_3 \rangle \right)\\
&= \mu(z_2,z_1)|_{U_+}\\
&= \langle u_0^* \otimes a_0^*,
Y(u_2\otimes a_2,\uz_1)Y(u_1\otimes a_1,\uz_2)u_3\otimes a_3 \rangle|_{U_+}.
\end{align*}
Combining this with \eqref{eq_club},
we have:
\begin{align*}
&\sum_{\la' \in A(\la^0,\la^2,\la^1,\la^3)}
C_{\la^0, \la^2,\la^1,\la^3}^{\la'}(u_0,u_2,u_1,u_3; z_1,z_2)
\langle a_0^\vee, a_2 \cdot_{\la_0} (a_1\cdot_{\la'} a_3)\rangle\\
&=\sum_{\la' \in A(\la^0,\la^2,\la^1,\la^3)} 
C_{\la^0, \la^2,\la^1,\la^3}^{\la'}(u_0,u_2,u_1,u_3; z_1,z_2) 
\left(   \sum_{\la \in A(\la^0,\la^1,\la^2,\la^3)} 
B_{\la^0, \la^1,\la^2,\la^3}^{\la,\la'}
\langle a_0^\vee, a_1 \cdot_{\la_0} (a_2\cdot_\la a_3)\rangle
 \right).
\end{align*}
By Proposition \ref{block_independent}, we must have
\begin{align*}
\langle a_0^\vee, a_2 \cdot_{\la_0} (a_1\cdot_{\la'} a_3)\rangle
= \sum_{\la \in A(\la^0,\la^1,\la^2,\la^3)} 
B_{\la^0, \la^1,\la^2,\la^3}^{\la,\la'}
\langle a_0^\vee, a_1 \cdot_{\la_0} (a_2\cdot_\la a_3)\rangle
\end{align*}
for any $\la'\in A(\la^0, \la^2,\la^1,\la^3)$, that is, (2) holds.

Conversely, assume that (2) holds. Set $F(z_1,z_2)=\langle u_0^* \otimes a_0^*, Y(u_1\otimes a_1,\uz_1)Y(u_2\otimes a_2,\uz_2)u_3\otimes a_3 \rangle$.
By Proposition \ref{multi_monodromy}, 
the formal power series $F(z_1,z_2)$ is absolutely convergent to a (possibly) multi-valued real-analytic function on $Y_2(\C^\times)$.
We will first show that this function is in fact single-valued.
For the sake of simplicity,
we assume that $u_i=\ket{\la^i}$ and $u_0 =\bra{\la^0}$,
that is, the lowest weight vectors.
By the similar argument in Lemma \ref{block_limit}, there exists a formal power series $G(z) \in \C[[z^\R,\z^\R]]$ such that 
$F(z_1,z_2)=z_1^{l(\la^0)-l(\la^1)-l(\la^2)-l(\la^3)}\z_1^{r(\la^0)-r(\la^1)-r(\la^2)-r(\la^3)}
G(z_2/z_1)$.
By (FA1), $s(\la^i)=l(\la^i)-r(\la^i) \in \Z$ for all $i=0,1,2,3$.
Thus, $z_1^{l(\la^0)-l(\la^1)-l(\la^2)-l(\la^3)}\z_1^{r(\la^0)-r(\la^1)-r(\la^2)-r(\la^3)}
=(z_1\z_1)^{r(\la^0)-r(\la^1)-r(\la^2)-r(\la^3)} z_1^{s(\la^0)-s(\la^1)-s(\la^2)-s(\la^3)}$
 is a single-valued real analytic function on $Y_2(\C^\times)$.

Since the image of the map $Y_2(\C^\times) \rightarrow \CP,\; (z_1,z_2) \mapsto z_2/z_1$
is $\CPm$, it suffices to show that the formal power series $G(z)$ is single-valued on $\CPm$.
We note that the fundamental group $\pi_1(\CPm)$ is generated by two cycles.
One is a cycle around $0$ 
and the other is a cycle around $\infty$.

Since by Lemma \ref{block_limit}
\begin{align}
G(z)&= \sum_{\la \in A(\la^0,\la^1,\la^2,\la^3)}\langle a_0^\vee, a_1 \cdot_{\la_0} (a_2\cdot_\la a_3)\rangle
C_{\la^0, \la^1,\la^2,\la^3}^{\la}(z)\label{eq_cycle} \\
&= \sum_{\la \in A(\la^0,\la^1,\la^2,\la^3)} z^{l(\la)-l(\la^2)-l(\la^3)}\z^{r(\la)-r(\la^2)-r(\la^3)}
\left(\langle a_0^\vee, a_1 \cdot_{\la_0} (a_2\cdot_\la a_3)\rangle+O(z,\z)\right), \nonumber
\end{align}
the monodromy of $G(z)$ around $z=0$ is trivial by (FA1).
In particular, $G(z)\in \C((z,\z,|z|^\R))$.

Next we consider the monodromy of $G(z)$ around $z=\infty$.
More precisely, let $p_\infty$ be a small cycle around $\infty$
and we will consider the analytic continuation of $G(z)$
along the path $\ga_0^{-1}\circ p_\infty \circ \ga_0$.
By \eqref{eq_club} and the assumption (2),
\begin{align}
&\si^* A_{\ga_0} \left( \langle u_0^* \otimes a_0^*,
Y(u_1\otimes a_1,\uz_1)Y(u_2\otimes a_2,\uz_2)u_3\otimes a_3 \rangle \right) \label{eq_cross} \\
&=\sum_{\la' \in A(\la^0,\la^2,\la^1,\la^3)} 
C_{\la^0, \la^2,\la^1,\la^3}^{\la'}(u_0,u_2,u_1,u_3; z_1,z_2) 
\left(   \sum_{\la \in A(\la^0,\la^1,\la^2,\la^3)} 
B_{\la^0, \la^1,\la^2,\la^3}^{\la,\la'}
\langle a_0^\vee, a_1 \cdot_{\la_0} (a_2\cdot_\la a_3)\rangle
 \right)  \nonumber \\
&=\sum_{\la' \in A(\la^0,\la^2,\la^1,\la^3)} 
C_{\la^0, \la^2,\la^1,\la^3}^{\la'}(u_0,u_2,u_1,u_3; z_1,z_2) 
\langle a_0^\vee, a_2 \cdot_{\la_0} (a_1 \cdot_{\la'} a_3)\rangle  \nonumber  \\
&=\langle u_0^* \otimes a_0^*,
Y(u_2\otimes a_2,\uz_1)Y(u_1\otimes a_1,\uz_2)u_3\otimes a_3 \rangle. \nonumber 
\end{align}
Thus, similarly to \eqref{eq_cycle}, by (FA1) again the monodromy of $G(z)$
along the path $\ga_0^{-1}\circ p_\infty \circ \ga_0$ is trivial.
Denote the analytic continuation of $G(z)$ by $g(z)$ which is a single-valued real analytic function
on $\CPm$. Since multi-index Virasoro conformal blocks have expansions around $0,1,\infty$,
$g(z)$ has conformal singularities at $0,1,\infty$. Thus, $g(z) \in \F$,
which implies that there exists $\mu(z_1,z_2)\in \GCor_2$ such that
\begin{align*}
\langle u_0^* \otimes a_0^*,
Y(u_1\otimes a_1,\uz_1)Y(u_2\otimes a_2,\uz_2)u_3\otimes a_3 \rangle &=\mu(z_1,z_2)|_{|z_1|>|z_2|}.
\end{align*}
Furthermore, by \eqref{eq_cross},
$\langle u_0^* \otimes a_0^*,Y(u_2\otimes a_2,\uz_2)Y(u_1\otimes a_1,\uz_1)u_3\otimes a_3 \rangle$
is also the expansion of the same function $\mu(z_1,z_2)$, that is,
\begin{align*}
\langle u_0^* \otimes a_0^*,Y(u_2\otimes a_2,\uz_2)Y(u_1\otimes a_1,\uz_1)u_3\otimes a_3 \rangle
&=  \mu(z_1,z_2)|_{|z_2|>|z_1|}
\end{align*}
Thus, by Proposition \ref{locality} and Proposition \ref{prevertex},  (1) holds.
\end{proof}

Now, we introduce a notion of an $(l,r)$-framed algebra.
An $(l,r)$-framed algebra (or framed algebra for short) is a finite-dimensional $\Isr$-graded vector space $S=\bigoplus_{\la \in \Isr} S_\la$
equipped with a linear map $\cdot: S \otimes S \rightarrow S$
and a distinguished non-zero element $1 \in S_0$ which satisfy
(FA1), (FA2) and (FA3) together with the following condition:
\begin{enumerate}
\item[FA4)]
For any $\la^i \in \Isr$, $a_i \in S_{\la^i}$
and $\la' \in A(\la^0,\la^1,\la^2,\la^3)$ (i=1,2,3),
$$
a_2\cdot_{\la_0} ( a_1 \cdot_{\la'} a_3 )
= \sum_{\la \in A(\la^0,\la^1,\la^2,\la^3)}
B_{\la^0,\la^1,\la^2,\la^3}^{\la,\la'}
a_1\cdot_{\la_0}(a_2 \cdot_\la a_3).
$$
\end{enumerate}

\begin{rem}
The framed algebra is nothing but an explicit definition of a commutative algebra object
in the braided tensor category $(\mathrm{Rep} \Vir_{\ft})^{\otimes l} \otimes (\overline{\mathrm{Rep} \Vir_{\ft}})^{\otimes r}$.
\end{rem}

Now, by Proposition \ref{prevertex}, Proposition \ref{vertex} and Lemma \ref{preframed},
we have:
\begin{thm}\label{correspondence}
If $(S,\cdot,1)$ is a framed algebra,
then $(F_S,Y, 1)$ in Proposition \ref{prevertex} is a framed full vertex operator algebra.
Conversely, if $(F,Y, \1)$ is a framed full vertex operator algebra,
then $(S_F,\cdot,\1)$ in Lemma \ref{preframed} is a framed algebra.
\end{thm}

The purpose of this paper is to systematically construct examples of framed algebras.
We note that we can define a category of framed algebra and
prove that the category of framed algebras and 
the category of framed full vertex operator algebras are equivalent.
But, we do not need this result for the main purpose of this paper.
 The interested reader may consult Appendix of this paper.

\subsection{Associativity, ideal and bilinear form} \label{subsec_associative}
The product of a framed algebra is commutative up to the connection matrix.
In this section, we show that a framed algebra is ``associative'' and there always exists an invariant bilinear form.

The following lemma says that a framed algebra is associative:
\begin{lem}\label{associativity}
Let $S$ be a framed algebra.
For $\la^i \in \Isr$ and $a_i \in S_{\la^i}$,
the following conditions hold:
\begin{enumerate}
\item
$a_1 \cdot_{\la^0} a_2 = (-1)^{s(\mu)+s(\la^1)+s(\la^2)} a_2 \cdot_{\la^0} a_1$ for $\la^0 \in \la^1 \star \la^2$.
\item
For $\mu' \in A(\la^0,\la^1,\la^2,\la^3)$ and $\la^0 \in \Isr$,
$$a_1 \cdot_{\la^0} (a_2 \cdot_{\mu'} a_3)
= (-1)^{s(\mu')+s(\mu)+s(\la^2)+s(\la^0)}\sum_{\mu \in A(\la^0,\la^3,\la^1,\la^2)} B_{\la^0,\la^3,\la^1,\la^2}^{\mu,\mu'}
(a_1 \cdot_\mu a_2) \cdot_{\la_0} a_3.$$
\end{enumerate}
\end{lem}
\begin{proof}
Applying $a_3=1$ to the definition (4) of a framed algebra
and by Lemma \ref{connection_vacuum},
we have 
\begin{align*}
a_2 \cdot_{\la^0} a_1
&= a_2 \cdot_{\la^0} (a_1 \cdot_{\la^1} 1)\\
&=\sum_{\mu \in A(\la^0,\la^1,\la^2,0)}
B_{\la^0,\la^1,\la^2, 0}^{(\mu,\la^1)} <a_0^\vee, a_1 \cdot_{\la^0} a_2 \cdot_{\mu} 1>\\
&=B_{\la^0,\la^1,\la^2, 0}^{(\la^2,\la^1)} a_1 \cdot_{\la^0} a_2\\
&=\exp(\pi i(l(\la^0)-l(\la^1)-l(\la^2)))\exp(-\pi i (r(\la^0)-r(\la^1)-r(\la^2)))a_1 \cdot_{\la^0} a_2\\
&=(-1)^{s(\la^0)+s(\la^1)+s(\la^2)} a_1 \cdot_{\la^0} a_2,
\end{align*}
where we used $s(\la^i) \in \Z$.
Thus,
\begin{align*}
a_1& \cdot_{\la^0} (a_2 \cdot_{\mu'} a_3) \\
&= (-1)^{s(\mu')+s(\la^2)+s(\la^3)} a_1 \cdot_{\la^0} (a_3 \cdot_{\mu'} a_2)\\
&= (-1)^{s(\mu')+s(\la^2)+s(\la^3)}\sum_{\mu \in A(\la^0,\la^3,\la^1,\la^2)}
B_{\la^0,\la^3,\la^1,\la^2}^{\mu,\mu'}
a_3 \cdot_{\la_0}(a_1 \cdot_\mu a_2)\\
&= (-1)^{s(\mu')+s(\la^2)+s(\la^3)}\sum_{\mu \in A(\la^0,\la^3,\la^1,\la^2)}
B_{\la^0,\la^3,\la^1,\la^2}^{\mu,\mu'}
(-1)^{s(\la^0) +s(\la^3) + s(\mu)}
(a_1 \cdot_\mu a_2) \cdot_{\la_0} a_3\\
&= (-1)^{s(\mu')+s(\mu)+s(\la^2)+s(\la^0)}\sum_{\mu \in A(\la^0,\la^3,\la^1,\la^2)} B_{\la^0,\la^3,\la^1,\la^2}^{\mu,\mu'}
(a_1 \cdot_\mu a_2) \cdot_{\la_0} a_3.
\end{align*}
%
\end{proof}

An ideal of a framed algebra $S$ is a $\Isr$-graded subspace $M \subset S$ such that $a \cdot m \in M$ for any $a\in S$ and $m\in M$.
By Lemma \ref{associativity}, any ideal of $S$
is automatically a two-sided ideal, that is,
$m \cdot a \in M$ for any $a\in S$ and $m\in M$.
A framed algebra is said to be simple
if it does not contain a proper ideal.

Then, we have:
\begin{prop}
\label{simple}
Let $S$ be a framed algebra. 
Then, $S$ is simple if and only if $(F_S,Y,\1)$ is a simple full vertex algebra.
\end{prop}
\begin{proof}
Let $S$ be a framed algebra.
Suppose that $I$ is an ideal of $F_S$.
Since $I$ is decomposed into simple $L_{l,r}(0)$-modules,
there exists $\Isr$-graded subspace $M=\bigoplus_{\la \in \Isr} M_\la \subset \bigoplus_{\la \in \Isr} S_\la$
such that $I= \bigoplus_{\la \in \Isr} L_{l,r}(\la)\otimes M_\la$.
By definition of $F_S$, $M$ is an ideal of the framed algebra $S$.
The reverse direction can be obtained similarly.
\end{proof}

We will use the following lemmas later:
By the fusion rule, $0 \in \la^1 \star \la^2$ if and only if
$ \la^1=\la^2$. Hence, we have:
\begin{lem} \label{vacuum_product}
Let $\la^i \in \Isr$ and $a_i\in S_{\la^i}$ for $i=1,2$.
If $a_1 \cdot_0 a_2 \neq 0$,
then $\la^1 = \la^2$.
\end{lem}

\begin{lem}\label{non-zero}
Let $S$ be a simple framed algebra
and $a_i \in S_{\la^i}$ for $\la^i \in \Isr$ ($i=1,2$).
If $a_1 \cdot a_2=0$,
then $a_1=0$ or $a_2=0$.
\end{lem}
\begin{proof}
We may assume that $a_2 \neq 0$.
Then,
by Lemma \ref{associativity}, $Sa_2 \equiv \{b \cdot a_2 \}_{b \in S}$ is an ideal of $S$.
Since $S$ is simple and $a_2 \in Sa_2$,
$Sa_2= S$.
Thus, there exists $a' \in S_{\la^2}$ such that $a' \cdot_0 a_2
=1$ by Lemma \ref{vacuum_product} and (FA2).
Then, by Lemma \ref{associativity}, 
$a_2 \cdot_0 a' = (-1)^{2s(\la^2)} a' \cdot_0 a_2 =1$
and $a_1=a_1\cdot (a_2 \cdot_0 a')=
\sum_{\mu} \text{Coefficient} \times  (a_1 \cdot_{\mu} a_2) \cdot a' =0$.
Hence, $a_1=0$.
\end{proof}

Let $S$ be a framed algebra.
Define the bilinear form $(-,-): S\otimes S \rightarrow \C$
by $(a,b)1 = (-1)^{s(\la)} a \cdot_{0} b$ for any $a \in S_\la$ and $b \in S_{\la'}$,
where we use (FA2).
\begin{prop}
\label{bilinear}
The bilinear form $(-,-):S\otimes S \rightarrow \C$ satisfies 
the following conditions:
for any $a_i \in S_{\la^i}$ and $\la^i \in \Isr$ ($i=1,2,3$),
\begin{enumerate}
\item
$S_{\la^1}$ and $S_{\la^2}$ are orthogonal
to each other for $\la^1 \neq \la^2$;
\item
The bilinear form is symmetric;
\item
The bilinear form is invariant, that is,
$(a_1\cdot a_2, a_3)=(-1)^{s(\la^1)}(a_2, a_1 \cdot a_3)$.
\item
The bilinear form is non-degenerate
if and only if $S$ is simple.
\end{enumerate}
\end{prop}
\begin{proof}
(1) clearly follows from Lemma \ref{vacuum_product}.
By Lemma \ref{associativity},
$(a_1,a_2)1= (-1)^{s(\la^1)} a_1\cdot_0 a_2 = (-1)^{s(\la^2)}a_2 \cdot_0 a_1= (a_2,a_1)1$.
Since
\begin{align*}
(a_1\cdot a_2, a_3)1
&=(a_1\cdot_{\la^3} a_2, a_3)1= (a_3,a_1\cdot_{\la^3} a_2)1\\
&= (-1)^{s(\la^1)+s(\la^2)+s(\la^3)} (a_3,a_2\cdot_{\la^3} a_1)1\\
&= (-1)^{s(\la^1)+s(\la^2)} a_3 \cdot_0 (a_2\cdot_{\la^3} a_1)\\
&= (-1)^{s(\la^1)+s(\la^2)} B_{0,\la^2,\la^3,\la^1}a_2 \cdot_0 (a_3\cdot_{\la^2} a_1)\\
&= (-1)^{s(\la^3)+s(\la^2)}(a_2,a_3\cdot_{\la^2} a_1)\\
&= (-1)^{s(\la^1)}(a_2,a_1\cdot a_3),
\end{align*}
(3) holds.
Let $R_{(-,-)}$ be the radical of the bilinear form,
i.e., 
$$R_{(-,-)}=\ker(S \rightarrow S^*, a \mapsto (a,-)).$$
Let $m \in R_{(-,-)}$ and $a \in S_\la$.
Since $(am,-)=(-1)^{s(\la)}(m,a-)=0$, we have $am \in R_{(-,-)}$.
By (1), $R_{(-,-)}$ is an $\Isr$-graded subspace of $S$.
Hence, $R_{(-,-)}$ is an ideal of $S$
and $1 \notin R_{(-,-)}$.
Thus, $R_{(-,-)}=0$ if $S$ is simple.
Conversely, we assume that $R_{(-,-)}=0$
and $M \subset S$ is an ideal with $1 \notin M$.
Since $M$ is an $\Isr$-graded subspace,
$M \cap S_0= 0$.
Thus, for any $a \in S$ and $m \in M$,
$a \cdot_0 m=0$. Hence, $M\subset R_{(-,-)}=0$,
which implies that $S$ is simple.
\end{proof}
Hereafter, we assume that a framed algebra is simple.


\subsection{Code subalgebra}
Let $S$ be a simple $(l,r)$-framed algebra.
In this section, we introduce some subgroups of $\Z_2^{l+r}$, which plays an important role in the study of framed algebras. Such subgroups are called codes
and were originally introduce in the study of framed vertex operator algebras \cite{DGH}.

Recall that the set $\Isr$ has the following decomposition:
$$\Isr = \coprod_{d \in \Z_2^{r+s}} \{d\} \times \Z_2^d.$$
Accordingly, the framed algebra $S$ also has the following decomposition:
\begin{align*}
S&=\bigoplus_{d \in \Z_2^{l+r}}A_S(d)\\
A_S(d)&= \bigoplus_{c \in \Z_2^{l+r}} S_{d,c}. 
\end{align*}

Set
\begin{align*}
C_S&=\{\al \in \Z_2^{l+r}\;|\; S_{0,\al} \neq 0 \},\\
D_S&=\{d \in \Z_2^{l+r}\;|\;A_S(d)\neq 0\},\\
I_S&= \{(d,c) \in (\Z_2^{l+r})^2\;|\;S_{d,c}\neq 0 \}.
\end{align*}

Then,
by Lemma \ref{fusion_formula},
\begin{align}
S_{d,c}\cdot S_{d',c'}\subset \bigoplus_{\ga \in \Z_2^{dd'}} S_{d+d',(d+d')_\perp (c+c') + \ga}. \label{eq_fusion_decomposition}
\end{align}
In particular, $A_S(d)\cdot A_S(d') \subset A_S(d+d')$ for any $d,d' \in D_S$.
Thus, by Lemma \ref{non-zero},
$D_S$ is a subgroup of $\Z_2^{l+r}$ and $S$ is a $D_S$-graded algebra.
Similarly, since $S_{0,\al}\cdot S_{0,\be} \subset S_{0,\al +\be}$ for any $\al,\be \in C_S$,
 $C_S$ is a subgroup of $\Z_2^{l+r}$.
Denote $S_{0,\al}$ simply by $S_{\al}$ for $\al \in C_S$.
Then, $A_S(0)=\bigoplus_{\al \in C_S} S_{\al}$.
For $v \in S_{d,c}$ and $v' \in S_{d',c'}$,
by \eqref{eq_fusion_decomposition}, 
$$v \cdot v' = \sum_{\ga \in \Z_2^{dd'}}v \cdot_{(d+d',(d+d')_\perp(c+c')+\ga)}v'.$$
We sometimes denote $v \cdot_{(d+d',(d+d')_\perp(c+c')+\ga)}v'$ by
$v \cdot_{(d+d')_\perp(c+c')+\ga}v'$ for short.

\begin{rem}
\label{rem_simple}
Let $\al \in C_S$, $(d,c) \in I_S$ and $a \in S_\al$, $v\in S_{d,c}$.
It is noteworthy that $a\cdot v$ consists of only one component,
that is,
\begin{align*}
a\cdot v = \sum_{\ga \in \Z_2^0} a \cdot_{\dpe \al+c+\ga} v= a\cdot_{\dpe c+\al} v.
\end{align*}
\end{rem}

\begin{prop}
\label{even_code}
The subsets $C_S,D_S \subset \Z_2^{l+r}$ satisfy the following conditions:
\begin{enumerate}
\item
$C_S$ and $D_S$ are subgroups of $\Z_2^{l+r}$;
\item
$|\al| \in 2\Z$ and $|d| \in 8\Z$ for any $\al \in C_S$ and $d\in D_S$;
\item
$|\al d| \in 2\Z$ for any $\al \in C_S$ and $d\in D_S$.
\end{enumerate}
\end{prop}
\begin{proof}
(1) was already shown. 
If $(d,c) \in I_S$, then by (FA1) $s(d,c)=\frac{1}{16}|d|+\frac{1}{2}|c| \in \Z$. Thus, (2) follows.
If $(0,\al), (d,c) \in I_S$,
then by Lemma \ref{fusion_formula} and Lemma \ref{non-zero}, $(d,\dpe \al+c) \in I_S$.
Thus, $\frac{1}{16}|d| +\frac{1}{2}|\dpe \al+c| \in \Z$.
Combining this with $s(d,c)\in \Z$,
we have
$|\dpe \al+c| -|c| \in 2\Z$.
Since $|\dpe \al+c| -|c|= |\dpe \al| + 2|\dpe c \al|$,
we have $|\dpe \al| \in 2\Z$.
Since $|d\al|=|\reg \al|- |\dpe \al|= |\al|- |\dpe \la|$
and $|\al| \in 2\Z$, $|d\al|\in 2\Z$.
\end{proof}
Thus, $C_S$ and $D_S$ are linear codes
and $D_S$ is a subcode of the dual code $C_S^\perp = \{\be \in \Z_2^{l+r}\;|\;
(\be,C_S) \subset 2\Z \}$.

%
The purpose of this section is to prove 
that $A_S(0)$ is isomorphic to a twisted group algebra
$\C[\hat{C_S}]$ and $A_S(d)$ is an irreducible $\Isr$-graded $\C[\hat{C_S}]$-module.

The following lemma is very important to study a framed algebra:
\begin{lem}
\label{simple_current}
Let $\al_1,\al_2 \in C_S$ and $(d,c) \in I_S$
and $a_1 \in S_{\al_1}$, $a_2\in S_{\al_2}$ and $v \in S_{d,c}$.
Then,
$a_2 \cdot (a_1 \cdot v)=(-1)^{|\al_1\al_2|} a_1\cdot (a_2\cdot v)$ and $a_1 \cdot (a_2 \cdot v) =(a_1 \cdot a_2) \cdot v$.
In particular, $A_S(0)$ is an associative algebra
and $S$ is an $A_S(0)$-module.
\end{lem}
\begin{proof}
By (FA4), Remark \ref{rem_simple} and Theorem \ref{connection},
$a_2 \cdot (a_1 \cdot v)=a_2 \cdot_{d,\dpe (\al_1+\al_2) +c} (a_1 \cdot_{d,\dpe \al_1 +c} v)=
(-1)^{|\al_1\al_2|} a_1\cdot (a_2\cdot v)$ holds.
Applying Lemma \ref{associativity},
we have
\begin{align}
a_1 \cdot (a_2 \cdot v)
&= (-1)^{s(d,\dpe(\al_1+\al_2)+c)-s(d,\dpe\al_2+c)+s(\al_2)
+s(\al_1+\al_2)} 
B_{(d,\dpe(\al_1+\al_2)+c),(d,c),\al_1,\al_2}
 (a_1 \cdot a_2)\cdot v \label{eq_simple_current} \\
&= (-1)^{\frac{1}{2}\Bigl(
|\dpe(\al_1+\al_2)+c|-|\dpe\al_2+c|+|\al_2|
+|\al_1+\al_2|\Bigr)}
(-1)^{|\al_1c|} (-1)^{\frac{1}{2}|d\al_1|}(-1)^{|d\al_1(\al_2+\dpe(\al_1+\al_2)+c)|} (a_1 \cdot a_2)\cdot v. \nonumber
\end{align}
By using Remark \ref{rem_formula}, we have
$(-1)^{\frac{1}{2}\Bigl(
|\dpe(\al_1+\al_2)+c|-|\dpe\al_2+c|\Bigr)}
=(-1)^{\frac{1}{2}|\dpe \al_1|}(-1)^{|\dpe \al_1(\dpe \al_2 +c)|}$
and $(-1)^{\frac{1}{2}\Bigl(|\al_1+\al_2|+|\al_2|\Bigr)}
=(-1)^{\frac{1}{2}|\al_1|}(-1)^{|\al_1\al_2|}$.
Since $d \dpe=0$ and $dc=0$, we have $(-1)^{|d\al_1(\al_2+\dpe(\al_1+\al_2)+c)|}=(-1)^{|d\al_1\al_2|}$.
Hence, the RHS of \eqref{eq_simple_current} is equal to $(-1)^{\frac{1}{2}(|\al_1|+|\dpe \al_1|+|d\al_1|)}
(-1)^{|\al_1\al_2|+|\dpe \al_1(\dpe \al_2 +c)|+|\al_1c|+|d\al_1\al_2|}
(a_1 \cdot a_2)\cdot v$.
Since $|\dpe \al_1|+|d\al_1|=|\al_1|$, by Proposition \ref{even_code} $(-1)^{\frac{1}{2}(|\al_1|+|\dpe \al_1|+|d\al_1|)}=(-1)^{|\al_1|}=1$.
Since $\dpe c=c$, 
$(-1)^{|\dpe \al_1(\dpe \al_2 +c)|}=(-1)^{|\al_1c|+|\dpe \al_1\al_2|}.$
Thus,
$(-1)^{|\al_1\al_2|+|\dpe \al_1(\dpe \al_2 +c)|+|\al_1c|+|d\al_1\al_2|}
=(-1)^{2|\al_1\al_2|+2|\al_1 c|}=1$. Hence, the assertion holds.
\end{proof}

\begin{lem}
For any $\al \in C_S$, $\dim S_{\al}=1$.
\end{lem}
\begin{proof}
Let $0\neq a \in S_\al$ for $\al \in C_S$.
By Lemma \ref{non-zero}, $0\neq a\cdot a \in S_{2\al}=S_{0}$.
By (FA2), we may assume that $a \cdot a =1$.
Then, for any $a' \in S_{\al}$,
by Lemma \ref{simple_current},
$a'=a'\cdot 1= a' \cdot (a\cdot a)= (-1)^{|\al|} a \cdot (a' \cdot a)   \in \C a$. Thus, $\dim S_{\la}=1$.
\end{proof}

Thus,
$A_S(0)=\bigoplus_{\al \in C_S} S_\al$
is isomorphic to the group algebra $\C[C_S]$
as a vector space.

Let $\varepsilon:C_S \rightarrow \Z_2$ be a bilinear map
satisfying 
\begin{enumerate}
\item
$\varepsilon(\al,\al)=(-1)^{\frac{|\al|}{2}}$;
\item
$\varepsilon(\al,\be)\varepsilon(\be,\al)=(-1)^{|\al\be|}$.
\end{enumerate}
By \cite{FLM}, such a bilinear map exists
and gives an element in $Z^2(C_S,\C^\times)$, the group cohomology with the coefficient in $\C^\times$ (see for example \cite{M1}).
Let $\C[\hat{C_S}]$ be the vector space 
with the basis $\{e_\al \}_{\al \in C_S}$
and
define a product on $\C[\hat{C_S}]$
by
$e_\al e_\be = \varepsilon(\al,\be)e_{\al+\be}$.
Since $\varepsilon(-,-)$ is a two-cocycle, the 
product is associative.
Since $\C^\times$ is an injective object in the category of abelian groups, similarly to the proof of Proposition in \cite{M1}, 
we have:
\begin{prop}
If a two-cocycle $\mu(-,-) \in Z^2(C_S, \C^\times)$
satisfies 
$\mu(\al,\be)\mu(\be,\al)=(-1)^{|\al \be|}$ and $\mu(\al,\al)=(-1)^{\frac{|\al|}{2}}$
 for any $\al,\be \in C_S$.
Then, there exists a map $f:C_S \rightarrow \C^\times$
such that $\mu(\al,\be)f(\al+\be)f(\al)^{-1}f(\be)^{-1}=\varepsilon(\al,\be)$ for any $\al,\be \in C_S$,
that is, $\mu(-,-)=\varepsilon(-,-) \in H^2(C_S,\C^\times)$.
\end{prop}

Then, by Lemma \ref{non-zero},
Lemma \ref{simple_current} and the above proposition,
we have:
\begin{prop}
There exists a $\C$-algebra isomorphism between
$A_S(0)$ and $\C[\hat{C_S}]$ which preserves the $\Isr$-grading.
\end{prop}

\begin{rem}
\label{rem_invariance}
For $\al \in C_S$,
since $e_\al \cdot e_\al= \epsilon(\al,\al) e_0=(-1)^{|\al|}e_0$, we have $e_\al^{-1}=(-1)^{s(\al)} e_\al$.
Thus, by Proposition \ref{bilinear}, the invariant bilinear form on a framed algebra $S$
is $\C[\hat{C_S}]$-invariant, i.e.,
$
(e_\al\cdot v_1,e_\al\cdot v_2)=(v_1,v_2)
$
for any $\al \in C_S$ and $v_1,v_2 \in S$.
\end{rem}

Hereafter, we identify $A_S(0)$ as $\C[\hat{C_S}]$
and take a basis $e_\al \in S_\al$ for $\al \in C_S$ 
such that $e_\al \cdot e_\be =\varepsilon(\al,\be)e_{\al+\be}$.
%
By Lemma \ref{simple_current}, $A_S(d)$ is an $A_S(0)$-module.
\begin{prop}\label{simple_decomposition}
Let $d \in D_S$.
Suppose that $N \subset A_S(d)$ is an $\Isr$-graded subspace such that $a \cdot n \in N$ for any $a\in A_S(0)$ and $n\in N$.
Then, $N=0$ or $N=A_S(d)$,
that is, $A_S(d)$ is an irreducible $\Isr$-graded $A_S(0)$-module.
\end{prop}
\begin{proof}
Let $N \subset A_S(d)$ be a non-zero sub-$A_S$-module.
Then, by Lemma \ref{associativity},
 $SN \equiv \{a\cdot n\}_{a\in S, n \in N}$ is an ideal of the framed algebra $S$.
Since $S$ is simple, $SN=S$.
Thus, $A_S(d)= A_S(d) \cap SN = \{a\cdot n\}_{a\in A_S(0), n \in N} = N$.
\end{proof}

We constructed even codes $(D_S,C_S)$ from the framed algebra $S$,
which we call {\it the structure codes} of the framed algebra.
Conversely, let $C \subset \Z_2^{l+r}$ be an even code,
that is, $C$ is a subgroup of $\Z_2^{l+r}$ and $|\al| \in 2\Z$ for any $\al \in C$.
Then, we have:
\begin{prop}
The twisted group algebra $\C[\hat{C}]$ is a simple framed algebra
whose codes are $(0,C)$.
\end{prop}
\begin{proof}
Since $C$ is even, (FA1) holds. (FA2) and (FA3) follow from the definition of $\C[\hat{C}]$ with $1=e_0$.
Let $\al_1,\al_2,\al_3 \in C$.
Then, by definition
$e_{\al_1}\cdot (e_{\al_2} \cdot e_{\al_3})=(-1)^{|\al_1\al_2|}e_{\al_2}\cdot (e_{\al_1} \cdot e_{\al_3})$, which implies (FA4). The simplicity is obvious.
\end{proof}
The above proposition gives trivial examples of framed algebras.
General framed algebras can be obtained by module extensions of $A_S(0)=\C[\hat{C}]$.

We end this section by introducing a three-dimensional framed algebra which is not a twisted group algebra and one of the simplest non-trivial framed algebra.

Let $S_{\mathrm{Ising}}=\C 1 \oplus \C a \oplus \C d$
be a three-dimensional vector space
with the $\Is^{(1,1)}$-grading,
$S_0=\C 1$, $S_{\frac{1}{2},\frac{1}{2}}=\C a$ and
$S_{\frac{1}{16},\frac{1}{16}}=\C d$.
We define a product on $S_{\mathrm{Ising}}$ by 
$a \cdot a =1$, $a \cdot d=d \cdot a=1$,
and $d \cdot d=1+a$ with unit $1$.

It is not hard to show that $S_{\mathrm{Ising}}$ is a simple framed algebra.
This fact will be shown in a general setting in Section \ref{sec_construction}.
So we only check (FA4) for $(a_1,a_2,a_3)=(d,d,d)$ and $\la' = (\ft,\ft)$ and $\la^0=(\fs,\fs)$ here.
Since
\begin{align*}
d \cdot_{\fs,\fs} (d \cdot_{\ft,\ft} d)= d \cdot a=d
\end{align*}
and
\begin{align*}
\sum_{\la \in A((\fs,\fs),(\fs,\fs),(\fs,\fs),(\fs,\fs))} 
&B_{(\fs,\fs),(\fs,\fs),(\fs,\fs),(\fs,\fs)}^{\la, (\ft,\ft)}
d \cdot_{\fs,\fs} (d \cdot_{\la} d)\\
&=\sum_{\la =(0,0), (\fs,\fs)}
\frac{1}{2} d \cdot_{\fs,\fs} (d \cdot_{\la} d)
= \frac{1}{2}( d \cdot 1+ d\cdot a)=d,
\end{align*}
(FA5) holds.

The codes of $S_{\mathrm{Ising}}$ are
$C_{S_{\mathrm{Ising}}} = <(1,1)> \subset \Z_2^{1+1}$
and $D_{S_{\mathrm{Ising}}} = <(1,1)> \subset \Z_2^{1+1}$.
The corresponding conformal field theory is the critical 2d Ising model.

Note that it is exceptional that $S_{\mathrm{Ising}}$ becomes an associative algebra.

\subsection{Classification of $\C[\hat{C_S}]$-modules}\label{sec_rep}
By Proposition \ref{simple_decomposition},
$A_S(d)$ is an irreducible module of $A_S(0)$.
In this section, we further study the structure of $A_S(d)$.
The results of this section is essentially obtained in the chiral setting \cite{Mi2} (see also \cite[Section 4.2]{LY}).
We rewrite them in terms of a framed algebra here.

We first recall that $A_S(d)= \bigoplus_{(d,c) \in I_S} S_{d,c}$.
Let $(d,c) \in I_S$ 
and consider the stabilizer subgroup
$C_S^d = \{\al \in C_S\;|\; e_\al\cdot S_{d,c} \subset S_{d,c}\}$.
Then, by the fusion rule and Lemma \ref{non-zero},
it is clear that $C_S^d = \{\al \in C_S\;|\; \dpe \al =0 \}$
or $C_S^d$ consists of vectors whose support is
in $d$. Since $C_S^d$ is a subgroup of $C_S$,
$\C[\hat{C_S^d}] = \bigoplus_{\al \in C_S^d}e_\al$ is a subalgebra of $\C[\hat{C_S}]$,
where the two-cocycle of $C_S^d$ is given by the restriction of $\varepsilon(-,-)$.
The following lemma can be proved similarly to the proof of Proposition \ref{simple_decomposition}:
\begin{lem}
Let $(d,c) \in I_S$.
If $N$ is a subspace $S_{d,c}$ such that
$e_\al\cdot N \subset N $ for any $\al \in C_S^d$,
then $N=0$ or $N=\C[\hat{C_S^d}]$,
that is, $S_{d,c}$ is a simple $\C[\hat{C_S^d}]$-module.
\end{lem}
By the universality of the tensor product,
we have a $\C[\hat{C_S}]$-module homomorphism,
$\phi_{d,c}: \C[\hat{C_S}] \otimes_{\C[\hat{C_S^d}]} S_{d,c} \rightarrow A_S(d)$. The map $\phi_{d,c}$ is surjective
since $A_S(d)$ is simple.
We claim that $\phi_{d,c}$ is isomorphism.
By the surjectivity, it suffices to show that
$\dim \C[\hat{C_S}] \otimes_{\C[\hat{C_S^d}]} S_{d,c} =  \dim A_S(d)$.  By the construction, $\dim \C[\hat{C_S}] \otimes_{\C[\hat{C_S^d}]} S_{d,c} = [C_S: C_S^d] \dim S_{d,c}$, where $ [C_S: C_S^d]$ is the index of groups.
Since for any $\al \in C_S$,
$e_\al\cdot: S_{d,c} \rightarrow S_{d,c+\dpe \al}$ is isomorphism, $\dim S_{d,c + \dpe \al}$ is equal to $\dim S_{d,c}$. Thus, we have:
\begin{lem}
The map $\phi_{d,c}: \C[\hat{C_S}] \otimes_{\C[\hat{C_S^d}]} S_{d,c} \rightarrow A_S(d)$ is a $\C[\hat{C_S}]$-module  isomorphism and preserves the $\Isr$-grading.
\end{lem}
Finally, we study the structure of a $\C[\hat{C_S^d}]$-module $S_{d,c}$. 
A subgroup $H \subset \Z_2^{l+r}$ is said to be isotropic if
$|\al\be| \in 2\Z$ for any $\al,\be \in H$.
Let $H_d \subset C_S^d$ be a maximal isotropic subspace.
Since $(-1)^{|\al\be|}=1$ for any $\al,\be \in H_d$,
$e_\al$ and $e_\be$ commute with each other in $\C[\hat{C_S}]$. Thus, $\bigoplus_{\al \in H_d}\C e_\al \subset \C[\hat{C_S}]$ is isomorphic to the (untwisted) group algebra $\C[H_d]$.
We fix an isomorphism $\C[H_d] \rightarrow \C[\hat{C_S}]$
and identify it as a subalgebra of $\C[\hat{C_S}]$.
Since $H_d$ is a finite group,
$S_{d,c}$ is decomposed into the sum of one-dimensional modules of $H_d$. 
Let $\chi: H_d \rightarrow \C^\times$ be a character
and let $e(d,c,\chi) \in S_{d,c}$ satisfy $e(d,c,\chi) \neq 0$ and 
$e_\al \cdot e(d,c,\chi)= \chi(\al) e(d,c,\chi)$ for any $\al \in H_d$,
an eigenvector.
Then,
by the universality of the tensor product again,
 $\psi:  \C[\hat{C_S^d}]\otimes_{\C[H_d]} \C \chi \rightarrow S_{d,c}$ is a (surjective) $\C[\hat{C_S^d}]$-module homomorphism.
By the representation theory of extraspecial 2-groups,
$\C[\hat{C_S^d}]\otimes_{\C[H_d]} \C \chi$ is an irreducible representation of $\C[\hat{C_S^d}]$
(see for example \cite{FLM}). Thus, the map is an isomorphism of $\C[\hat{C_S^d}]$-modules.

We observe that for $\al \in C_S^d$,
$e_\al \cdot e(d,c,\chi)$ is an eigenvector with the eigenvalue
$(-1)^{|\al-|}\chi: H_d \rightarrow \C^\times$,
where $(-1)^{|\al-|}: H_d \rightarrow \C^\times$
is the character defined by
$\be \mapsto (-1)^{|\al\be|}$.

We summarize the above argument as follows:
Let $\mathrm{Irr}\,\C[\hat{C_S^d}]$ be the set of irreducible $\C[\hat{C_S^d}]$-modules
and $H_d^\vee: \mathrm{Hom}(H_d, \C^\times)$ the set of characters.
For $M \in \mathrm{Irr}\,\C[\hat{C_S^d}]$,
there exists $\chi \in H_d^\vee$
such that $M=\bigoplus_{\al \in C_S^d/H_d} M_{(-1)^{|\al-|}\chi} $.
Thus, the correspondence $\eta: M \mapsto \chi_M \in H_d^\vee/C_S^d$ defines a map.
Hence, we have:
\begin{prop}
\label{rep_bijection}
The map $\eta: \mathrm{Irr}\,\C[\hat{C_S^d}] \rightarrow
H_d^\vee/ C_S^d$ gives a bijection.
\end{prop}

\subsection{Products among $\C[\hat{C_S}]$-modules} \label{subsec_products}

In this section, we derive an important formula for the product
$\cdot: A_S(d) \times A_S(d') \rightarrow A_S(d+d')$.
\begin{lem}\label{product_formula}
Let $a \in S_\al$ and $v_1 \in S_{d^1,c^1}$ and $v_2 \in S_{d^2,c^2}$
for $\al \in C_S$ and $(d^1,c^1), (d^2,c^2) \in I_S$.
Then, for $\ga \in \Z_2^{d^1d^2}$,
\begin{align*}
a \cdot (v_1 \cdot_{(d^1+d^2)_\perp (c^1+c^2) + \ga} v_2)&=(-1)^{\frac{1}{2}|d^1\al|}
(-1)^{|\al c^1|+|d^1d^2\al|+|\ga \al|+|d^1\al c^2|}
v_1 \cdot_{(d^1+d^2)_\perp (c^1+c^2+\al) + \ga} (a \cdot v_2).\\
a \cdot (v_1 \cdot v_2)&= (-1)^{|\dpe^1\al|+|d^1\al c^2|} (a \cdot v_1) \cdot v_2.
\end{align*}
\end{lem}
\begin{proof}
By Theorem \ref{connection},
for $\ga \in \Z_2^{d^1d^2}$
\begin{align*}
a \cdot (v_1 \cdot_{(d^1+d^2)_\perp (c^1+c^2) + \ga} v_2)
&=a \cdot_{(d^1+d^2)_\perp (\al+c^1+c^2) + \ga} (v_1 \cdot_{(d^1+d^2)_\perp (c^1+c^2) + \ga} v_2)\\
&= (-1)^{|\al c^1|}(-1)^{\frac{1}{2}|d^1\al|}(-1)^{|d^1\al(c^2+\ga+(d^1+d^2)_\perp (c^1+c^2+\al))|} 
v_1 \cdot_{(d^1+d^2)_\perp (c^1+c^2+\al) + \ga} (a \cdot v_2)\\
&=(-1)^{|\al c^1|}(-1)^{\frac{1}{2}|d^1\al|}
(-1)^{|d^1\al(c^2+\ga+d^2\al)|} 
v_1 \cdot_{(d^1+d^2)_\perp (c^1+c^2+\al) + \ga} (a \cdot v_2)\\
&=(-1)^{\frac{1}{2}|d^1\al|}
(-1)^{|\al c^1|+|d^1d^2\al|+|\ga \al|+|d^1\al c^2|}
v_1 \cdot_{(d^1+d^2)_\perp (c^1+c^2+\al) + \ga} (a \cdot v_2).
\end{align*}

By Lemma \ref{associativity}, for $\ga \in \Z_2^{d^1d^2}$,
\begin{align*}
a& \cdot_{(d^1+d^2)_\perp (c^1+c^2+\al)+\ga} (v_1 \cdot_{(d^1+d^2)_\perp (c^1+c^2)+\ga} v_2)\\
&=(-1)^{s(d^1+d^2,(d^1+d^2)_\perp (c^1+c^2+\al)+\ga)-
s(d^1+d^2,(d^1+d^2)_\perp (c^1+c^2)+\ga)+s(d^1,\dpe^1 \al+ c^1)
-s(d^1,c^1)}\\
&B_{(d^1+d^2,(d^1+d^2)_\perp (c^1+c^2+\al)+\ga),
(d^2,c^2),\al,(d^1,c^1)} (a\cdot v_1)\cdot_{(d^1+d^2)_\perp (c^1+c^2+\al)+\ga} v_2\\
&=(-1)^{\frac{1}{2}\Bigl(
|(d^1+d^2)_\perp (c^1+c^2+\al)+\ga|-|(d^1+d^2)_\perp (c^1+c^2)+\ga|
+|\dpe^1 \al+ c^1|-|c^1|\Bigr)}\\
&(-1)^{|c^2\al|}(-1)^{\frac{1}{2}|d^2\al|}(-1)^{|d^2\al(c^1+d^1\al+\ga)|}(a\cdot v_1)\cdot_{(d^1+d^2)_\perp (c^1+c^2+\al)+\ga} v_2\\
&=(-1)^{|\dpe^1d^2\al|+|\dpe^1\al|+|\al((d^1+d^2)_\perp(c^1+c^2)+\ga)|+|\dpe^1\al c^1|}
(-1)^{|c^2\al|+|d^2\al(c^1+d^1\al+\ga)|}(a\cdot v_1) \cdot_{(d^1+d^2)_\perp (c^1+c^2+\al)+\ga} v_2\\
&=
(-1)^{|\dpe^1\al|+|d^1\al c^2|} (a\cdot v_1)\cdot_{(d^1+d^2)_\perp (c^1+c^2+\al)+\ga} v_2.
\end{align*}
\end{proof}

By the above lemma, we have:
\begin{cor}
Let $v_{d^1} \in S_{(d^1,0)}$ and $v_{d^2} \in S_{(d^2,0)}$ for
$d^1,d^2\in D_S$.
Then,  for any $\al^1,\al^2 \in C_S$,
$$(e_{\al^1} \cdot v_{d^1})\cdot (e_{\al^2} \cdot v_{d^2})
=\sum_{\ga \in \Z_2^{d^1d^2}}
(-1)^{|\dpe^1\al|+|d^1\dpe^2 \al^1\al^2|+|d^1d^2\al^2|+|\ga \al^2|+
\frac{1}{2}|d^1\al^2|}
e_{\al^1} \cdot (e_{\al^2} \cdot (v_{d^1}\cdot_{(d^1+d^2,\ga)} v_{d^2})).$$
\end{cor}
Hence, if we know the product $v_{d^1}\cdot v_{d^2} \in A_S(d^1+d^2)$,
then the product on $A_S(d^1)\otimes A_S(d^2)$ are uniquely determined by
the $A_S(0)$-module structure.

\subsection{Code vertex operator algebra}\label{sec_code_vertex}
In this section, we consider the case of $r=0$,
a chiral conformal field theory.
If $\al,\be \in \Z_2^l$ satisfy $|\al|,|\be| \in 2\Z$,
then $|\al+\be|=|\al|+|\be|-2|\al\be|\in 2\Z$.
Thus, all even codewords in $\Z_2^l$ forms
a (maximal even) subgroup of $\Z_2^l$.
We denote it by $C_l^{\eve}$
and the corresponding full vertex algebra constructed in Theorem \ref{correspondence}
by $V_{C_l^{\eve}}$.

Since $V_{C_l^{\eve}}$ consists of only holomorphic fields,
by Proposition \ref{vertex_algebra}, $V_{C_l^{\eve}}$ is a vertex operator algebra.
Recall that the degree one space of a vertex operator algebra inherits a Lie algebra structure.
Set $(C_l^\eve)[2]=\{\al \in C_l^\eve\;|\;|\al|=2 \}$.
Then, $\{e_\al \}_{\al \in (C_l^\eve)_2}$ forms a basis of $(V_{C_l^{\eve}})_1$.
Hence, $\dim(V_{C_l^{\eve}})_1 = \# (C_l^\eve)[2] = \binom{l}{2}$.
By the construction of $V_{C_l^{\eve}}$, we have:
\begin{lem}\label{ope_Lie}
The $0$ and $1$-th products 
on $(V_{C_l^{\eve}})_1$ are given by:
\[
e_\al(0)e_\be=
\begin{cases}
\varepsilon(\al,\be)e_{\al+\be} & \text{if } \al+\be \in (C_l^\eve)_2,\\
0 & \text{otherwise}
\end{cases}
\]
and 
\[
e_\al(1)e_\be =
\begin{cases}
(-1) \1 & \text{if } \al=\be,\\
0 & \text{otherwise}
\end{cases}
\]
for $\al,\be \in (C_l^\eve)_2$.
\end{lem}
By the above lemma, it is not hard to show that
$(V_{C_l^{\eve}})_1$ is isomorphic to
the orthogonal Lie algebra $\mathrm{so}(l)$ as Lie algebras.
We will examine the structure of vertex algebras $V_{C_l^{\eve}}$
in more detail.

In the case of $l=1$, since $C_1^{\eve}=0$,
$V_{C_1^{\eve}}$ is nothing but the Virasoro vertex operator algebra $L(\ft,0)$.
In the case of $l=2$,
it is well-known that $V_{C_2^{\eve}}$ is isomorphic to the lattice vertex operator algebra 
$V_{2\Z}$ associated with the rank one lattice $\Z\al$ with $(\al,\al)=4$ (see \cite{DGH}).

In general for $k \geq 1$,
since $C_{2k}^{\eve}$ contains mutually orthogonal $k$ vectors $(1100\cdots)$, 
$(0011\cdots)$, $(00\cdots 11)$, $V_{C_{2k}^{\eve}}$ 
is an extension of the lattice vertex operator algebra $V_{(2\Z)^k}$
and thus a lattice vertex operator algebra.
Set $D_1=2\Z$ and for $k\geq 2$ let $D_k$ be the lattice generated by $(\pm1,\pm1,0,0,\cdots),
(0,\pm1,\pm1,0,\cdots), (0,0,\pm1,\pm1,\cdots),\dots, (0,0,0,\dots,\pm 1,\pm 1)
$ in $\R^k$.
If $k \geq 4$, then $D_k$ is the root lattice of type $D_k$.
It is easy to show that $V_{C_{2k}^{\eve}}$ is isomorphic to the lattice vertex operator algebra  $V_{D_k}$ for any $k \geq 1$.

Let us consider the odd cases starting from $l=3$.
Then, $\{e_{110},e_{011},e_{101}\}$ forms the basis of $(V_{C_3^\eve})_1$
and $\{2i e_{110}, -e_{011}-ie_{101}, e_{011}-ie_{101}\}$ forms a $\mathrm{sl}_2$-triple $\{h,e,f\}$
for some two-cocycle. By Lemma \ref{ope_Lie},
$h(1)h= 4 \1$. Thus, $V_{C_3^{\eve}}$ contains a subalgebra isomorphic to $V_{A_{1,2}}$,
 the simple affine vertex operator algebra of level $2$ associated with $\mathrm{sl}_2$.
Since $L(\ft,0)^{\otimes 3} \subset V_{A_{1,2}}$ and $(110),(011),(101)$ generate
$C_3^{\eve}$, we have $V_{C_3^{\eve}} \cong V_{A_{1,2}}$.

Similarly, for general odd integers $l=2k+1$ with $k\geq 2$,
$V_{C_{2k+1}^{\eve}}$ is isomorphic to  $V_{B_{k,1}}$,
the simple affine vertex operator algebra of level $1$ type $B_k$.

Hence, we have:
\begin{prop}
\label{classify_code_vertex}
The vertex operator algebras $\{V_{C_l^{\eve}}\}_{l=1,2,\dots}$ are isomorphic to:
\begin{enumerate}
\item
$V_{C_1^{\eve}}\cong L(\frac{1}{2},0)$, the simple Virasoro VOA of central charge $\ft$;
\item
$V_{C_{2k}^{\eve}} \cong V_{D_k}$, the lattice VOA for any $k \geq 1$;
\item
$V_{C_3^{\eve}} \cong V_{A_{1,2}}$, the simple affine VOA of type $A_1$ at level $2$;
\item
$V_{C_{2k+1}^{\eve}} \cong V_{B_{k,1}}$, the simple affine VOA of type $B_k$ at level $1$ for any $k \geq 2$.
\end{enumerate}
\end{prop}

\section{Code conformal field theory}\label{sec_construction}
In this section, we construct modular invariant conformal field theories from codes.
Hereafter, we assume that $l=r$.
In this case,
there is a group homomorphism 
$$\Delta: \Z_2^r \rightarrow \Z_2^{r+r},\;g \mapsto (g,g).$$
Let $G \subset \Z_2^r$ be a subgroup satisfying
$(1^r) \in G$
and $D_G=\Delta G \subset \Z_2^{r+r}$ the image of $G$ under $\Delta$
and set 
$$C_G = (\Delta G)^\perp =\{\al \in \Z_2^{r+r}\;|\; |\al d| \in 2\Z \text{ for any }d \in D_G \}.$$
The purpose of this section is to construct
an $(r,r)$-framed algebra with the codes $(D_G,C_G)$.

For any $\ga,\ga' \in \Delta\Z_2^r$, since
\begin{align}
|\ga \ga' |=|\ga \ga' |_l-|\ga \ga' |_r=0, \label{eq_isotropic}
\end{align}
the image of $\Delta$ is isotropic. This property is very useful in the construction.

In Section \ref{sec_preparation},
we explicitly construct a two-cocycle $\varepsilon(-,-) \in H^2(C_G,\Z_2)$
and representations $A_G(d)$ of $\C[\hat{C_G}]$ for each $d \in D_G$.
In Section \ref{sec_product},
we construct the product on $S_G=\bigoplus_{d \in D_G} A_G(d)$
and prove that $S_G$ is a simple $(r,r)$-framed algebra.
In Section \ref{sec_modular},
the character of the framed full vertex operator algebra $F_G$ associated with $S_G$
is shown to satisfy the modular invariance.
In Section \ref{sec_classify_code}, the codes $G$ are classified up to $r \leq 6$.

\subsection{Two-cocycle and representations}\label{sec_preparation}
We first construct a two-cocycle $\varepsilon(-,-) \in H^2(C_G,\Z_2)$.
Set $C_{r,r}^{\eve}=\{\al \in \Z_2^{r+r}\;|\; |\al| \in 2\Z \}=(\rreg)^\perp$,
which is the maximal even code in $\Z_2^{r+r}$.
Since $(1^r) \in G$, $\rreg=\Delta (1^r) \in D_G$
and $C_G = D_G^\perp \subset C_{r,r}^{\eve}$.
We will construct a two-cocycle on $C_{r,r}^{\eve}$
and define a two-cocycle on $C_G$ by restriction.

Let $e_i \in \Z_2^{r}$ be the vector whose $i$-th component is $1$ and all the other components are $0$
and set $\tilde{e}_i=(e_i,0)\in \Z_2^{r+r}$.
Then, it is clear that
$\{\Delta e_1,\dots, \Delta e_r, \tilde{e}_1-\tilde{e}_2,\tilde{e}_2-\tilde{e}_3,\dots, \tilde{e}_{r-1}-\tilde{e}_r\}$ is a basis of $C_{r,r}^{\eve}$.
We consider it as an ordered basis
$(v_1,v_2,\dots,v_{2r-1})= (\Delta e_1,\dots, \Delta e_r, \tilde{e}_1-\tilde{e}_2,\tilde{e}_2-\tilde{e}_3,\dots, \tilde{e}_{r-1}-\tilde{e}_r)$.
Let $\varepsilon(-,-): C_{r,r}^{\eve} \times C_{r,r}^{\eve} \rightarrow \Z_2$
be the bilinear map defined by
\begin{align*}
\varepsilon(v_i,v_j)=\begin{cases}
(-1)^{\frac{1}{2}|v_i|} & i=j\\
1 & i<j \\
(-1)^{|v_iv_j|} & i >j.
\end{cases}
\end{align*}
Then, $\varepsilon(-,-)$ satisfies $\varepsilon(\al,\be)\varepsilon(\be,\al)=(-1)^{(\al,\be)}$ for any $\al,\be \in C_{r,r}^\eve$.
The following lemma is very important:
\begin{lem}\label{trivial_diag}
For any $g, g' \in \Z_2^r$, $|\Delta g \Delta g'|=0$ and $\varepsilon(\Delta g,\Delta g')=1$
\end{lem}
\begin{proof}
The first equation follows from \eqref{eq_isotropic}.
Since $(\Delta e_1,\dots, \Delta e_r)$ forms a basis of $\Delta \Z_2^r \subset C_{r,r}^\eve$,
the second equation follows from the definition of the bilinear form $\varepsilon(-,-)$.
\end{proof}

Define the twisted group algebra $\C[\hat{C_G}]=\bigoplus_{\al \in C_G}\C e_\al$ by
the restriction of the two-cocycle $\varepsilon(-,-)$, i.e.,
 $e_\al \cdot e_\be =\varepsilon(\al,\be)e_{\al+\be}$ for $\al,\be\in C_G$.

Let $d \in \Delta \Z_2^r$.
Recall that we set $\dpe = \rreg +d$ and $\Z_2^d = \{\al \in \Z_2^{r+r}\;|\; \dpe \al=0 \}$, the set of codewords 
whose support is in $d$.
Then, set 
\begin{align*}
C_G^d &= \Z_2^d \cap C_G\\
\Delta^d &=\Z_2^d \cap \Delta \Z_2^r.
\end{align*}

In Section \ref{sec_rep} we show that a framed algebra with the code $(C_G,D_G)$
is a direct sum of irreducible $\C[\hat{C_G^d}]$-modules
and each irreducible $\C[\hat{C_G^d}]$-module can be constructed
from a one-dimensional representation of a maximal isotropic subspace in $C_G^d$ (see Proposition \ref{rep_bijection}).

By the virtue of Lemma \ref{trivial_diag}, we have:
\begin{lem}\label{maximal_isotropic}
For any $d \in D_G$, $\Delta^d$ is a maximal isotropic subspace of
$C_G^d$.
\end{lem}
\begin{proof}
Recall that by \eqref{eq_isotropic}, $|\Delta^{\rreg} \Delta^{\rreg}| \in 2\Z$, i.e., $\Delta^{\rreg}$ is isotropic.
Thus, $\Delta^d \subset \Delta^{\rreg}$ is isotropic.
By $D_G \subset \Delta^{\rreg}$, $|\Delta^d D_G|=0$, which implies $\Delta^d \subset C_G$.

It suffices to show that $\Delta^d$ is maximal.
Let $\al \in C_G^d$ satisfy $|\Delta^d \al|=0$.
We assume that $\al \notin \Delta \Z_2^r$.
Then, the left and right $i$-th component of $\al$
is different for some $i$.
Then, $|\al \Delta e_i| =\pm 1$ and $\Delta e_i \in \Delta^d$,
a contradiction.
Thus, $\al \in \Delta \Z_2^r \cap C_G^d$.
Hence, $\al \in \Delta^d$,
which implies that $\Delta^d$ is maximal.
\end{proof}

For any $\ga,\ga' \in \Delta^\rreg$,
by Lemma \ref{trivial_diag}, $e_{\ga} e_{\ga'}=\varepsilon(\ga,\ga') e_{\ga+\ga'}=e_{\ga+\ga'}$.
Thus, $\C[\Delta^d]$, the group algebra of $\Delta^d$, is canonically a subalgebra of $\C[\hat{C_G}]$ for any $d \in \Delta^{\rreg}$.
Let $d \in D_G$ and $\C t_d$ be the trivial representation of $\Delta^d$.
Set $$A_G(d) = \C[\hat{C_G}]\otimes_{\C[\Delta^d]} \C t_d,$$
the induced module
and $$S_G = \bigoplus_{d\in D_G}A_G(d).$$
We denote the action of $\C[\hat{C_G}]$ on $S_G$
by $\cdot$
and define an $\Isrr$-grading on $S_G$
by $e_\al \cdot t_d \in (S_G)_{d, \dpe \al}$ for any $\al \in C_G$.

We end this section by determining the dimension of $A_G(d)$.
Set
\begin{align*}
\Z_2^{r+0} &= \{(\al,0)\in \Z_2^{r+r}\;|\; \al \in \Z_2^r \}\\
C_G^{\lef} &= \Z_2^{r+0} \cap C_G.
\end{align*}
Then, we have $C_G= C_G^\lef \oplus \Delta^{\rreg}$
and $\Delta^\rreg = \Delta^d \oplus \Delta^{\dpe}$ for any $d \in D_G$.
Hence, $\{e_\al \cdot t_d\}_{\al \in C_G^\lef \oplus \Delta^{\dpe}}$ is a basis of $A_G(d)$.
Thus, $\dim_\C A_G(d) = \# (C_G^\lef \oplus \Delta^{\dpe})$.
Since $\#\Delta^{\dpe}= 2^{r-|d|_l}$ and $\# C_G^\lef= \frac{\# \Z_2^r}{\# G} = 2^{r-\dim G}$
and $\# C_G = 2^{2r-\dim G}$,
we have
\begin{align}
\dim_\C A_G(d)= 2^{2r-\dim G - |d|_l} =\frac{1}{2^{|d|_l}} \# C_G.
\end{align}

\subsection{Construction of product}\label{sec_product}
In this section, we will define a product on $S_G=\bigoplus_{d \in D_G} A_G(d)$.
For each $d \in \Delta^\rreg$,
set 
$$
\delta_{d}= \sum_{\al \in \Delta^d} e_{\al} \in \C[\hat{C_G}].
$$
Since by Lemma \ref{trivial_diag} the two-cocycle $\varepsilon(-,-)$ is trivial on $\Delta^\rreg$,
we have:
\begin{lem}
\label{delta_inv}
Let $d \in \Delta^\rreg$ and $\ga \in \Delta^d$.
Then, $\delta_d \cdot e_\ga= e_\ga \cdot \delta_d = \delta_d \in \C[\hat{C_G}]$.
\end{lem}

For $d^1,d^2 \in D_G$,
define a map $m_{d^1,d^2}: \C[\hat{C_G}] \times \C[\hat{C_G}]
\rightarrow A_G(d^1+d^2)$ by
$$
m_{d^1,d^2}(e_{\al^1}, e_{\al^2})
= (-1)^{|d^1\dpe^2 \al^1\al^2|+|d^1d^2\al^2|+\frac{1}{2}|d^1\al^2|}
e_{\al^1}\cdot \de_{d^1d^2} \cdot  e_{\al^2} \cdot t_{d^1+d^2}
$$
for $\al^1,\al^2 \in C_G$.
\begin{lem}
For any $\ga^1 \in \Delta^{d^1}$, $\ga^2 \in \Delta^{d^2}$
and $\al,\be \in C_G$,
$$
m_{d^1,d^2}(e_\al e_{\ga^1}, e_\be e_{\ga^2})
=m_{d^1,d^2}(e_\al, e_\be).
$$
\end{lem}
\begin{proof}
Set $\ze_1=\dpe^2 \ga^1, \ze_{12}=d^2  \ga^1$ and 
$\ze_2 = \dpe^2 \ga^1$, $\ze_{21}=d^1 \ga^2$.
Then, $\ga^1=\ze_1+\ze_{12}$, $\ga^2=\ze_2+\ze_{21}$.
Since $\ze_{12}, \ze_{21} \in \Delta^{d^1d^2}$
and $\ze_1,\ze_2 \in \Delta^{d^1+d^2}$, by Lemma \ref{delta_inv} and definition of $A_G(d^1+d^2)$,
we have $e_{\ze_{21}} \cdot \de_{d^1d^2}
= \de_{d^1d^2}=\de_{d^1d^2} \cdot e_{\ze_{12}}$
and $e_{\ze_1} \cdot t_{d^1+d^2}=t_{d^1+d^2}=e_{\ze_2} \cdot t_{d^1+d^2}$.
By Lemma \ref{trivial_diag}, $e_{\ga^1}=e_{\ze_1}\cdot e_{\ze_{12}}$ and $e_{\ga^2}=e_{\ze_2}\cdot e_{\ze_{21}}$, we have
\begin{align*}
m_{d^1,d^2}(e_\al e_{\ga^1}, e_\be e_{\ga^2})
&= m_{d^1,d^2}(e_\al e_{\ze_1}e_{\ze_{12}}, e_\be e_{\ze_2}e_{\ze_{21}})\\
&= (-1)^{|d^1\dpe^2(\al+\ga^1)(\be+\ga^2)|+|d^1d^2(\be+\ga^2)|+\frac{1}{2}|d^1(\be+\ga^2)|}
(e_\al \cdot e_{\ga^1}) \cdot \de_{d^1d^2} \cdot (e_{\be}\cdot e_{\ga^2})\cdot t_{d^1+d^2}\\
&= (-1)^{|d^1\dpe^2(\al+\ze_1)\be|+|d^1d^2(\be+\ze_{21})|+\frac{1}{2}|d^1(\be+\ze_{21})|}
e_\al \cdot e_{\ze_1} \cdot \de_{d^1d^2} \cdot e_{\be}\cdot e_{\ze_{21}} \cdot t_{d^1+d^2}.
\end{align*}
By definition of $\varepsilon(-,-)$, we have $e_{\be}\cdot e_{\ze_{21}}=(-1)^{|\be \ze_{21}|}e_{\ze_{21}} \cdot e_{\be}$ and $e_{\ze_1} \cdot \de_{d^1d^2} \cdot e_{\be}
= (-1)^{|\ze_1\be|} \de_{d^1d^2} \cdot e_{\be}\cdot e_{\ze_1}$.
Thus,
\begin{align*}
\text{RHS}&=
(-1)^{|d^1\dpe^2(\al+\ze_1)\be|+|d^1d^2(\be+\ze_{21})|+\frac{1}{2}|d^1(\be+\ze_{21})|}
(-1)^{|\be\ze_{21}|+|\be \ze_1|}
e_\al \cdot \de_{d^1d^2} \cdot e_{\be} \cdot t_{d^1+d^2}\\
&=
(-1)^{|d^1\dpe^2\al\be|+|\ze_1 \be|
+|d^1d^2\be|+|\ze_{21}| + \ft(|d^1\be|+|\ze_{21}|-2|\be\ze_{21}|)}
(-1)^{|\be\ze_{21}|+|\be \ze_1|}
e_\al \cdot \de_{d^1d^2} \cdot e_{\be} \cdot t_{d^1+d^2}\\
&=
(-1)^{|d^1\dpe^2\al\be|+|d^1d^2\be|+\ft|d^1\be|}
e_\al \cdot \de_{d^1d^2} \cdot e_{\be} \cdot t_{d^1+d^2}\\
&=m_{d^1,d^2}(e_\al, e_\be),
\end{align*}
where we used $|\ze|=|\ze|_l-|\ze|_r=0$ for any $\ze \in \Delta^\rreg$.
\end{proof}

By the above lemma, $m_{d^1,d^2}$ gives the well-defined map
$m_{d^1,d^2}:A_G(d^1)\times A_G(d^2)\rightarrow A_G(d^1+d^2)$.
Combining $m_{d^1,d^2}$ for all $d^1,d^2\in D_G$,
we have a product $\cdot:S_G \otimes S_G \rightarrow S_G$.
We note that the action of $\C[\hat{C_G}]$ on $A_G(d)$
and $m_{0,d}:\C[\hat{C_G}] \otimes A_G(d) \rightarrow A_G(d)$ coincide with each other.
In fact, for any $\al,\be \in C_G$ and $d \in D_G$,
\begin{align*}
m_{0,d}(e_\al, e_\be \cdot t_d)
&=e_\al \cdot \de_0 \cdot e_\be \cdot t_d\\
&= e_\al \cdot (e_\be \cdot t_d).
\end{align*}
Thus, we can use the same notation $\cdot$.

We will show that $S_G$ is an $(r,r)$-framed algebra.
Recall that $S_G$ is spanned by $\{e_\al \cdot t_d \in (S_G)_{d,\dpe\al}\}_{\al \in C_G, d \in D_G}$.
For any $d \in D_G$ and $\al \in C_G$,
$s(d,\dpe \al)=\fs|d|+\ft|\dpe \al|=\ft|\dpe \al|$.
Since $\rreg \in D_G$, $\dpe=d+\rreg \in D_G$.
Thus, by $C_G =D_G^\perp$, $\ft|\dpe \al| \in \Z$, which implies (FA1).

For any $d \in D_G$ and $\al \in C_G$,
\begin{align*}
(e_\al \cdot t_d) \cdot e_0 &=m_{d,0}(e_\al \cdot t_d,e_0)\\
&=e_\al \cdot \de_0 \cdot e_0 \cdot t_d\\
&= e_\al \cdot t_d,
\end{align*}
which implies (FA2).

\begin{lem}\label{const_fusion}
For $d^1,d^2 \in D_G$ and $\al^1,\al^2\in C_G$,
$$(e_{\al^1} \cdot t_{d^1}) \cdot (e_{\al^2} \cdot t_{d^2}) \in
\bigoplus_{\ga \in \Delta^{d^1d^2}} (S_{G})_{(d^1+d^2,
(d_1+d_2)_\perp(\al^1+\al^2)+\ga)}.$$
\end{lem}
\begin{proof}
By definition,
$e_{\al^1}\cdot \de_{d^1d^2} \cdot  e_{\al^2} \cdot t_{d^1+d^2}
=\sum_{\ga \in \Delta^{d^1d^2}} e_{\al^1}\cdot e_{\ga} \cdot  e_{\al^2} \cdot t_{d^1+d^2}$.
Since
$e_{\al^1}\cdot e_{\ga} \cdot  e_{\al^2} \cdot t_{d^1+d^2}
\in (S_G)_{(d^1+d^2),(d^1+d^2)_\perp (\al^1+\al^2+\ga)}.
$
The assertion follows from $(d^1+d^2)_\perp \ga = \ga$.
\end{proof}

Hence, it suffices to show that (FA4).
Since $|\ga|_l=|\ga|_r$ for any $\ga \in \Delta^\rreg$, we have:
\begin{lem}
\label{simpler_connection}
For $d^1,d^2,d^3 \in D_G$,
$\exp(\frac{-\pi i}{8}|d^1d^2|)=1$ and 
$(\frac{1+i}{2})^{|d^1d^2d^3|_l}(\frac{1-i}{2})^{|d^1d^2d^3|_r}=(\frac{1}{2})^{|d^1d^2d^3|_l}$.
\end{lem}

By the above lemma and Theorem \ref{connection},
the connection matrix of an $(r,r)$-framed algebra with the code $(\Delta G,(\Delta G)^\perp)$
is simpler than general cases.

Let $\al^1,\al^2,\al^3 \in C_G$, $d^1,d^2,d^3\in D_G$.
We will consider the products
$(e_{\al^1} \cdot t_{d^1})\cdot \Bigl((e_{\al^2} \cdot t_{d^2}) \cdot (e_{\al^3} \cdot t_{d^3})\Bigr)$
and $(e_{\al^2} \cdot t_{d^2})\cdot \Bigl((e_{\al^1} \cdot t_{d^1}) \cdot (e_{\al^3} \cdot t_{d^3})\Bigr)$
and show (FA4).
By Lemma \ref{const_fusion},
$(e_{\al^1} \cdot t_{d^1})\cdot \Bigl((e_{\al^2} \cdot t_{d^2}) \cdot (e_{\al^3} \cdot t_{d^3})\Bigr)$
is in 
\begin{align}
\bigoplus_{\ga' \in \Delta^{d^1(d^2+d^3)}} \bigoplus_{\ga \in \Delta^{d^2d^3}} (S_G)_{(d^1+d^2+d^3,(d^1+d^2+d^3)_\perp(\al^1+\al^2+\al^3+\ga+\ga')}.
\label{eq_direct2}
\end{align}
Since $(d^1+d^2+d^3)_\perp(d^2d^3)=\dpe^1d^2d^3$ and 
$(d^1+d^2+d^3)_\perp d^1(d^2+d^3)=\dpe^2 d^3d^1+\dpe^3d^1d^2$,
\eqref{eq_direct2} can be written as:
\begin{align}
\bigoplus_{\ga^1 \in \Delta^{\dpe^1d^2d^3)}}
\bigoplus_{\ga^2 \in \Delta^{\dpe^2d^3d^1)}}
\bigoplus_{\ga^3 \in \Delta^{\dpe^3d^1d^2)}}
(S_G)_{(d^1+d^2+d^3,(d^1+d^2+d^3)_\perp(\al^1+\al^2+\al^3)+\ga^1+\ga^2+\ga^3)}.\label{eq_direct}
\end{align}
Let $\ga^1 \in \Delta^{\dpe^1d^2d^3}$,
$\ga^2 \in \Delta^{\dpe^2d^1d^3}$, $\ga^3 \in \Delta^{\dpe^3d^1d^2}$.
By \eqref{eq_direct}, it suffices to consider the products
$(e_{\al^1} \cdot t_{d^1})\cdot_{d^1+d^2+d^3,(d^1+d^2+d^3)_\perp(\al^1+\al^2+\al^3)+\ga^1+\ga^2+\ga^3)} \Bigl((e_{\al^2} \cdot t_{d^2}) \cdot (e_{\al^3} \cdot t_{d^3})\Bigr)$
and\\
 $(e_{\al^2} \cdot t_{d^2})\cdot_{d^1+d^2+d^3, (d^1+d^2+d^3)_\perp(\al^1+\al^2+\al^3)+\ga^1+\ga^2+\ga^3)} \Bigl((e_{\al^1} \cdot t_{d^1}) \cdot (e_{\al^3} \cdot t_{d^3})\Bigr)$.

By \eqref{eq_direct2}, each
$(\ga,\ga') \in \Delta^{d^2d^3}\times \Delta^{d^1(d^2+d^3)}$
which can contribute to 
$$(e_{\al^1} \cdot t_{d^1})\cdot_{d^1+d^2+d^3,(d^1+d^2+d^3)_\perp(\al^1+\al^2+\al^3)+\ga^1+\ga^2+\ga^3)} \Bigl((e_{\al^2} \cdot t_{d^2}) \cdot (e_{\al^3} \cdot t_{d^3})\Bigr)$$
must satisfy
\begin{align*}
\ga^1&=\dpe^1 \ga\\
\ga^2&=\dpe^2 \ga'\\
\ga^3&=\dpe^3 \ga'.
\end{align*}
Since $\ga' \in \Delta^{d^1(d^2+d^3)}$, $\ga'$ is uniquely determined by this equation, more specifically, $\ga'=\ga^2+\ga^3$.
However, there is an ambiguity in the choice of $\ga$.
In fact, $\ga$ must be of the form $\{\ga^1+\ga^s\}_{\ga^s \in \Delta^{d^1d^2d^3}}$
and is not uniquely determined since $\dpe^1 \ga^s=0$.

Let $\ga^s \in \Delta^{d^1d^2d^3}$ and we will calculate
$(e_{\al^1} \cdot t_{d^1})\cdot_{d^1+d^2+d^3,(d^1+d^2+d^3)_\perp(\al^1+\al^2+\al^3)+\ga^1+\ga^2+\ga^3)} \Bigl((e_{\al^2} \cdot t_{d^2}) \cdot_{d^2+d^3,(d^2+d^3)_\perp(\al^2+\al^3)+\ga^s+\ga^1} (e_{\al^3} \cdot t_{d^3})\Bigr)$.
By the definition,
\begin{align*}
&(e_{\al^2} \cdot t_{d^2}) \cdot_{(d^2+d^3, (d^2+d^3)_\perp(\al_2+\al_3) +\ga^s+\ga^1)} (e_{\al^3}\cdot t_{d^3})\\
&=(-1)^{|d^2d^3\al^3|+|d^2\dpe^3\al^2\al^3|+\frac{1}{2}|d^2\al^3|}
e_{\al^2} \cdot e_{\ga^s+\ga^1} \cdot e_{\al^3} \cdot t_{d^2+d^3}.
\end{align*}
Thus,
\begin{align*}
&(e_{\al^1}\cdot t_{d^1})\cdot_{(d^1+d^2+d^3)_\perp(\al^1+\al^2+\al^3) + \ga^1+\ga^2+\ga^3}
\Bigl((e_{\al^2} \cdot t_{d^2}) \cdot_{(d^2+d^3)_\perp(\al_2+\al_3) +\ga^s+\ga^1} (e_{\al^3}\cdot t_{d^3})\Bigr)\\
&=(-1)^{|d^2d^3\al^3|+|d^2\dpe^3\al^2\al^3|+\frac{1}{2}|d^2\al^3|}
(-1)^{|d^1(d^2+d^3)(\al^2+\al^3)|
+|d^1(d^2+d^3)_\perp \al^1(\al^2+\al^3+\ga^s)|+
\frac{1}{2}|d^1(\al^2+\al^3+\ga^s)|}\\
&e_{\al^1} \cdot e_{\ga^2+\ga^3} \cdot e_{\al^2}
\cdot e_{\ga^1+\ga^s} \cdot e_{\al^3} \cdot t_{d^1+d^2+d^3}\\
&=(-1)^{|d^2d^3\al^3|+|d^2\dpe^3\al^2\al^3|+\frac{1}{2}|d^2\al^3|}
(-1)^{|d^1(d^2+d^3)(\al^2+\al^3)|
+|d^1(d^2+d^3)_\perp \al^1(\al^2+\al^3)|+
\frac{1}{2}|d^1(\al^2+\al^3)|}(-1)^{|\ga^s(\al^1+\al^2+\al^3)|}
\\
&(-1)^{|\al^2(\ga^2+\ga^3)|+|\ga^s\al^3|}
e_{\al^1} \cdot e_{\al^2} \cdot e_{\ga^1+\ga^2+\ga^3} \cdot e_{\al^3} \cdot t_{d^1+d^2+d^3}\\
&=(-1)^{|d^2d^3\al^3|+|d^2\dpe^3\al^2\al^3|+\frac{1}{2}|d^2\al^3|}
(-1)^{|d^1(d^2+d^3)(\al^2+\al^3)|
+|d^1(d^2+d^3)_\perp \al^1(\al^2+\al^3)|+
\frac{1}{2}|d^1(\al^2+\al^3)|}
\\
&(-1)^{|\al^2(\ga^2+\ga^3)|+|\ga^s(\al^1+\al^2)|}
e_{\al^1} \cdot e_{\al^2} \cdot e_{\ga^1+\ga^2+\ga^3} \cdot e_{\al^3} \cdot t_{d^1+d^2+d^3}
\end{align*}
and similarly for $\ga^u \in \Delta^{d^1d^2d^3}$
\begin{align*}
&(e_{\al^2} \cdot t_{d^2})\cdot_{(d^1+d^2+d^3)_\perp(\al^1+\al^2+\al^3) + \ga^1+\ga^2+\ga^3}
\Bigl((e_{\al^1} \cdot t_{d^1}) \cdot_{(d^1+d^3, (d^1+d^3)_\perp(\al_1+\al_3) +\ga^u+\ga^2)} (e_{\al^3} \cdot  t_{d^3})\Bigr)\\
&=(-1)^{|d^1d^3\al^3|+|d^1\dpe^3\al^1\al^3|+\frac{1}{2}|d^1\al^3|}
(-1)^{|d^2(d^1+d^3)(\al^1+\al^3)|
+|d^2(d^1+d^3)_\perp \al^2(\al^1+\al^3)|+
\frac{1}{2}|d^2(\al^1+\al^3)|}\\
&(-1)^{|\al^1(\ga^1+\ga^3)|+|\ga^u(\al^1+\al^2)|}(-1)^{|\al^1\al^2|}
 e_{\al^1} \cdot e_{\al^2} \cdot e_{\ga^1+\ga^2+\ga^3}\cdot e_{\al^3} \cdot t_{d^1+d^2+d^3}.
\end{align*}
We remark that $(\frac{1}{2})^{|d^1d^2d^3|_l}=\frac{1}{\# \Delta^{d^1d^2d^3}}$.
Thus, by Lemma \ref{simpler_connection} and Theorem \ref{connection},
\begin{align}
&B_{(d^1+d^2+d^3,(d^1+d^2+d^3)_\perp(\al^1+\al^2+\al^3)+\ga^1+\ga^2+\ga^3,
(d^1,\dpe^1\al^1),(d^2,\dpe^2\al^2),(d^3,\dpe^3\al^3))}^{
(d^2+d^3,(d^2+d^3)_\perp(\al^2+\al^3)+\ga^1+\ga^s),
(d^1+d^3,(d^1+d^3)_\perp(\al^1+\al^3)+\ga^2+\ga^s)} \label{eq_connection1} \\
&=\frac{1}{\# \Delta^{d^1d^2d^3}}
(-1)^{|\dpe^1\dpe^2\al^1\al^2|}
(-i)^{|d^2 \dpe^1\al^1|+|d^1\dpe^2 \al^2|}
i^{|d^1d^2\dpe^3(\al^1+\al^2)+\ga^3|} \nonumber
\\
&(-i)^{|d^1d^2d^3(\al^1+\al^2)+(\ga^s+\ga^u)|}
(-1)^{|(d^1\dpe^2\al^2+d^2\dpe^1\al^1)\left(\dpe^3\al^3+
(d^1+d^2+d^3)_\perp(\al^1+\al^2+\al^3)+(\ga^1+\ga^2+\ga^3) \right)|}. \nonumber
\end{align}
Since
\begin{align*}
&(-i)^{|d^2 \dpe^1\al^1|+|d^1\dpe^2 \al^2|}
i^{|d^1d^2\dpe^3(\al^1+\al^2)+\ga^3|}(-i)^{|d^1d^2d^3(\al^1+\al^2)+(\ga^s+\ga^u)|}\\
&=(-i)^{|d^1d^2\al^1|+|d^2\al^1|-2|d^1d^2\al^1|+|d^1d^2\al^2|+|d^1\al^2|-2|d^1d^2\al^2|}
i^{|d^1d^2\dpe^3(\al^1+\al^2)|-2|\ga^3(\al^1+\al^2)|}
(-i)^{|d^1d^2d^3(\al^1+\al^2)| - 2|(\al^1+\al^2)(\ga^s+\ga^u)|}\\
&=(-1)^{|d^1d^2(\al^1+\al^2)|+|(\al^1+\al^2)(\ga^s+\ga^u)|+|\ga^3(\al^1+\al^2)|+\ft(|d^1\al^2|+|d^2\al^1|)}
(-i)^{|d^1d^2\al^1|+|d^1d^2\al^2|}
i^{|d^1d^2\dpe^3(\al^1+\al^2)|}
(-i)^{|d^1d^2d^3(\al^1+\al^2)|}
\end{align*}
and
\begin{align*}
&(-i)^{|d^1d^2\al^1|+|d^1d^2\al^2|}
i^{|d^1d^2\dpe^3(\al^1+\al^2)|}
(-i)^{|d^1d^2d^3(\al^1+\al^2)|}\\
&=(-i)^{|d^1d^2\al^1|+|d^1d^2\al^2|}
i^{|d^1d^2(\al^1+\al^2)|+|d^1d^2d^3(\al^1+\al^2)|-2|d^1d^2d^3(\al^1+\al^2)|}
(-i)^{|d^1d^2d^3(\al^1+\al^2)|}\\
&=(-1)^{|d^1d^2d^3(\al^1+\al^2)|}
(-i)^{|d^1d^2\al^1|+|d^1d^2\al^2|}
i^{|d^1d^2\al^1|+|d^1d^2\al^2)|-2|d^1d^2\al^1\al^2|}\\
&=(-1)^{|d^1d^2d^3(\al^1+\al^2)|+|d^1d^2\al^1\al^2|},
\end{align*}
we have
\begin{align*}
&B_{(d^1+d^2+d^3,(d^1+d^2+d^3)_\perp(\al^1+\al^2+\al^3)+\ga^1+\ga^2+\ga^3,
(d^1,\dpe^1\al^1),(d^2,\dpe^2\al^2),(d^3,\dpe^3\al^3))}^{
(d^2+d^3,(d^2+d^3)_\perp(\al^2+\al^3)+\ga^1+\ga^s),
(d^1+d^3,(d^1+d^3)_\perp(\al^1+\al^3)+\ga^2+\ga^s)}\\
&=\frac{1}{\# \Delta^{d^1d^2d^3}}
(-1)^{|\dpe^1\dpe^2\al^1\al^2|}(-1)^{|(d^1\dpe^2\al^2+d^2\dpe^1\al^1)\left(\dpe^3\al^3+(d^1+d^2+d^3)_\perp(\al^1+\al^2+\al^3)+(\ga^1+\ga^2+\ga^3) \right)|}\\
&(-1)^{|\ga^3(\al^1+\al^2)|+|(\ga^s+\ga^u)(\al^1+\al^2)|+
|d^1d^2d^3(\al^1+\al^2)|+
|d^1d^2(\al^1+\al^2+\al^1\al^2)|+\frac{1}{2}(|d^1\al^2|+|d^2\al^1|)
}.
\end{align*}
Since
\begin{align*}
&(-1)^{|(d^1\dpe^2\al^2+d^2\dpe^1\al^1)\left(\dpe^3\al^3+(d^1+d^2+d^3)_\perp(\al^1+\al^2+\al^3)+(\ga^1+\ga^2+\ga^3) \right)|}(-1)^{|d^1d^2d^3(\al^1+\al^2)|}\\
&=
(-1)^{|\al^1\ga^1|+|\al^2\ga^2|}(-1)^{|d^1d^2d^3(\al^1+\al^2)|+|d^1\dpe^2\dpe^3\al^2\al^3|+
|\dpe^1d^2\dpe^3\al^1\al^3|+|d^1\dpe^2d^3\al^2(\al^1+\al^2+\al^3)|
+|\dpe^1d^2d^3\al^1(\al^1+\al^2+\al^3)|}\\
&=
(-1)^{|\al^1\ga^1|+|\al^2\ga^2|}(-1)^{|d^1d^2d^3(\al^1+\al^2)|+|d^1\dpe^2\dpe^3\al^2\al^3|+
|\dpe^1d^2\dpe^3\al^1\al^3|+|d^1\dpe^2d^3\al^2|+|d^1\dpe^2d^3\al^2(\al^1+\al^3)|
+|\dpe^1d^2d^3\al^1|+|\dpe^1d^2d^3\al^1(\al^2+\al^3)|}\\
&=(-1)^{|\al^1\ga^1|+|\al^2\ga^2|+|d^1d^3\al^2|+|d^2d^3\al^1|}
(-1)^{|d^1\dpe^2\dpe^3\al^2\al^3|+|\dpe^1d^2\dpe^3\al^1\al^3|+|d^1\dpe^2d^3\al^2(\al^1+\al^3)|
+|\dpe^1d^2d^3\al^1(\al^2+\al^3)|}\\
&=(-1)^{|\al^1\ga^1|+|\al^2\ga^2|+|d^1d^3\al^2|+|d^2d^3\al^1|}
(-1)^{|d^1\dpe^2\al^2\al^3|+|\dpe^1d^2\al^1\al^3|+|(d^1+d^2)d^3\al^1\al^2|},
\end{align*}
finally we have
\begin{align*}
&B_{(d^1+d^2+d^3,(d^1+d^2+d^3)_\perp(\al^1+\al^2+\al^3)+\ga^1+\ga^2+\ga^3,
(d^1,\dpe^1\al^1),(d^2,\dpe^2\al^2),(d^3,\dpe^3\al^3))}^{
(d^2+d^3,(d^2+d^3)_\perp(\al^2+\al^3)+\ga^1+\ga^s),
(d^1+d^3,(d^1+d^3)_\perp(\al^1+\al^3)+\ga^2+\ga^s)}\\
&=\frac{1}{\# \Delta^{d^1d^2d^3}}
(-1)^{|\al^1\ga^1|+|\al^2\ga^2|+
|d^1\dpe^2 \al^2\al^3|+|d^2\dpe^1\al^1\al^3|
+|(d^1+d^2)d^3\al^1\al^2|
+|d^1d^3\al^2|+|d^2d^3\al^1|}\\
&(-1)^{|\dpe^1\dpe^2\al^1\al^2|}(-1)^{|\ga^3(\al^1+\al^2)|+|(\ga^s+\ga^u)(\al^1+\al^2)|+
|d^1d^2(\al^1+\al^2+\al^1\al^2)|+\frac{1}{2}(|d^1\al^2|+|d^2\al^1|)
}
\end{align*}
Set 
\begin{align*}
X_{12}(\ga^s)&= (-1)^{|d^2d^3\al^3|+|d^2\dpe^3\al^2\al^3|+\frac{1}{2}|d^2\al^3|}
(-1)^{|d^1(d^2+d^3)(\al^2+\al^3)|
+|d^1(d^2+d^3)_\perp \al^1(\al^2+\al^3)|+
\frac{1}{2}|d^1(\al^2+\al^3)|}\\
&\times (-1)^{|\al^2(\ga^2+\ga^3)|+|\ga^s(\al^1+\al^2)|}\\
X_{21}(\ga^u)&=
(-1)^{|d^1d^3\al^3|+|d^1\dpe^3\al^1\al^3|+\frac{1}{2}|d^1\al^3|}
(-1)^{|d^2(d^1+d^3)(\al^1+\al^3)|
+|d^2(d^1+d^3)_\perp \al^2(\al^1+\al^3)|+
\frac{1}{2}|d^2(\al^1+\al^3)|}\\
&\times (-1)^{|\al^1(\ga^1+\ga^3)|+|\ga^u(\al^1+\al^2)|+|\al^1\al^2|}\\
B(\ga^s,\ga^u)&=(-1)^{|\al^1\ga^1|+|\al^2\ga^2|+
|d^1\dpe^2 \al^2\al^3|+|d^2\dpe^1\al^1\al^3|
+|(d^1+d^2)d^3\al^1\al^2|
+|d^1d^3\al^2|+|d^2d^3\al^1|}\\
&(-1)^{|\dpe^1\dpe^2\al^1\al^2|}(-1)^{|\ga^3(\al^1+\al^2)|+|(\ga^s+\ga^u)(\al^1+\al^2)|+
|d^1d^2(\al^1+\al^2+\al^1\al^2)|+\frac{1}{2}(|d^1\al^2|+|d^2\al^1|)
}.
\end{align*}
Then, (FA4) is equivalent to prove:
\begin{align}
X_{21}(\ga^u)=\frac{1}{\# \Delta^{d^1d^2d^3}} \sum_{\ga^s \in \Delta^{d^1d^2d^3}}
B(\ga^s,\ga^u)X_{12}(\ga^s)
\label{eq_need}
\end{align}
for any $\ga^u \in \Delta^{d^1d^2d^3}$.

Hence, in order to prove (FA4),
it suffices to show that $1=X_{21}(\ga^u)X_{12}(\ga^s)B(\ga^u,\ga^s)$.
In
\begin{align*}
&X_{21}(\ga^u)X_{12}(\ga^s)B(\ga^u,\ga^s)\\
&=(-1)^{|d^2d^3\al^3|+|d^2\dpe^3\al^2\al^3|+\frac{1}{2}|d^2\al^3|}
(-1)^{|d^1(d^2+d^3)(\al^2+\al^3)|
+|d^1(d^2+d^3)_\perp \al^1(\al^2+\al^3)|+
\frac{1}{2}|d^1(\al^2+\al^3)|}\\
&(-1)^{|\al^2(\ga^2+\ga^3)|}(-1)^{|\ga^s(\al^1+\al^2)|}\\
&(-1)^{|d^1d^3\al^3|+|d^1\dpe^3\al^1\al^3|+\frac{1}{2}|d^1\al^3|}
(-1)^{|d^2(d^1+d^3)(\al^1+\al^3)|
+|d^2(d^1+d^3)_\perp \al^2(\al^1+\al^3)|+
\frac{1}{2}|d^2(\al^1+\al^3)|}\\
&(-1)^{|\al^1(\ga^1+\ga^3)|}(-1)^{|\ga^u(\al^1+\al^2)|}(-1)^{|\al^1\al^2|}\\
&(-1)^{|\al^1\ga^1|+|\al^2\ga^2|+
|d^1\dpe^2 \al^2\al^3|+|d^2\dpe^1\al^1\al^3|
+|(d^1+d^2)d^3\al^1\al^2|
+|d^1d^3\al^2|+|d^2d^3\al^1|}\\
&(-1)^{|\dpe^1\dpe^2\al^1\al^2|}(-1)^{|\ga^3(\al^1+\al^2)|+|(\ga^s+\ga^u)(\al^1+\al^2)|+
|d^1d^2(\al^1+\al^2+\al^1\al^2)|+\frac{1}{2}(|d^1\al^2|+|d^2\al^1|)
},
\end{align*}
all the terms which involve $\ga^1,\ga^2,\ga^3,\ga^s,\ga^u$ cancel
each other.
Thus,
\begin{align*}
\text{RHS} =&(-1)^{|d^2d^3\al^3|+|d^2\dpe^3\al^2\al^3|+\frac{1}{2}|d^2\al^3|}
(-1)^{|d^1(d^2+d^3)(\al^2+\al^3)|
+|d^1(d^2+d^3)_\perp \al^1(\al^2+\al^3)|+
\frac{1}{2}|d^1(\al^2+\al^3)|}\\
&(-1)^{|d^1d^3\al^3|+|d^1\dpe^3\al^1\al^3|+\frac{1}{2}|d^1\al^3|}
(-1)^{|d^2(d^1+d^3)(\al^1+\al^3)|
+|d^2(d^1+d^3)_\perp \al^2(\al^1+\al^3)|+
\frac{1}{2}|d^2(\al^1+\al^3)|}(-1)^{|\al^1\al^2|}\\
&(-1)^{|d^1\dpe^2 \al^2\al^3|+|d^2\dpe^1\al^1\al^3|
+|(d^1+d^2)d^3\al^1\al^2|
+|d^1d^3\al^2|+|d^2d^3\al^1|}\\
&(-1)^{|\dpe^1\dpe^2\al^1\al^2|}(-1)^{|d^1d^2(\al^1+\al^2+\al^1\al^2)|+\frac{1}{2}(|d^1\al^2|+|d^2\al^1|)
}.
\end{align*}
and similarly all the terms which involve $\al^3$ cancel,
\begin{align*}
\text{RHS} =&(-1)^{|d^1(d^2+d^3)\al^2|
+|d^1(d^2+d^3)_\perp \al^1\al^2|+
\frac{1}{2}|d^1\al^2|}(-1)^{|d^2(d^1+d^3)\al^1|
+|d^2(d^1+d^3)_\perp \al^2\al^1|+
\frac{1}{2}|d^2\al^1|}(-1)^{|\al^1\al^2|}\\
&(-1)^{|(d^1+d^2)d^3\al^1\al^2|
+|d^1d^3\al^2|+|d^2d^3\al^1|}
(-1)^{|\dpe^1\dpe^2\al^1\al^2|}(-1)^{|d^1d^2(\al^1+\al^2+\al^1\al^2)|+\frac{1}{2}(|d^1\al^2|+|d^2\al^1|)
}
\end{align*}
and thus all terms cancel.
Hence, we have:
\begin{thm}\label{construction}
For any code $G \subset \Z_2^r$ with $(1^r) \in G$,
$S_G$ is a simple framed algebra with the codes $(\Delta G,(\Delta G)^\perp)$.
\end{thm}
\begin{proof}
It suffices to show that $S_G$ is simple,
which follows from the fact that $A_G(d)$ is an irreducible $\C[\hat{C_G}]$-module
and the product $A_G(d) \otimes A_G(d') \rightarrow A_G(d+d')$ is non-zero.
\end{proof}

By Theorem \ref{correspondence},
$F_G = F_{S_G}$ is a full vertex operator algebra.
Thus, we have:
\begin{cor}
For any code $G \subset \Z_2^r$ with $(1^r) \in G$,
$F_G$ is a simple framed full vertex operator algebra.
\end{cor}

We can calculate all the four point correlation functions of $F_G$ explicitly
by using the code $G$. For this purpose, it is convenient to introduce the length function on $G \subset \Z_2^r$. Define a map $|-|_\Delta: \Z_2^r \rightarrow \Z_{\geq 0}$
by $|g|_\Delta = \sum_{i=1}^r g_i$ for $g=(g_1,g_2,\dots,g_r) \in \Z_2^r$.

For the sake of simplicity, we only compute some special correlators.
General correlators can be obtained similarly.
Recall that by the discussion in Section \ref{sec_rep} $(S_G)_{(d^1+d^2+d^3,0)}$ is an irreducible representation of $\C[\hat{C_G^{d^1+d^2+d^3}}]$
and decomposed into the direct sum of distinct one dimensional representations
of $\C[\hat{\Delta^{d^1+d^2+d^3}}]$, i.e.,
\begin{align}
(S_G)_{(d^1+d^2+d^3,0)}=\bigoplus_{\chi \in (\Delta^{d^1+d^2+d^3})^\vee} ((S_G)_{(d^1+d^2+d^3,0)})_{\chi}.
\label{eq_chi_decom}
\end{align}
Furthermore by construction
$t_{d^1+d^2+d^3}$ spans the trivial representation of $\C[\Delta^{d^1+d^2+d^3}]$,\\
$
(S_G)_{(d^1+d^2+d^3,0)})_{\text{trivial}}=\C t_{d^1+d^2+d^3}.
$
Let $t_{d^1+d^2+d^3}^\vee \in \mathrm{Hom}_\C(S_G,\C)$ be the linear map defined by the 
composition of projections
$$S_G=\bigoplus_{(d,c)}(S_G)_{d,c} \rightarrow (S_G)_{(d^1+d^2+d^3,0)}$$
and
$$(S_G)_{(d^1+d^2+d^3,0)}=\bigoplus_{\chi \in (\Delta^{d^1+d^2+d^3})^\vee} ((S_G)_{(d^1+d^2+d^3,0)})_{\chi} \rightarrow (S_G)_{(d^1+d^2+d^3,0)})_{\text{trivial}}\cong \C,$$
where $(S_G)_{(d^1+d^2+d^3,0)})_{\text{trivial}}=\C t_{d^1+d^2+d^3} \cong \C$
is normalized by $t_{d^1+d^2+d^3}^\vee(t_{d^1+d^2+d^3})=1$.
\begin{rem}
\label{rem_bilinear}
We recall that by Proposition \ref{bilinear}, $S_G$ has a natural non-degenerate symmetric bilinear form 
$(-,-):S_G \times S_G \rightarrow \C$.
By the invariance, for $\chi,\chi' \in (\Delta^{d^1+d^2+d^3})^\vee$,
$((S_G)_{(d^1+d^2+d^3,0)})_{\chi}$ and $((S_G)_{(d^1+d^2+d^3,0)})_{\chi'}$
are orthogonal if $\chi \neq \chi'$.
Thus, $t_{d^1+d^2+d^3}^\vee(-)=(t_{d^1+d^2+d^3},-) \in S_G^*$.
\end{rem}

Then, we have:
\begin{prop}
\label{correlator}
For $d^1,d^2,d^3 \in D_G$, set $d^0=d^1+d^2+d^3$.
\begin{align*}
\langle
&t_{d^0}^*,Y(t_{d^1},\uz_1)Y(t_{d^2},\uz_2),t_{d^3}
\rangle\\
&=
2^{-\ft |d^1d^2d^3|_l}
(z_1\z_1)^{-\frac{1}{8}|d^1d^3|_l}(z_2\z_2)^{-\frac{1}{8}|d^2d^3|_l}
((z_1-z_2)(\z_1-\z_2))^{-\frac{1}{8}|d^1d^2|_l}\\
&\left(
\sqrt{z_1\z_1}+\sqrt{z_2\z_2}+\sqrt{(z_1-z_2)(\z_1-\z_2)}
\right)^{\ft |d^1d^2d^3|_l}.
\end{align*}
\end{prop}
\begin{proof}
By \eqref{eq_correlator},
\begin{align*}
\langle
t_{d^1+d^2+d^3}^*,&Y(t_{d^1},\uz_1)Y(t_{d^2},\uz_2),t_{d^3}
\rangle\\
&=\sum_{\ga \in \Delta^{d^2d^3}}
C_{(d^1+d^2+d^3,0),(d^1,0),(d^2,0),(d^3,0)}^{(d^2+d^3,\ga)}(z_1,z_2)
\langle
t_{d^1+d^2+d^3}^*, t_{d^1} \cdot \left(t_{d^2} \cdot_{(d^2+d^3,\ga)} t_{d^3}\right)
\rangle\\
&=\sum_{\ga \in \Delta^{d^2d^3}}
C_{(d^1+d^2+d^3,0),(d^1,0),(d^2,0),(d^3,0)}^{(d^2+d^3,\ga)}(z_1,z_2)
\langle
t_{d^1+d^2+d^3}^*, t_{d^1} \cdot (e_\ga \cdot t_{d^2+d^3})
\rangle\\
&=\sum_{\ga \in \Delta^{d^2d^3}}
C_{(d^1+d^2+d^3,0),(d^1,0),(d^2,0),(d^3,0)}^{(d^2+d^3,\ga)}(z_1,z_2)
\langle
t_{d^1+d^2+d^3}^*,
(-1)^{|d^1(d^2+d^3)\ga|+\ft|d^1\ga|}
\de_{d^1(d^2+d^3)} \cdot e_\ga \cdot
t_{d^1+d^2+d^3}
\rangle\\
&=\sum_{\ga \in \Delta^{d^1d^2d^3}}
C_{(d^1+d^2+d^3,0),(d^1,0),(d^2,0),(d^3,0)}^{(d^2+d^3,\ga)}(z_1,z_2).
\end{align*}
By Table \ref{table_vacuum},
\begin{align*}
C_{\fs,\fs,0,0}(x,y)&=C_{\fs,0,\fs,0}(x,y)=C_{\fs,0,0,\fs}(x,y)=1\\
C_{0,\fs,\fs,0}(x,y)&=(x-y)^{-\frac{1}{8}}\\
C_{0,\fs,0,\fs}(x,y)&=x^{-\frac{1}{8}},\;\;C_{0,0,\fs,\fs}(x,y)=y^{-\frac{1}{8}}.
\end{align*}
Thus, we have
\begin{align*}
\langle
t_{d^1+d^2+d^3}^*,&Y(t_{d^1},\uz_1)Y(t_{d^2},\uz_2),t_{d^3}
\rangle\\
&=
2^{-|d^1d^2d^3|_l} (z_1\z_1)^{-\frac{1}{8}|d^1d^3|_l}(z_2\z_2)^{-\frac{1}{8}|d^2d^3|_l}
((z_1-z_2)(\z_1-\z_2))^{-\frac{1}{8}|d^1d^2|_l}\\
&\left(\sqrt{\sqrt{z_1}+\sqrt{z_2})(\sqrt{\z_1}+\sqrt{\z_2})}+\sqrt{\sqrt{z_1}-\sqrt{z_2})(\sqrt{\z_1}-\sqrt{\z_2})}\right)^{|d^1d^2d^3|_l},
\end{align*}
where we used, by Table \ref{table_block},
\begin{align*}
&C_{\fs,\fs,\fs,\fs}^0(z_1,z_2)\overline{C_{\fs,\fs,\fs,\fs}^0(z_1,z_2)}
+C_{\fs,\fs,\fs,\fs}^{\ft}(z_1,z_2)\overline{C_{\fs,\fs,\fs,\fs}^{\ft}(z_1,z_2)}\\
&= \frac{1}{2}\left(z_1\z_1z_2\z_2(z_1-z_2)(\z_1-\z_2)\right)^{-\frac{1}{8}}
\left(\sqrt{\sqrt{z_1}+\sqrt{z_2})(\sqrt{\z_1}+\sqrt{\z_2})}+\sqrt{\sqrt{z_1}-\sqrt{z_2})(\sqrt{\z_1}-\sqrt{\z_2})}
\right).
\end{align*}
and
\begin{align*}
&\left(\sqrt{\sqrt{z_1}+\sqrt{z_2})(\sqrt{\z_1}+\sqrt{\z_2})}+\sqrt{\sqrt{z_1}-\sqrt{z_2})(\sqrt{\z_1}-\sqrt{\z_2})}
\right)^2\\
&=2(\sqrt{z_1\z_1}+\sqrt{z_2\z_2}+\sqrt{(z_1-z_2)(\z_1-\z_2)})
\end{align*}
\end{proof}
The correlator computed the above proposition is
a limit of the four point correlation function (see Remark \ref{rem_four_point})
As expected, the four point correlation function is a more symmetric form.
Let $\langle-\rangle:F_G \rightarrow \C$ be the projection
$F_G=\bigoplus_{h,\h \in \R^2}(F_G)_{h,\h}\rightarrow (F_G)_{0,0}=\C\1\cong \C$,
where $\C\1 \cong \C$
is normalized by $\langle \1 \rangle=1$.
The linear map $\langle-\rangle$ is called the vacuum expectation value
and the formal power series of the form,
\begin{align*}
\langle Y(a_0,\uz_0)Y(a_1,\uz_1)Y(a_2,\uz_2)Y(a_3,\uz_3)\1 \rangle,
\end{align*}
is called {\it a four point correlation function}.
It was proved in \cite{M2} that under certain condition
all four point correlation functions are convergent to a real analytic function on 
$$X_4(\CP)=\{(z_1,z_2,z_3,z_4) \in (\CP)^4\;|\; z_i\neq z_j \}
$$
and can be calculated from the specialize one like
$$
\langle
t_{g^0}^*,Y(t_{g^1},\uz_1)Y(t_{g^2},\uz_2),t_{g^3}
\rangle.
$$
%
%
%
%
By using these results, we have:
\begin{cor}
\label{true_four_point}
Let $d^0,d^1,d^2,d^3 \in D_G$.
Then,
\begin{align*}
\langle
Y(t_{d^0},\uz_0)&Y(t_{d^1},\uz_1)Y(t_{d^2},\uz_2)Y(t_{d^3},\uz_3)\1
\rangle\\
&=2^{-|d^0d^1d^2d^3|_l}\delta_{d^0+d^1+d^2+d^3,0}
\Pi_{0 \leq i < j \leq 3} \left((z_i-z_j)(\z_i-\z_j)\right)^{-\frac{1}{8}|d^i d^j|_l}
F(z_0,z_1,z_2,z_3)^{|d^0d^1d^2d^3|_l}
\end{align*}
where we set
\begin{align*}
F(z_0,z_1,z_2,z_3)
=\left(|(z_0-z_1)(z_2-z_3)|^\ft
+|(z_0-z_2)(z_1-z_3)|^\ft
+|(z_0-z_3)(z_1-z_2)|^\ft\right)^\ft.
\end{align*}
\end{cor}

For example, let $r=1$ and $G=\langle 1 \rangle $.
Then, the corresponding $(1,1)$-framed algebra is the three dimensional framed algebra
$S_{\text{Ising}}=\C1\oplus\C a\oplus \C d$ with the product
$d\cdot d=1+a$ and $a\cdot d=d\cdot a=d$ and the grading $(S_{\text{Ising}})_{0,0}=\C1,
(S_{\text{Ising}})_{\ft,\ft}=\C a$ and $(S_{\text{Ising}})_{\fs,\fs}=\C d$.
By the above result, 
the four point correlation function for $(d,d,d,d) \in S_{\text{Ising}}^4$ is
\begin{align*}
\frac{1}{2}
\Pi_{0 \leq i < j \leq 3} \left((z_i-z_j)(\z_i-\z_j)\right)^{-\frac{1}{8}}
\left(|(z_0-z_1)(z_2-z_3)|^\ft
+|(z_0-z_2)(z_1-z_3)|^\ft
+|(z_0-z_3)(z_1-z_2)|^\ft\right)^\ft,
\end{align*}
which is exactly equal to the correlator of the critical Ising model in the physics literature (see for example \cite{BPZ}).

Another example of a four point correlation function is given in the following proposition:
\begin{prop}
\label{correlator2}
Let $\al^0, \al^1,\al^2,\al^3 \in C_G^\lef$.
Then,
\begin{align*}
\langle
&Y(e_{\al^0}\cdot t_{\rreg},\uz_0)Y(e_{\al^1}\cdot t_{\rreg},\uz_1)Y(e_{\al^2}\cdot t_{\rreg},\uz_2)Y(e_{\al^3}\cdot t_{\rreg},\uz_3)\1
\rangle\\
=&\delta_{\al^0+\al^1+\al^2+\al^3,0} 2^{-r}(-1)^{|\al^1\al^3|} (1,e_{\al^0}\cdot e_{\al^1}\cdot e_{\al^2}\cdot e_{\al^3})
\Pi_{0 \leq i < j \leq 3} \left((z_i-z_j)(\z_i-\z_j)\right)^{-\frac{r}{8}}\\
&\times
F(z_0,z_1,z_2,z_3)^{r-|\al^0\al^1+\al^0\al^2+\al^0\al^3+\al^1\al^2+\al^1\al^3+\al^2\al^3|_l}\\
&G_{01,23}(z_0,z_1,z_2,z_3)^{|\al^0\al^1+\al^2\al^3|_l}
G_{02,13}(z_0,z_1,z_2,z_3)^{|\al^0\al^2+\al^1\al^3|_l}
G_{03,12}(z_0,z_1,z_2,z_3)^{|\al^0\al^3+\al^1\al^2|_l}
\end{align*}
where $G_{ij,kl}(z_1,z_2,z_3,z_4)$ are defined by
\begin{align*}
G_{01,23}(z_0,z_1,z_2,z_3)
&=\left(-|(z_0-z_1)(z_2-z_3)|^\ft
+|(z_0-z_2)(z_1-z_3)|^\ft
+|(z_0-z_3)(z_1-z_2)|^\ft \right)^\ft,\\
G_{02,13}(z_0,z_1,z_2,z_3)
&=\left( |(z_0-z_1)(z_2-z_3)|^\ft
-|(z_0-z_2)(z_1-z_3)|^\ft
+|(z_0-z_3)(z_1-z_2)|^\ft \right)^\ft,\\
G_{03,12}(z_0,z_1,z_2,z_3)
&=\left( |(z_0-z_1)(z_2-z_3)|^\ft
+|(z_0-z_2)(z_1-z_3)|^\ft
-|(z_0-z_3)(z_1-z_2)|^\ft \right)^\ft.
\end{align*}
\end{prop}
\begin{proof}
We will first calculate 
$(e_\al^0\cdot t_{\rreg},Y(e_\al^1\cdot t_{\rreg},\uz_1)Y(e_\al^2\cdot t_{\rreg},\uz_2)e_\al^3\cdot t_{\rreg})$.
By Proposition \ref{bilinear} and Remark \ref{rem_invariance},
\begin{align*}
(e_{\al^0}\cdot t_{\rreg},&Y(e_{\al^1}\cdot t_{\rreg},\uz_1)Y(e_{\al^2}\cdot t_{\rreg},\uz_2)e_{\al^3}\cdot t_{\rreg})\\
&=\sum_{\ga \in \Delta^{\rreg}}C_{(\rreg,0),(\rreg,0),(\rreg,0),(\rreg,0)}^{(0^{r+r},\al^2+\al^3+\ga)}(z_1,z_2)
(-1)^{\ft |\al^3|}\left(e_{\al^0}\cdot t_{\rreg}, (e_{\al^1}\cdot t_{\rreg}) \cdot 
(e_{\al^2} \cdot e_\ga \cdot e_{\al^3})
\right)\\
&=\sum_{\ga \in \Delta^{\rreg}}C_{(\rreg,0),(\rreg,0),(\rreg,0),(\rreg,0)}^{(0^{r+r},\al^2+\al^3+\ga)}(z_1,z_2)
(-1)^{\ft |\al^3|}\left( (e_{\al^1}\cdot t_{\rreg})\cdot (e_{\al^0}\cdot t_{\rreg}), e_{\al^2} \cdot e_\ga \cdot e_{\al^3}\right)\\
&=\sum_{\ga,\ga' \in \Delta^{\rreg}}C_{(\rreg,0),(\rreg,0),(\rreg,0),(\rreg,0)}^{(0^{r+r},\al^2+\al^3+\ga)}(z_1,z_2)
(-1)^{\ft (|\al^3|+|\al^0|)}
\left(e_{\al^1}\cdot e_{\ga'} \cdot e_{\al^0}, e_{\al^2} \cdot e_\ga \cdot e_{\al^3}\right)\\
&=\sum_{\ga,\ga' \in \Delta^{\rreg}}C_{(\rreg,0),(\rreg,0),(\rreg,0),(\rreg,0)}^{(0^{r+r},\al^2+\al^3+\ga)}(z_1,z_2)
(-1)^{\ft (|\al^3|+|\al^1|)}
\left(1,e_{\al^0} \cdot e_{\ga'} \cdot   e_{\al^1}\cdot e_{\al^2} \cdot e_\ga \cdot e_{\al^3}\right).
\end{align*}
The left-hand-side is zero unless $\ga+\ga'=\al^0+\al^1+\al^2+\al^3$.
Since $\ga+\ga'\in \Delta^\rreg$, $\al^0+\al^1+\al^2+\al^3 \in C_G^\lef$ and $C_G^\lef \cap \Delta^\rreg=0$,
we may assume that $\ga=\ga'$ and $\al^0=\al^1+\al^2+\al^3$. Then,
\begin{align*}
&=\left(1,e_{\al^0} \cdot e_{\al^1}\cdot e_{\al^2}\cdot e_{\al^3} \right)(-1)^{\ft (|\al^3|+|\al^1|)}
\sum_{\ga \in \Delta^{\rreg}}(-1)^{|\ga(\al^1+\al^2)|}
C_{(\rreg,0),(\rreg,0),(\rreg,0),(\rreg,0)}^{(0^{r+r},\al^2+\al^3+\ga)}(z_1,z_2).
\end{align*}
By Table \ref{table_block},
\begin{align}
&C_{\fs,\fs,\fs,\fs}^0(z_1,z_2)\overline{C_{\fs,\fs,\fs,\fs}^0(z_1,z_2)}
-C_{\fs,\fs,\fs,\fs}^{\ft}(z_1,z_2)\overline{C_{\fs,\fs,\fs,\fs}^{\ft}(z_1,z_2)}\nonumber \\
&= \frac{1}{2}\left(z_1\z_1z_2\z_2(z_1-z_2)(\z_1-\z_2)\right)^{-\frac{1}{8}}
\left(\sqrt{\sqrt{z_1}+\sqrt{z_2})(\sqrt{\z_1}-\sqrt{\z_2})}+\sqrt{\sqrt{z_1}-\sqrt{z_2})(\sqrt{\z_1}+\sqrt{\z_2})}
\right)\nonumber \\
&C_{\fs,\fs,\fs,\fs}^\ft(z_1,z_2)\overline{C_{\fs,\fs,\fs,\fs}^0(z_1,z_2)}
+C_{\fs,\fs,\fs,\fs}^{0}(z_1,z_2)\overline{C_{\fs,\fs,\fs,\fs}^{\ft}(z_1,z_2)} \label{eq_rewrite} \\
&= \frac{1}{\sqrt{2}}\left(z_1\z_1z_2\z_2(z_1-z_2)(\z_1-\z_2)\right)^{-\frac{1}{8}}
\left(
\sqrt{\sqrt{z_1}+\sqrt{z_2})(\sqrt{\z_1}+\sqrt{\z_2})}
-\sqrt{\sqrt{z_1}-\sqrt{z_2})(\sqrt{\z_1}-\sqrt{\z_2})}
\right), \nonumber \\
&C_{\fs,\fs,\fs,\fs}^\ft(z_1,z_2)\overline{C_{\fs,\fs,\fs,\fs}^0(z_1,z_2)}
-C_{\fs,\fs,\fs,\fs}^{0}(z_1,z_2)\overline{C_{\fs,\fs,\fs,\fs}^{\ft}(z_1,z_2)} \nonumber \\
&= \frac{1}{\sqrt{2}}\left(z_1\z_1z_2\z_2(z_1-z_2)(\z_1-\z_2)\right)^{-\frac{1}{8}}
\left(
\sqrt{\sqrt{z_1}+\sqrt{z_2})(\sqrt{\z_1}-\sqrt{\z_2})}
-\sqrt{\sqrt{z_1}-\sqrt{z_2})(\sqrt{\z_1}+\sqrt{\z_2})}
\right), \nonumber
\end{align}
We would like to rewrite \eqref{eq_rewrite} in more symmetric way.
An important observation is that
\begin{align*}
&\left(\sqrt{\sqrt{z_1}+\sqrt{z_2})(\sqrt{\z_1}-\sqrt{\z_2})}+\sqrt{\sqrt{z_1}-\sqrt{z_2})(\sqrt{\z_1}+\sqrt{\z_2})}
\right)^2\\
&=\sqrt{z_1\z_1}-\sqrt{z_2\z_2} +\sqrt{(z_1-z_2)(\z_1-\z_2)}
\end{align*}
and similarly
\begin{align*}
&\left(
\sqrt{\sqrt{z_1}+\sqrt{z_2})(\sqrt{\z_1}+\sqrt{\z_2})}
-\sqrt{\sqrt{z_1}-\sqrt{z_2})(\sqrt{\z_1}-\sqrt{\z_2})}
\right)^2\\
&=\sqrt{z_1\z_1}+\sqrt{z_2\z_2}
-\sqrt{(z_1-z_2)(\z_1-\z_2)}\\
&\left(
\sqrt{\sqrt{z_1}+\sqrt{z_2})(\sqrt{\z_1}-\sqrt{\z_2})}
-\sqrt{\sqrt{z_1}-\sqrt{z_2})(\sqrt{\z_1}+\sqrt{\z_2})}
\right)^2\\
&=\sqrt{z_1\z_1}-\sqrt{z_2\z_2}
-\sqrt{(z_1-z_2)(\z_1-\z_2)}.
\end{align*}
Hence, by the results in \cite{M2}, the assertion holds.
\end{proof}

\begin{rem}
As it is mentioned in Remark \ref{rem_four_point},
the pure braid group ${P}_4$ acts on the functions $G_{ij,kl}(z_1,z_2,z_3,z_4)$.
The action of the standard generators $\si_{12},\si_{23},\si_{34} \in P_4$ can be written as:
\begin{align*}
\si_{12}:&(G_{12,34},G_{13,24},G_{14,23})
\mapsto(G_{12,34},-G_{14,23},G_{13,24})\\
\si_{23}:&(G_{12,34},G_{13,24},G_{14,23})
\mapsto (G_{13,24},-G_{12,34},G_{14,23})\\
\si_{34}:&(G_{12,34},G_{13,24},G_{14,23})
\mapsto (G_{12,34},G_{14,23},-G_{13,24}),
\end{align*}
which generates $S_4 \cong \Z_2^2 \rtimes S_3$.
\end{rem}

The following proposition follows from the construction:
\begin{prop}\label{orthogonal}
Let $r_1,r_2\in \Z_{>}$ and $G_1 \subset \Z_2^{r_1}$
and $G_2 \subset \Z_2^{r_2}$ be codes with
$(1^{r_i}) \in G_i$ for $i=1,2$.
Then,
$F_{G_1\perp G_2}$ is isomorphic to $F_{G_1}\otimes F_{G_2}$
as a full vertex operator algebra,
where $G_1 \perp G_2$ is a subcode of $\Z_2^{r_1+r_2}$.
\end{prop}

Let us consider the chiral part (holomorphic part) of $F_G$ (see Proposition \ref{vertex_algebra}).
Since the $\frac{1}{16}$-part is diagonal,
it does not contribute to the holomorphic part.
Thus, the holomorphic part of $F_G$ is equal to the holomorphic part of $F_{\C[\hat{C_G}]}$
and is determined by $C_G^\lef = C_G \cap (\Z_2^r \times 0)$,
which is exactly equal to 
$$G^\perp=\{\al \in \Z_2^r\;|\; |\al G|_\Delta\subset 2\Z \},$$
the dual code of $G \subset \Z_2^r$ with respect to $|-|_\Delta$.
Set $$
G^\perp[2]=\{\al \in G^\perp\;|\; |\al|_\Delta=2 \}.
$$
Then, it is clear that $\dim(F_G)_{1,0}=\# G^\perp[2]$.
Hence, we have:
\begin{prop}\label{chiral}
The holomorphic part of $F_G$ is isomorphic to
$F_{\C[\hat{G^\perp}]}$.
Furthermore, 
\begin{align*}
\dim (F_G)_{1,0}= \dim (F_G)_{0,1}=\#G^\perp[2].
\end{align*}
\end{prop}
Since $(1^r) \in G$, $G^\perp$ is an even code. Thus, $G^\perp$ is a subcode of the maximal even code
$C_{r,0}^\eve \subset \Z_2^r$.
The vertex operator algebra structure of $F_{\C[\hat{C_{r,0}^\eve}]}$
is determined in Proposition \ref{classify_code_vertex}.
Thus, we can determine the vertex operator algebra structure of $F_{\C[\hat{G^\perp}]}$
as a subalgebra of $F_{\C[\hat{C_{r,0}^\eve}]}$.

\subsection{Modular invariance}\label{sec_modular}
Let $G \subset \Z_2^r$ be a code satisfying $(1^r) \in G$.
The $(q,\bar{q})$-character of $F_G$ is defined by
\begin{align*}
\mathrm{Ch}_q F_G &=\mathrm{tr}q^{L(0)-\frac{r}{48}}\bar{q}^{\Ld(0)-\frac{r}{48}}|_{F_G} \in \Z[[q^\R,\bar{q}^\R]] \\
&=\sum_{h,\h \in \R} \dim (F_G)_{h,\h}\;q^{h-\frac{r}{48}}\bar{q}^{\h-\frac{r}{48}}.
\end{align*}
We will soon see that the $(q,\bar{q})$-character
is a real analytic function on the upper half-plane $\mathcal{H}$
with the identification $q=\exp(2\pi i \tau)$ and $\bar{q}=\exp(-2\pi i \bar{\tau})$
for $\tau \in \mathcal{H}$.
In this section, we prove the modular invariance of the $(q,\bar{q})$-character of
a full vertex operator algebra $F_G$.

Let $\C \Is$ be the $\C$-vector space with the basis
$\theta_{0}, \theta_{\frac{1}{2}}, \theta_{\frac{1}{16}}$
and
set $\C \Isrr = (\C \Is)^{\otimes r} \otimes (\C \Is)^{\otimes r} =\bigoplus_{\mu \in \Isrr} \C \theta_\mu$.

A character of a framed algebra $S=\bigoplus_{\mu \in \Isrr} S_\mu$ is defined by
$$
Z_S = \sum_{\mu \in \Isrr} \dim S_{\mu}\;  \theta_\mu \in \C \Isrr.
$$

Let $\C[X^{r,r}]$ be the space of real analytic functions on $\mathcal{H}$ spanned by the product of theta functions 
$\{\chi_0^{a_1}(\tau)\overline{\chi_0^{b_1}(\tau)}
\chi_\ft^{a_2}(\tau)\overline{\chi_\ft^{b_2}(\tau)}
\chi_\fs^{a_3}(\tau)\overline{\chi_\fs^{b_3}(\tau)}\}_{a_1,a_2,a_3,b_1,b_2,b_3 \in \Z_{\geq 0}}$
and $F: \C[\Isrr] \rightarrow \C[X^{r,r}]$
the linear map defined by
$F(\theta_\mu)= \Pi_{i=1}^r \chi_{\mu_i}(\tau)\overline{\chi_{\mu_i}(\tau)}$
for $\mu=(\mu_1,\dots,\mu_r, \bar{\mu}_1,\dots,\bar{\mu}_r) \in \Isrr$.
By Lemma \ref{character},
we have:
\begin{lem}
The $(q,\bar{q})$--character of the full vertex operator algebra $F_{G}$
is equal to $F(Z_{S_G})$.
\end{lem}

Recall that $A_G(d)$ is the induced module $\C[\hat{C_G}]\otimes_{\C[\Delta^d]} \C t_d$
for any $d \in D_G$. Hence, we have:
\begin{lem}\label{character_D}
For the framed algebra $S_G$,
$Z_{S_G}= \sum_{d \in D_G}\sum_{\al \in C_G} 2^{-|d|_l} \theta_{(d,\dpe \al)}$.
In particular, $\dim (S_G)_{d,\dpe \al} = \# [C_G^d: \Delta^d]= 2^{-|d|} \# C_G^d$ for any $d \in D_G$ and $\al \in C_G$.
\end{lem}

Set $Z_G=Z_{S_G}$.
By the assumption (FA1), the invariance of T-transformation, i.e.,
$F(Z_D)(\tau+1)=F(Z_D)(\tau)$, is clear.
Motivated by Lemma \ref{character},
we define a linear map $S: \C \Is \rightarrow \C \Is$
by
\begin{align*}
S(\theta_{0})&= 
\frac{1}{2}\theta_{0}+\frac{1}{2}\theta_{\frac{1}{2}}
+\frac{1}{\sqrt{2}}\theta_{\frac{1}{16}}\\
S(\theta_{\frac{1}{2}})&= 
\frac{1}{2}\theta_{0}+\frac{1}{2}\theta_{\frac{1}{2}}
-\frac{1}{\sqrt{2}}\theta_{\frac{1}{16}}\\
S(\theta_{\frac{1}{16}})&= 
\frac{1}{\sqrt{2}}\theta_{0} - \frac{1}{\sqrt{2}}\theta_{\frac{1}{2}}.
\end{align*}
We also denote the map 
$S^{\otimes r+r}:\Isrr \rightarrow \Isrr$ by $S$.
Since $F(Z_G)(\frac{-1}{\tau}) = F(S(Z_G))(\tau)$,
in order to prove the modular invariance of the character $\mathrm{Ch}_q F_G$, it suffices to show that
$S(Z_G)=Z_G$.

\begin{rem}
We remark that $F$ is not injective, since, for example,
$F(\theta_{(\frac{1}{2},0)})=\chi_0 \chi_{\frac{1}{2}}
=F(\theta_{(0,\frac{1}{2})})$.
Thus, the condition $S(Z_D)=Z_D$ is stronger than the modular invariance of the character.
In fact, this stronger condition is needed in order to define the conformal field theory $F_G$ on genus one surfaces.
\end{rem}

We will prepare some elementary lemma from linear algebra.
For $d \in \Z_2^{r+r}$ and a subgroup $H \subset \Z_2^d$,
set 
\begin{align*}
H^{\perp_d}&=\{\al \in \Z_2^d\;|\; |\al H| \subset 2\Z \}.
\end{align*}
It is clear that $H^{\perp_d} \subset d(H^\perp)=\{d\al\}_{\al \in H^\perp}$.
For $\al \in H^\perp$ and $h \in H$, since $H \subset \Z_2^d$,
$|\al h| = |\al (d h)|=|(d\al)h|$. Thus, $d\al \in H^{\perp_d}$.
Then, we have:
\begin{lem}\label{linear_algebra}
For $d \in \Z_2^{r+r}$ and a subgroup $A \subset \Z_2^{r+r}$,
$d(A^\perp) = (A^d)^{\perp_d}$.
\end{lem}
\begin{proof}
Let $\al \in A^\perp$. Since $A^d =A \cap \Z_2^d \subset A$,
for any $a \in A^d$, $|\al a| \in 2\Z$.
Since $|\al a|=|\al (da)|=|(d\al)a|$, $d\al \in (A^d)^{\perp_d}$,
which implies that $dA^\perp \subset (A^d)^{\perp_d}$.

We will show that $(A^d)^{\perp_d} \subset A^\perp$.
Let $b \in (A^d)^{\perp_d}$.
Let $f_b:A\rightarrow \Z_2$ be a $\Z_2$-linear map
defined by $f_b(\al) = (-1)^{|b\al|}$ for $\al \in A$.
Then, $f_b(A^d)=1$ by definition of $(A^d)^{\perp_d}$.
If $f_b\equiv 1$, then $b \in A^\perp$ and there is nothing to prove.
We will claim that there exits $\ga \in \Z_2^{\dpe}$ such that
$b+\ga \in A^\perp$. If such $\ga$ exists, then $d(b+\ga)=db=b$,
which implies that $(A^d)^{\perp_d} \subset A^\perp$.

Note that $A^d$ is the kernel of the left multiplication
$\dpe*: A \rightarrow \dpe A,\; \al \rightarrow \dpe\al$.
Thus, $A/A^d \cong \dpe A$
and $f_b$ defines a $\Z_2$-linear map $\tilde{f}_b:\dpe A \cong A/A^d \rightarrow \Z_2$.
Since $|-|:\Z_2^{\dpe} \times \Z_2^{\dpe} \rightarrow \Z_2, (\al,\be)\mapsto (-1)^{|\al\be|}$ is a perfect pairing,
by $\dpe A \subset \Z_2^{\dpe}$, there exists $\ga \in \Z_2^{\dpe}$ such that
$\tilde{f}_b(\al)=(-1)^{|\ga \al|}$ for any $\al \in \dpe A$.
For any $\al \in A$,
$
(-1)^{|\al (b+\ga)|}=(-1)^{|\al b|}(-1)^{|\al \ga|}
=f_b(\al)(-1)^{|\al (\dpe \ga)|}
=\tilde{f}_b(\dpe \al) (-1)^{|(\dpe\al) \ga|}=1$.
%
\end{proof}

Then, we have:
\begin{prop}
For any code $G \subset \Z_2^r$ with $(1^r)\in G$, 
$S(Z_G)=Z_G$.
\end{prop}
\begin{proof}
For $\mu \in \Isrr$, define a linear map $\pi_\mu:\C \Isrr \rightarrow \C$ by taking the coefficient of $\theta_\mu$.
By Lemma \ref{character_D},
it suffices to show that for $d',c' \in \Z_2^{r+r}$ with $d'c'=0$,
\begin{align*}
\pi_{(d',c')} S(Z_G)=
\begin{cases}
2^{-|d'|_l} \# C_G^{d'} & \text{ if $d' \in D_G$ and $c' = \dpe'\al$ for some $\al \in C_G$,}\\
0 & \text{otherwise}.
\end{cases}
\end{align*}
By the definition of $S$,
$\pi_{d',c'}(S(\theta_{d,c}))=0$ unless $d' d=(0,\dots,0)$,
i.e., the support of $d$ and $d'$ are disjoint.

Since $\{d \in D_G\;|\; dd'=(0,\dots,0) \}$ is equal to 
$D_G^{\dpe'}$,
by the definition of the map $S$,
\begin{align*}
\pi_{(d',c')} S(\theta_{d,c})&=(-1)^{|d'c|+|dc'|}\sqrt{2}^{-|d'|_l-|d'|_r-|d|_l-|d|_r}2^{-|(d+d')_\perp|_l-|(d+d')_\perp|_r}\\
&=\sqrt{2}^{|d'|_l+|d'|_r}4^{-r}(-1)^{|d'c|+|dc'|} 2^{|d|_l}.
\end{align*}
Thus, by Lemma \ref{character_D},
\begin{align*}
\pi_{(d',c')} S(Z_G)&=\sum_{d \in D_G^{\dpe'}}\sum_{\al \in C_G}
2^{-|d|_l}\pi_{(d',c')}(S(\theta_{d,\dpe \al}))\\
&=\frac{\sqrt{2}^{|d'|_l+|d'|_r}}{4^{r}} \sum_{d\in D_G^{\dpe'}}\sum_{\al \in C_G}(-1)^{|dc'|+|\dpe d'\al|}\\
&=\frac{\sqrt{2}^{|d'|_l+|d'|_r}}{4^{r}} \sum_{d\in D_G^{\dpe'}}\sum_{\al \in C_G}(-1)^{|dc'|+|d'\al|},
\end{align*}
where in the last equality we used $dd'=0$.
Since
\begin{align*}
\sum_{\al \in C_G}(-1)^{|d'\al|}
= \begin{cases}
\# C_G & \text{if $d' \in C_G^\perp$}\\
0 & \text{otherwise},
\end{cases}
\end{align*}
and the similar result for $D_G^{\dpe'}$,
we have
\begin{align*}
\pi_{(d',c')} S(Z_G)= \begin{cases}
2^{|d'|_l-2r} \#C_G \# D_G^{\dpe'} & \text{if $d' \in D_G$ and $c' \in (D_G^{\dpe'})^\perp$}\\
0 & \text{otherwise}.
\end{cases}
\end{align*}
Since $d'c'=0$,
$c' \in (D_G^{\dpe'})^\perp$ is equivalent to
$c' \in (D_G^{\dpe'})^{\perp_{\dpe'}}$.
By Lemma \ref{linear_algebra},
$(D_G^{\dpe'})^{\perp_{\dpe'}}=\dpe' (D_G^\perp)=\dpe' C_G$,
as desired.
Furthermore, by Lemma \ref{linear_algebra},
$\dim D_G^{\dpe'}=2|\dpe'|_l-\dim (D_G^{\dpe'})^{\perp_{\dpe'}}
= 2|\dpe'|_l- \dim \dpe' C_G
=2r-2|d'|_l- \dim C_G + \dim  C_G^{d'}$.
Hence, 
$\pi_{(d',c')} S(Z_G)$ is equal to $2^{-|d'|_l}\# C_G^{d'}$ if $d' \in D_G$ and $c' \in \dpe' C_G$
and zero otherwise.
\end{proof}

Hence, we have:
\begin{cor}
\label{cor_modular}
The character of the full vertex operator algebra $F_G$
satisfies the modular invariance.
\end{cor}

We end this section by determining the dimension of $S_G$.
For a code $G \subset \Z_2^r$,
{\it a code enumerate} is a polynomial 
$P_G(t)=\sum_{k=0}^r N_k(G) t^k  \in \Z[t]$ where the coefficients are defined by
$N_k(G)=\#\{g \in G\;|\; |g|_\Delta =k\}.$

For $d \in D_G$, by Lemma \ref{character_D},
$\dim A_G(d)= \frac{\# C_G}{2^{|d|_l}}$.

Thus, we have:
\begin{prop}
The dimension of the framed algebra $S_G$ is equal to $2^{2r - \dim G}P_G(\frac{1}{2})$.
\end{prop}

\subsection{Examples and classification of codes}\label{sec_classify_code}
In this section, we consider examples of the code conformal field theories constructed in Theorem \ref{construction}.
By Proposition \ref{orthogonal},
it suffices to study the case that the code $G \subset \Z_2^r$
is indecomposable. We will also classify indecomposable codes
up to $r \leq 6$.

Let $\langle \al_1,\dots,\al_k\rangle$ be the code generated by $\al_1,\dots,\al_k \in \Z_2^r$.
%
We introduce the following codes:
Let $G_r^{\min}$ be the subcode of $\Z_2^r$ generated by $(1^r)=(1,1,\dots,1)$ and
$G_r^{\mathrm{even}}$ is the subcode of $\Z_2^r$ consisting of
all even codewords if $r$ is even.
Note that 
$G_r^{\mathrm{even}}$ is generated by $r-1$ codewords, $(1,1,0,\dots,0), (0,1,1,0,\dots,0),\dots ,(0,\dots,0,1,1)$
and contains $\rreg$ since $r$ is even.

We first study $G_r^{\min}$. Set $D_r^{\min}=D_{G_r^{\min}}$ and $C_r^{\min} =(D_r^{\min})^\perp$.
Then, $C_r^{\min}=C_{r,r}^\eve$, the maximal even code in $\Z_2^{r+r}$.
We note that $G_r^{\min}$ is the minimal code which contains $1^r\in \Z_2^r$.
Thus, $G_r^{\min}$ maximizes the dual code $C_r^{\min}$.
The code enumerate and the dimension are in this case 
\begin{align*}
P_{G_r^{\min}}(t)&=1+t^r\\
\dim S_{G_r^{\min}} &= 2^{r-1}(2^r+1)
\end{align*}
and $S_{G_r^{\min}} = A_{G_r^{\min}}(0) \oplus A_{G_r^{\min}}(\rreg)$,
where $A_{G_r^{\min}}(0)=\C[\hat{C_r^{\eve}}]$ is a twisted group algebra and
$A_{G_r^{\min}}(\rreg)$ is a $2^{r-1}$-dimensional representation.
By proposition \ref{chiral},
the chiral part of the theory is isomorphic to
$F_{\C[\hat{C_r^{\eve}}]}$ and the dimension of the Lie algebra $(F_{G_r^{\min}})_{1,0}$ is $\#C_r^{\eve}=\binom{r}{2}$,
which is equal to the dimension of $\mathrm{SO}(r)$.
In fact, by Proposition \ref{classify_code_vertex},
they are isomorphic.
Moreover, we can identify the full vertex algebra $F_{G_r^{\min}}$
as the algebra of $\mathrm{SO}(r)$ WZW model at level one for $r \geq 4$.
We remark that if $r=3$, then $F_{G_3^{\min}}$ is $\mathrm{SU}(2)$ WZW model at level two.

Important informations of a code conformal field theory $F_G$
are the dimension of the framed algebra $S_G$
and the structure of the Lie algebra $(F_G)_{1,0} \cong (F_G)_{0,1}$, which we call {\it currents}.
In this section, we classify codes $G$ and determine $\dim S_G$ and $(F_G)_{1,0}$.

For example, for $r \geq 2\Z_{\geq 1}$,
the code enumerate of $G_r^{\eve}$ and the dimension of $S_{G_r^{\eve}}$ are
\begin{align*}
P_{G_r^{\eve}}(t)&=\frac{(1+t)^r+(1-t)^r}{2}\\
\dim S_{G_r^{\eve}} &= 3^r+1.
\end{align*}
In this case, there are no currents if $r \geq 4$
since $(G_r^{\eve})^\perp$ is $\langle 1^r\rangle$.

Now, we will classify codes.
The following lemma is clear:
\begin{lem}
\label{decomposable_one}
If a code $G$ contains a codeword $\al \in G$ with $|\al|_\Delta=1$,
then $G$ is decomposable, in fact, $G=\langle \al \rangle \perp G' \subset \Z_2 \perp \Z_2^{r-1}$,
where $G'=\{\be \in G\;|\; |\al\be|_\Delta =0 \}$.
\end{lem}

If the dimension or the codimension of $G$ is
low, then the classification is easy.
In fact, we have:
\begin{prop}
\label{code_low}
Let $G \subset \Z_2^r$ be a code with $1^r \in G$.
Then, the following properties are hold:
\begin{enumerate}
\item
If $\dim G = 1$, then $G=\langle 1^r \rangle$;
\item
If $\dim G = r$, then $G=\Z_2^r$;
\item
If $\dim G=2$, then $G$ is decomposable;
\item
If $\dim G = r-1$ and $G$ is indecomposable,
then $r$ is even and $G=G_{\eve}^r$.
\end{enumerate}
\end{prop}
\begin{proof}
(1), (2) and (3) are clear.
Assume that $\dim G = r-1$.
Then, $\dim G^\perp=1$. Let $\al \in \Z_2^r$ generate
$G^\perp$ and set $|\al|_\Delta=k$.
We may assume that $\al=(1^k0^{r-k})$.
If $k < r$, then $(0^{r-1}1^1) \in (G^\perp)^\perp=G$.
Thus, by Lemma \ref{decomposable_one},
if $G$ is indecomposable, $k=r$ and thus $G=\langle 1^r \rangle^\perp$. Furthermore, since $1^r \in G$, $r$ must be even.
\end{proof}

All codes with $r\leq 3$ are listed in Table \ref{table_low}:
\begin{table}[h]
\label{table_low}
\caption{all code CFTs of low rank}
  \begin{tabular}{|l|c|c|c|c|} \hline
$r$ & code $G$ & current & $\dim S_G$ & CFT \\ \hline \hline
1 & $\langle1\rangle$ &  $0$ & $3$ & Ising\\ \hline
2 & $\langle10,01\rangle$ & $0$ & $9$ & $\text{Ising}^{\otimes 2}$\\
 & $\langle 11 \rangle$ & $U(1)$ & $10$ & Dirac \\ \hline
3 & $\langle100,010,001\rangle$ & $0$ & $27$ & $\text{Ising}^{\otimes 3}$\\
 & $\langle100,011 \rangle$ & $U(1)$ & $30$ & $\text{Ising}\otimes \text{Dirac}$\\
 & $\langle 111 \rangle$ & $\mathrm{SU}(2)$ & $36$ & $\mathrm{SU}(2)$ WZW-model at level $2$\\
 \hline
\end{tabular}
\end{table}
%

Now, we consider an indecomposable code $G$ with $\dim G=3$ for any $r \in \Z_{>0}$. Let $\al_1 \in G$ be the non-zero shortest codeword
and set $a=|\al_1|_\Delta$, that is, $a = \min_{\al \in G, \al \neq 0}|\al|_\Delta$.
Without loss of generality, we may assume that
$\al_1=(1^a 0^{r-a}) \in \Z_2^r$, the first $a$ components are $1$.
Then, $G$ is generated by $\al_1=(1^a 0^{r-a})$ and $\al_2=(0^a 1^{r-a})$ and some element $\al_3 \in \Z_2^r$.
We may assume that 
$\al_3=(1^b0^{a-b}1^c0^{r-a-c})$ for some $a \geq b \geq 0$
and $r-a \geq c \geq 0$.
By adding $\al_1$ and $\al_2$,
we may further assume that 
$\frac{a}{2} \geq b$ and $\frac{r-a}{2} \geq c$.
If $b=0$,
then the code $G$ is decomposed into $\langle(1^a)\rangle
\perp \langle (1^c0^{r-a-c}), (1^{r-a}) \rangle \subset \Z_2^a \Z_2^{r-a}$. Thus, $b>0$ and similarly $c>0$.
Hence, by Lemma \ref{decomposable_one},
we can assume that $a,b,c$ satisfy
\begin{align}
\frac{r}{2} \geq b+c \geq a \geq 2,\;\;\;\frac{a}{2} \geq b \geq 1,\;\;\;\frac{r-a}{2} \geq c \geq 1.
\label{eq_abc}
\end{align}
Then, such integers $a,b,c$ uniquely determine the code $G$.
We denote it by $G_{a;b,c}^r$.

\begin{prop}
\label{rank3}
If $G$ is a indecomposable code with $\dim G=3$,
then $G = G_{a;b,c}^r$ for some $a,b,c \in \Z_{\geq 0}$
satisfying \eqref{eq_abc}.
\end{prop}
In the case of $r=5$,
all possible $a,b,c \in \Z_{\geq 0}$
satisfy \eqref{eq_abc} are only $(a,b,c)=(2,1,1)$.
Thus, by Lemma \ref{decomposable_one},
all possible indecomposable codes $G \subset \Z_2^5$
with $(1^5) \in G$
are $G_{2;1,1}^5$ and $\langle11111\rangle$.

In the case of $r=6$,
all possible $a,b,c \in \Z_{\geq 0}$
satisfy \eqref{eq_abc} are $(a,b,c)=(2,1,2), (2,1,1)$.
Thus,
$G=G_{2;1,2}^6=\langle (110000), (001111), (101100) \rangle$
or $G=G_{2;1,1}^6=\langle (110000), (001111), (101000) \rangle$.

Finally, we consider a indecomposable code $G$ with $\dim G=r-2$.
In this case, $\dim G^\perp =2$. It is clear that
$G^\perp$ is an indecomposable code.
We note that the condition $(1^r) \in G$ 
is equivalent to the condition that $G^\perp$ is even.
Similarly to the above,
any two dimensional code is
generated by
$(1^a0^{r-a}), (1^b0^{a-b}1^c0^{r-a-c})$.
If $r-a-c >0$, then $(0^{r-1}1^1) \in G$,
which contradicts to the condition that $G$ is
indecomposable.
Set
$$
E_{a;b}^r = \langle (1^a0^{r-a}), (1^b0^{a-b}1^{r-a})\rangle
\subset \Z_2^r
$$
for $r>a \geq b \geq 0$.
Then, $G^\perp$ is isomorphic to $E_{a;b}^r$ for some $a,b \in \Z$. 
We note that
$E_{a;b}^r$ consists of
$
\{(0^r), (1^a0^{r-a}), (1^b0^{a-b}1^{r-a}), (0^b1^{a-b}1^{r-a})\}$.
Since $G^\perp$ is even,
$a,b+r-a,a-b+r-a \in 2\Z$.
Since $G^\perp$ is indecomposable,
$r>a > b > 0$.
We may assume that 
$(1^a0^{r-a})$ has the minimal length and $\frac{a}{2}\leq b$.
Thus, $a,b,c$ satisfies
\begin{align}
a,b-r \in 2\Z,\;\; r> a >b>0,\;\; \frac{a}{2} \geq b,\;\; b+r \geq 2a,\;\; r-b \geq  a.\label{eq_abc2}
\end{align}
Hence, we have:
\begin{prop}
If $G$ is a indecomposable code with $\dim G=r-2$,
then $G = (E_{a;b}^r)^\perp$ for some $a,b \in \Z_{\geq 0}$
satisfying \eqref{eq_abc2}.
\end{prop}
In the case of $r=6$,
all possible $a,b,c \in \Z_{\geq 0}$
satisfy \eqref{eq_abc2} are $(a,b)=(4,2)$.
Then, the weight enumerator of $(E_{4;2}^6)^\perp$
is 
$$P_{(E_{4,2}^6)^\perp}(t)=1+3t^2+8t^3+3t^4+t^6.$$
By an easy computation, we have:
\begin{lem}
For $r\geq 5$ and $k \geq 1$,
\begin{align*}
P_{G_{2;1,k}^r}(t)&=1+t^2+2t^{k+1}+2t^{r-k-1}+t^{r-2}+t^r.
\end{align*}
\end{lem}
We summarize the property of these codes in Table \ref{table_code_information}:
\begin{table}[h]
\caption{Some series of code CFTs}
\label{table_code_information}
  \begin{tabular}{|l|c|c|c|c|} \hline
code & enumerator & $\dim S_G$ & current \\ \hline
$G_r^\regt$ & $1+t^r$ & $2^{r-1}(2^r+1)$ & $\mathrm{SO}(r)$\\
$G_r^{\mathrm{even}}$ & $\frac{(1+t)^r+(1-t)^r}{2}$ & $3^r+1$
& $0$ ($r \geq 4$)\\
$G_r^{2;1,1}$ & $1+3t^2+3t^{r-2}+t^r$& $2^{r-3}(7\cdot 2^{r-2}+13)$& $\mathrm{SO}(r-3)$ ($r \geq 4$) \\
$G_r^{2;1,2}$ & $1+t^2+2t^3+2t^{r-3}+t^{r-2}+t^r$& 
$2^{r-3}(12 \cdot 2^{r-3}+21)$& $\mathrm{U}(1) \times \mathrm{SO}(r-4)$ ($r \geq 4$) \\  \hline
\end{tabular}
\end{table}

All indecomposable codes with $r\leq 6$ are listed in Table \ref{table_code_all}
\begin{table}[h]
\caption{all indecomposable code CFTs of rank 4,5,6}
\label{table_code_all}
  \begin{tabular}{|l|c|c|c|c|} \hline
$r$ & code & current & $\dim S_G$ &  \\ \hline \hline
4 & $\langle 1111 \rangle$ & $\mathrm{SO}(4)$ & 136 & $G_4^\regt$ \\
 & $\langle1111\rangle^\perp $ & $0$ & $82$ & $G_4^\text{even}$ \\ \hline
5 & $\langle11111\rangle$ &$\mathrm{SO}(5)$& 528 & $G_5^\regt$\\
 & $\langle11000,00111,01100\rangle$ & $\mathrm{U}(1)$ & $276$ & $G_5^{2;1,1}$ \\ \hline
6 & $\langle11111\rangle$ &$\mathrm{SO}(6)$& 2080 & $G_6^\regt$ \\
& $\langle110000,001111,101000\rangle$& $\mathrm{SO}(3)$ & $1000$ & $G_6^{2;1,1}$\\
& $\langle110000,001111,101100\rangle$&$\mathrm{U}(1)^2$ & $936$ & $G_6^{2;1,2}$ \\ 
& $\langle110000,001100,000011,101010\rangle$ &$0$ & $756$ & $(E_6^{4;2})^\perp$\\
 & $\langle111111\rangle^\perp$& 0 & 730 & $G_6^\text{even}$ \\
\hline
\end{tabular}
\end{table}

\begin{rem}
From the table, it seems that for given $r \in \Z_{>0}$
the smallest (resp. the second smallest) dimension of $S_G$ is $3^r$ (resp. $3^r+1$ if $r$ is even)
which is given by the CFT $\mathrm{Ising}^{\otimes r}$ (resp. $F_{G_\eve^r}$).
Furthermore, it seems that the largest dimension of $S_G$ is $2^{r-1}(2^{r}+1)$
which is given by the $\mathrm{SO}(r)$-WZW model.
\end{rem}

\section{Deformation of code CFTs}
In our previous paper \cite{M3},
we construct an exactly marginal deformation of a conformal field 
theory in terms of a full vertex algebra.
The deformation is called a current-current deformation in physics.
In this section, we study the current-current deformation of code CFTs and calculate deformed four point correlation functions.
In Section \ref{sec_cc_def}, we recall the current-current deformation of a full vertex algebra in a general setting
and in Section \ref{sec_cc_class} apply it to code CFTs.
\subsection{Current-current deformation of full vertex algebras} \label{sec_cc_def}
In this section, we recall the definition of the current-current deformation of a full vertex algebra from \cite{M3}.

Let $H_l$ and $H_r$ be real finite dimensional vector subspaces equipped with
non-degenerate symmetric bilinear forms $(-,-)_l:H_l \times H_l \rightarrow \R$ and $(-,-)_l:H_r \times H_r \rightarrow \R$.
Let $M_{H_l}(0)$ and $M_{H_r}(0)$ be affine Heisenberg vertex algebras
associated with $(H_l,(-,-)_l)$ and $(H_r,(-,-)_r)$.
Set $H=H_l \oplus H_r$ and let $p, \overline{p}:H \rightarrow H$
 be projections on $H_l$ and $H_r$
and 
$$
M_{H,p}=M_{H_l}(0)\otimes \overline{M_{H_r}(0)},
$$
the tensor product of the vertex algebra $M_{H_l}(0)$
and the conjugate vertex algebra $\overline{M_{H_r}(0)}$
(see Proposition \ref{conjugate} and Proposition \ref{full_tensor}).

In this section, we consider a class of a full vertex algebra
which is an $M_{H,p}$-module (like an algebra over a ring).
More precisely, let $F$ be a full vertex algebra
and we assume that $M_{H,p}$ is a subalgebra of $F$, $M_{H,p} \subset F$.
Then, since $H_l = (M_{H,p})_{1,0}$ and $H_r = (M_{H,p})_{0,1}$,
$F \subset F_{1,0}$ and $H_r \subset F_{0,1}$.

We note that the subspaces $H_l$ and $H_r$ satisfy the following conditions:
For any $h_l,h_l' \in H_l$ and $h_r,h_r' \in H_r$,
\begin{enumerate}
\item[H1)]
$H_l \subset F_{1,0}$ and $H_r \subset F_{0,1}$;
\item[H2)]
 $\D H_l=0$ and $D H_r=0$;
\item[H3)]
$h_l(1,-1)h_l'=(h_l,h_l')_l \va$, $h_r(-1,1)h_r'=(h_r,h_r')_r \va$;
\item[H4)]
$h_l(n,-1)h_l'=0$, $h_r(-1,n)h_r'=0$
 for any $n=0$ or $n \in \Z_{\geq 2}$.
\end{enumerate}

Since $h_l \in H_l$ is a holomorphic vector, by Lemma \ref{hol_commutator},
$Y(h_l,\uz)=\sum_{n\in \Z} h_l(n,-1)z^{-n-1}$.
For $h \in H$, set
$$
h[n]_0 = (ph)(n,-1)+(\p h)(-1,n).
$$
By Lemma \ref{hol_commutator} and Lemma \ref{hol_commute},
for any $h,h' \in H$,
\begin{align*}
[h[n]_0,h'[m]_0]=\left((ph,ph')_l+(\p h,\p h')_r\right)\delta_{n+m,0}.
\end{align*}
Thus, $\{h[n]_0\}_{h\in H, n\in \Z}$ defines an action of the
affine Heisenberg Lie algebra $\hat{H_l\oplus H_r}$ on $F$.

For $\al \in H$, set
$$
F^\al=\{v \in F\;|\; h[0]_0 v=\left((ph,p\al)_l+(\p h,\p \al)_r\right)v \text{ for any } h \in H\}.
$$

{\it A full $\mathcal{H}$-vertex operator algebra}, denoted by $(F,H,p,\om,\omb)$, is a full vertex operator algebra $(F,\om,\omb)$ with a subalgebra $M_{H,p}$
such that 
\begin{enumerate}
\item[FHO1)]
$h[0]_0$ are semisimple on $F$ with real eigenvalues for any $h \in H$;
\item[FHO2)]
For any $\al \in H$, there exists $N \in \R$ such that
$F_{t,\td} \cap F^\al=0$ for $t\leq N$ or $\td \leq N$;
\item[FHO3)]
$L(n)H_l=0$ and $\Ld(n)H_r=0$ for any $n\geq 1$.
\end{enumerate}
\begin{rem}
\label{rem_Cartan}
For good full vertex algebras, the split Cartan subalgebras of the Lie algebras $F_{1,0}$, $F_{0,1}$
define a full $\mathcal{H}$-vertex algebra structure on $F$.
\end{rem}

Let $(F,Y,H,p,\om,\omb)$ be a full $\mathcal{H}$-vertex algebra.
By (FHO1) and (FHO2) and the representation theory of the affine Heisenberg Lie algebra (\cite[Theorem 1.7.3]{FLM}), $F$ is generated by lowest weight vectors as a module of the affine Heisenberg Lie algebra.

More precisely, let $\Omega_{F,H}^\al$ be the set of all vectors $v \in F^\al$ such that
\begin{enumerate}
\item
$h[n]_0 v=0$ for any $h \in H$ and $n \geq 1$
\end{enumerate}
and set $$\Omega_{F,H}=\bigoplus_{\al \in H} \Omega_{F,H}^\al.$$
Then, $F$ is isomorphic to $M_{H,p}\otimes \Om_{F,H}=\bigoplus_{\al \in H} M_{H,p} \otimes \Omega_{F,H}^\al$ as an $M_{H,p}$-module.

For $h_1,h_2 \in H$,
define $z^{ph_1[0]_0}\z^{\p h_2[0]_0} \in \End F[z^\R,\z^\R]$ by
$$z^{ph_1[0]_0}\z^{\p h_2[0]_0}v=z^{(ph_1,p \al)_l}\z^{(\p h_2,\p \al)_r}v$$
for $\al \in H$ and $v \in F^\al$.

Define a bilinear form $(-,-)_\lat$ on $H$ by 
$$(\al,\be)_\lat=(p\al,p\be)_l - (\p\al,\p\be)_r$$
for $\al,\be \in H$
and let $\mathrm{O}(H_l\oplus -H_r)$
be the orthogonal group on the space $(H,(-,-)_\lat)$.
\begin{thm}{\cite[Section 5]{M3}}
\label{deformation}
Let $(F,Y,H,p,\om,\omb)$ be a full $\mathcal{H}$-vertex algebra.
There exists a family of vertex operators on $F$ parametrized by $\si \in \mathrm{O}(H_l\oplus -H_r)$
$$
Y_\si(-,\uz):F \rightarrow \End F[[z,\z,|z|^\R]]
$$
such that 
for any $\si \in \mathrm{O}(H_l\oplus -H_r)$,
\begin{enumerate}
\item
$(F,Y_\si,\bm{1},\om,\omb)$ is a full vertex operator algebra;
\item
For any $h \in H$,
$$Y_\si(h,\uz)=Y(h,\uz)+(p\si^{-1}h-ph)[0]_0 z^{-1}+(\p\si^{-1}h-\p h)[0]_0 \z^{-1};
$$
\item
For any $\al \in H$ and $v \in \Om_{F,H}^\al$,
\begin{align*}
Y_\si(v,\uz)&=\exp\Bigl(\sum_{n \geq 1} (p\si^{-1}\al-p\al)[-n]_0\frac{z^n}{n}+
(\p\si^{-1}\al-\p\al)[-n]_0\frac{\z^n}{n}\Bigr)\\
&Y(v,\uz) \exp\Bigl(\sum_{n \geq 1} (p\si^{-1}\al-p\al)[n]_0\frac{z^{-n}}{-n}+
(\p\si^{-1}\al-\p\al)[n]_0\frac{\z^{-n}}{-n}\Bigr)z^{(p\si^{-1}\al-p\al)[0]_0}\z^{(\p\si^{-1}\al-\p\al)[0]_0};
\end{align*}
\end{enumerate}
Furthermore, for any $t,\td \in \R$,
the $(L(0),\Ld(0))$-weight of $v \in \Om_{F,H}^\al \cap F_{t,\td}$
with respect to the full vertex operator algebra structure $(F,Y_\si,\bm{1},\om,\omb)$
is 
$$\left(t+\frac{(p\si^{-1}\al,p\si^{-1}\al)_l-(p\al,p\al)_l}{2},
\td+\frac{(\p\si^{-1}\al,\p\si^{-1}\al)_r-(\p\al,\p\al)_r}{2}\right)$$
and for any $\langle -\rangle :F\rightarrow \C$ such that $\langle h(-n) -\rangle$ for $n\geq 0$,
the four point correlation function for $\al_i \in H$ and $v_i\in \Om_{F,H}^{\al_i}$ satisfies
\begin{align*}
\langle &Y_\si(v_1,\uz_1)Y_\si(v_2,\uz_2)Y_\si(v_3,\uz_3)Y_\si(v_4,\uz_4)\1 \rangle\\
&=\Pi_{1 \leq i <j\leq 4}\left((z_i-z_j)(\z_i-\z_j)\right)^{(p\si^{-1}\al_i,p\si^{-1}\al_j)_l- (p\al_i,p\al_j)_l}
\langle Y(v_1,\uz_1)Y(v_2,\uz_2)Y(v_3,\uz_3)Y(v_4,\uz_4)\1 \rangle.
\end{align*}
\end{thm}

The family of full vertex operator algebras constructed in the above theorem
is called {\it the current-current deformation} of a full $\mathcal{H}$-vertex algebra in \cite{M3}.
We gives some remarks.
First, if $\si=\mathrm{id} \in \mathrm{O}(H_l\oplus -H_r)$,
then $\si\al-\al=0$ and $Y_\si(-,\uz)=Y(-,\uz)$.
Furthermore, if $\si \in \mathrm{O}(H_l) \times \mathrm{O}(H_r) \subset \mathrm{O}(H_l\oplus -H_r)$,
then $\si$ commutes with the projection $p$
and $(p\si^{-1}\al,p\si^{-1}\al)_l=(p\al,p\al)_l$, so nothing essentially changes.
In fact, the isomorphic classes of the current-current deformation are parametrized by
the double coset
\begin{align}
D_{F,H} \backslash \mathrm{O}(H_l\oplus -H_r) / \mathrm{O}(H_l) \times \mathrm{O}(H_r)
\label{eq_double},
\end{align}
where $D_{F,H}$ is a subgroup of $\mathrm{O}(H_l\oplus -H_r)$
introduced in \cite{M3} and is called {\it the duality group} (for more detail see \cite[Theorem 5.5]{M3}).

It is noteworthy that
since
$(z_i-z_j)(\z_i-\z_j) \in \R$ for any point of the configuration space $X_4=\{(z_1,z_2,z_3,z_4)\in (\CP)^4\;|\;z_i\neq z_j\}$,
\begin{align*}
&\left((z_i-z_j)(\z_i-\z_j)\right)^{(p\si^{-1}\al_i,p\si^{-1}\al_j)_l- (p\al_i,p\al_j)_l}\\
&=\exp\left( \left((p\si^{-1}\al_i,p\si^{-1}\al_j)_l- (p\al_i,p\al_j)_l\right)\mathrm{log}\left((z_i-z_j)(\z_i-\z_j)\right)\right)
\end{align*}
is a single-valued function on $X_4$.

\begin{rem}
The lowest weight space $\Om_{F,H}$ of a full 
$\mathcal{H}$-algebra $i: M_{H,p} \hookrightarrow F$
 is similar to the lowest weight space $S_F$ for a framed full vertex operator algebra $i:L_{l,r}(0) \hookrightarrow F$.
As with the framed algebra structure on $S_F$,
$\Om_{F,H}$ inherits an algebra structure which is
called a generalized full vertex algebra.
This algebra structure plays a key role in the construction of the
current-current deformation.
\end{rem}

\subsection{Current-current deformation of code CFTs} \label{sec_cc_class}
Let $G \subset \Z_2^r$ be a code such that $1^r \in G$
and $G^\perp$ the dual code.
Set $G^\perp[2]=\{\al \in G^\perp\;|\; |\al|_\Delta =2 \}$.
By Proposition \ref{chiral}, $G^\perp[2]$ determine the Lie algebra $(F_G)_{1,0}\cong (F_G)_{0,1}$.

Let $\al^1,\dots,\al^N \in G^\perp[2]$ be the maximal mutually orthogonal codewords,
i.e., $|\al^i\al^j|_\Delta=0$ for any $i \neq j$
and set $\al_l^i=(\al^i,0)\in \Z_2^{r+r}$ and $\al_r^i=(0,\al^i) \in \Z_2^{r+r}$ for $i=1,2,\dots,N$.
Then, $\al_l^i$ and $\al_r^i$ are in $C_G$. Let $\{e_\al \in \C[\hat{C_G}]\}_{\al \in C_G}$ be the basis of the twisted group algebra given explicitly in Section \ref{sec_construction}.

Then, by Proposition \ref{classify_code_vertex},
each $\al^i \in G^\perp[2]$ defines subalgebras of $F_G$ which
are isomorphic to the lattice vertex operator algebra $V_{2\Z}$
and its conjugate $\overline{V_{2\Z}}$.
Furthermore, $(V_{2\Z})_{1,0}$ is spanned by $e_{\al_l^i}$
and $(\overline{V_{2\Z}})_{0,1}$ is spanned by $e_{\al_r^i}$.

We note that by definition $e_{\al_l^i}\cdot e_{\al_l^i} = (-1)^{\frac{|\al_l^i|}{2}}=-1$ in $\C[\hat{C_G}]$.
Set $h_l^i=\sqrt{-1} e_{\al_l^i}$ and $h_r^i=\sqrt{-1}e_{\al_r^i}$ for $i=1,\dots,N$.
Then, $\{h_l^i\}_{i=1,\dots,N}$ satisfy
\begin{enumerate}
\item
$h_l^i(n,-1)h_l^j=0$ for any $n\geq 2$ or $n=0$;
\item
$h_l^i(1,-1)h_l^j=\delta_{i,j}$
\end{enumerate}
and the similar result holds for $\{h_r^i\}_{i=1,\dots,N}$.
Let $H_l^N$ (resp. $H_r^N$)
be the subspace of $F_G$ spanned by $\{h_l^i\}_{i=1,2,\dots,N}$ 
(resp. $\{h_r^i \}_{i=1,2,\dots,N}$)
and define a bilinear form on $H_l^N$ (resp. $H_r^N$)
by $(h_l^i,h_l^j)_l=\delta_{i,j}$ (resp. $(h_r^i,h_r^j)_r=\delta_{i,j}$).
Set $H^N=H_l^N\oplus H_r^N$ and let $p \in \End H^N$ be the projection onto $H_l^N$.
Then, we have a injective homomorphism 
$M_{H^N,p} \hookrightarrow F_G$.
We note that this embedding factors through
\begin{align}
M_{H^N,p} \hookrightarrow V_{(2\Z)^N}\otimes \overline{V_{(2\Z)^N}} \hookrightarrow F_G
\label{eq_factor}
\end{align}

We will show that $M_{H^N,p} \hookrightarrow F_G$ is a full $\mathcal{H}$-vertex operator algebra.
Since $(F_G)_{t,\td}=0$ if $t <0$ or $\td <0$, (FHO2) is clear.
From an easy computation, (FHO3) follows.
(FHO1) follows from \eqref{eq_factor} and the representation theory of a lattice vertex operator algebra (see for example \cite{LL}).
Hence, we have:
\begin{prop}
\label{code_H}
If $\al^1,\dots,\al^N \in G^\perp[2]$ are mutually orthogonal codewords,
then $M_{H^N,p} \hookrightarrow F_G$ is a full $\mathcal{H}$-vertex operator algebra
and the code CFT $F_G$ admits the current-current deformation parametrized by
$\mathrm{O}(N,N)/\mathrm{O}(N)\times \mathrm{O}(N)$.
\end{prop}

In order to calculate the deformed four point correlation function,
we need to decompose $F_G$ as an $V_{(2\Z)^N}\otimes \overline{V_{(2\Z)^\N}}$-module.
For the sake of simplicity, we first consider the case of $N=1$
and set $\al=\al^1=(1100\dots 0)\in \Z_2^r$ and $h_l=h_l^1, h_r=h_r^1$.
For $\epsilon,\epsilon'=\pm$, set
$$
S_G^{\epsilon,\epsilon'}=
\{h_l\cdot v=\epsilon v, h_r\cdot v=\epsilon' v \}.
$$
Since $h_l \cdot h_l =1$ and $h_r \cdot h_r =1$ in $\C[\hat{C_G}]$ and $S_G$ is a $\C[\hat{C_G}]$-module
by Lemma \ref{simple_current},
\begin{align}
S_G=\bigoplus_{\epsilon,\epsilon'=\pm} S_G^{\epsilon,\epsilon'}.\label{eq_epsilon}
\end{align}

Recall that 
$$F_G=\bigoplus_{(d,c)}L_{r,r}(d,c)\otimes (S_G)_{d,c}.$$
Let $d=(d_1,d_2,\dots,d_r,\bar{d}_1,\dots,\bar{d}_r) \in \Z_2^{r+r}$
and $c=(c_1,c_2,\dots,c_r,\bar{c}_1,\dots,\bar{c}_r) \in \Z_2^{r+r}$ satisfy $(S_G)_{d,c} \neq 0$.
Set $\hat{d}=(d_3,d_4,\dots,d_r,\bar{d}_3,\dots,\bar{d}_r)\in \Z_2^{r+r-4}$
and $\hat{c}=(c_3,c_4,\dots,c_r,\bar{c}_3,\dots,\bar{c}_r)\in \Z_2^{r+r-4}$.
Then, by Proposition \ref{even_code},
$|d\al_l|, |d\al_r| \in 2\Z$.
In fact, by construction, $(d_1,d_2,\bar{d}_1,\bar{d}_2)=(0,0,0,0)$ or $(1,1,1,1)$.

Assume that $(d_1,d_2,\bar{d}_1,\bar{d}_2)=(1,1,1,1)$.
Then, by the fusion rule,
$(S_G)_{d,c}$ is stable under the $\cdot$-multiplication by $h_l$ and $h_r$.
Set $(S_G)_{d,c}^{\epsilon,\epsilon'}=S_G^{\epsilon,\epsilon'}\cap (S_G)_{d,c}$.
Then, $(S_G)_{d,c}=\bigoplus_{\epsilon,\epsilon'=\pm} (S_G)_{d,c}^{\epsilon,\epsilon'}$.
It is easy to show that by \cite[Lemma 3.1]{DGH}
$(S_G)_{d,c}^{\epsilon,\epsilon'}$ is a direct sum of the copies of $V_{2\Z +\epsilon \frac{1}{2}}\otimes \overline{V_{2\Z +\epsilon' \ft}}$.
In fact,
\begin{align*}
L_{r,r}(d,c)\otimes (S_G)_{d,c}^{\epsilon,\epsilon'}
\cong 
V_{2\Z +\epsilon \frac{1}{2}}\otimes \overline{V_{2\Z +\epsilon' \ft}}
\otimes L_{r-2,r-2}(\hat{d},\hat{c})\otimes (S_G)_{d,c}^{\epsilon,\epsilon'}
\end{align*}
as a $V_{2\Z}\otimes \overline{V_{2\Z}}
\otimes L_{r-2,r-2}(0)$-module.

\begin{lem}
\label{average}
Assume that $(d_1,d_2,\bar{d}_1,\bar{d}_2)=(1,1,1,1)$.
Then, $\tilde{t}_d = (1+h_l)\cdot t_d$ satisfies
\begin{enumerate}
\item
$h_l [n]_0 \tilde{t}_d =0=h_r [n]_0 \tilde{t}_d$ for any $n\geq  1$;
\item
$h_l[0]_0 \tilde{t}_d = \tilde{t}_d = h_r[0]_0 \tilde{t}_d$.
\end{enumerate}
In particular, $\tilde{t}_d=(1+h_l)\cdot t_d \in \Om_{F_G,\R h_l\oplus \R h_r}^{\ft(h_l+h_r)}$.
Furthermore, $(1-h_l)\cdot t_d\in \Om_{F_G,\R h_l\oplus \R h_r}^{-\ft(h_l+h_r)}$.
\end{lem}
\begin{proof}
(1) follows from the representation theory of the lattice vertex algebra $V_{2\Z}$.
Since $h_l\cdot h_l=1$, we have $h_l\cdot (1+h_l)\cdot t_d=(h_l+1)\cdot t_d$.
Similarly, we have $h_r\cdot (1+h_l)\cdot t_d= (h_r+h_r\cdot h_l)\cdot t_d
=(h_r-1\epsilon{\al_l,\al_r}e_{\al_l+\al_r})t_d$.
Since ${\al_l+\al_r} \in \Delta^d$, by the definition of $A_G(d)$, we have
$h_r\cdot (1+h_l)\cdot t_d=(h_r-\epsilon{\al_l,\al_r})t_d$.
By definition of the two cocycle $\epsilon(-,-)$, 
$\epsilon(\al_l,\al_r)=\epsilon(\al_l,\al_l+\Delta(1100\dots))=\epsilon(\al_l,\al_l)\epsilon(\al_l, \Delta(1100\dots))=-1$.
Hence, $h_r\cdot t_d=h_r\cdot e_{\al_l+\al_r}\cdot t_d = h_l \cdot t_d$.
Thus, $h_r\cdot (1+h_l)\cdot t_d=(1+h_l)\cdot t_d$.
The case of $(1-h_l)\cdot t_d$ can be shown similarly.
\end{proof}

\begin{rem}
Readers may wonder that the eigenvalues of $h_l(0,-1)$ and $h_r(-1,0)$ on 
$V_{2\Z +\epsilon \frac{1}{2}}\otimes \overline{V_{2\Z +\epsilon' \ft}}$
are $2\Z +\epsilon \frac{1}{2}$ and $2\Z +\epsilon' \ft$,
however the eigenvalues of the $\cdot$-multiplication of $h_l$ and $h_r$ on $S_G$ are $\pm$.
This difference happens because of the normalization \eqref{eq_normal}
for the intertwining operator $I_{\ft,\fs}^{\fs}(-,z)$.
\end{rem}

We briefly discuss in the case of $(d_1,d_2,\bar{d}_1,\bar{d}_2)=(0,0,0,0)$.
There are $2^4$ possibilities for $(c_1,c_2,\bar{c}_1,\bar{c}_2) \in \Z_2^4$
and for the sake of simplicity we only consider the holomorphic part $(c_1,c_2)$.

Assume that $(c_1,c_2)$ is equal to $(0,1)$ or $(1,0)$.
Then, $(S_G)_{0^{2r},(01**)}\oplus (S_G)_{0^{2r},(10**)}$ is stable under the $\cdot$-multiplication of $h_l$
and the $S_G^{+,*}$ part (resp. $S_G^{-,*}$-part) corresponds to $V_{2\Z+\ft}$ (resp. $V_{2\Z-\ft}$).

Finally, assume that $(c_1,c_2)$ is equal to $(0,0)$ or $(1,1)$.
In this case, $h_l$ does not act as the zero-mode.
More specifically, let $v \in (S_G)_{0^{2r},(00**)}$.
Then, $h_l(-1,-1)v$ is again in $(S_G)_{0^{2r},(11**)}$
and thus $h_l\cdot v = h_l(-1,-1)v$.
Similarly, for $w \in (S_G)_{0^{2r},(11**)}$,
$h_l(1,-1)w$ is in $(S_G)_{0^{2r},(00**)}$ and $h_l\cdot w = h_l(1,-1)w$.

By using Lemma \ref{average}, we can construct lowest weight vector for the action of the Heisenberg Lie algebra
$\hat{H^N}$ for any $N >0$.
Let $\al^1,\dots,\al^N \in G^\perp[2]$ be mutually orthogonal codewords.
For $s=(s_1,\dots,s_N) \in \{\pm \}^N$,
set $h(s)=\Pi_{i=1}^N (1+s_i h_l^i)$.
Then, by Lemma \ref{average}, $h(s)t_\rreg \in \Om_{F,H^N}^{\ft \sum_{i=1}^N s_i(h_l^i+h_r^i)}$.
\begin{lem}
\label{average_func}
For $s^0,s^1,s^2,s^3 \in \{\pm \}^N$,
\begin{align*}
\langle
&Y(h(s^0)\cdot t_{\rreg},\uz_0)Y(h(s^1)\cdot t_{\rreg},\uz_1)Y(h(s^2)\cdot t_{\rreg},\uz_2)Y(h(s^3) \cdot t_{\rreg},\uz_3)\1
\rangle\\
&=2^{-r+3N} \delta_{s^0+s^1+s^2+s^3,0}F(z_0,z_1,z_2,z_3)^{r-2N} 
\Pi_{0 \leq i < j \leq 3} \left((z_i-z_j)(\z_i-\z_j)\right)^{\frac{1}{4}(p^i ,p^j)+\frac{1}{4}N -\frac{r}{8}}.
\end{align*}
\end{lem}
\begin{proof}
We first consider the case of $s^0=s^1=s^2=s^3=+^N$.
By dividing $N$ into $8$ regions, by Proposition \ref{correlator2}, we have
\begin{align*}
&Y(h(s^0)\cdot t_{\rreg},\uz_0)Y(h(s^1)\cdot t_{\rreg},\uz_1)Y(h(s^2)\cdot t_{\rreg},\uz_2)Y(h(s^3) \cdot t_{\rreg},\uz_3)\1\rangle\\
&=
2^{-r}\Pi_{0 \leq i < j \leq 3} \left((z_i-z_j)(\z_i-\z_j)\right)^{-\frac{r}{8}} F(z_0,z_1,z_2,z_3)^{r} \sum_{k_1,k_2,\dots,k_8} \frac{N!}{k_1!k_2!\cdots k_8!} \\
&\times
\left(-\frac{G_{01,23}(z_0,z_1,z_2,z_3)}{F(z_0,z_1,z_2,z_3)} \right)^{2k_1+2k_4}
\left(-\frac{G_{02,13}(z_0,z_1,z_2,z_3)}{F(z_0,z_1,z_2,z_3)} \right)^{2k_2+2k_5}
\left(-\frac{G_{03,12}(z_0,z_1,z_2,z_3)}{F(z_0,z_1,z_2,z_3)} \right)^{2k_3+2k_6}\\
&=2^{-r}\Pi_{0 \leq i < j \leq 3} \left((z_i-z_j)(\z_i-\z_j)\right)^{-\frac{r}{8}} F(z_0,z_1,z_2,z_3)^{r}
\left(2-2\frac{G_{01,23}^2+G_{02,13}^2+G_{03,12}^2}{F^2}
\right)^N
\end{align*}
Since 
$G_{01,23}^2+G_{02,13}^2+G_{03,12}^2 = F^2$,
the four point correlation function is equal to zero.
We next consider the case of $s^0=s^1=s^2=+^N$ and $s^3=-^N$.
Then, the four point correlation function is
\begin{align*}
&2^{-r}\Pi_{0 \leq i < j \leq 3} \left((z_i-z_j)(\z_i-\z_j)\right)^{-\frac{r}{8}} F(z_0,z_1,z_2,z_3)^{r} \sum_{k_1,k_2,\dots,k_8} \frac{N!}{k_1!k_2!\cdots k_8!}(-1)^{k_3+k_4+k_5+k_7} \\
&\times \left(-\frac{G_{01,23}(z_0,z_1,z_2,z_3)}{F(z_0,z_1,z_2,z_3)} \right)^{2k_1+2k_4}
\left(-\frac{G_{02,13}(z_0,z_1,z_2,z_3)}{F(z_0,z_1,z_2,z_3)} \right)^{2k_2+2k_5}
\left(-\frac{G_{03,12}(z_0,z_1,z_2,z_3)}{F(z_0,z_1,z_2,z_3)} \right)^{2k_3+2k_6}\\
&=0.
\end{align*}
Finally, we consider the case of $s^0=s^1=+^N$ and $s^2=s^3=-^N$.
Then, the four point correlation function is
\begin{align*}
&2^{-r}\Pi_{0 \leq i < j \leq 3} \left((z_i-z_j)(\z_i-\z_j)\right)^{-\frac{r}{8}} F(z_0,z_1,z_2,z_3)^{r} \sum_{k_1,k_2,\dots,k_8} \frac{N!}{k_1!k_2!\cdots k_8!} \\
&\times
\left(-\frac{G_{01,23}(z_0,z_1,z_2,z_3)}{F(z_0,z_1,z_2,z_3)} \right)^{2k_1+2k_4}
\left(+\frac{G_{02,13}(z_0,z_1,z_2,z_3)}{F(z_0,z_1,z_2,z_3)} \right)^{2k_2+2k_5}
\left(+\frac{G_{03,12}(z_0,z_1,z_2,z_3)}{F(z_0,z_1,z_2,z_3)} \right)^{2k_3+2k_6}\\
&=2^{-r}\Pi_{0 \leq i < j \leq 3} \left((z_i-z_j)(\z_i-\z_j)\right)^{-\frac{r}{8}} F(z_0,z_1,z_2,z_3)^{r}
\left(2-2\frac{G_{01,23}^2-G_{02,13}^2-G_{03,12}^2}{F^2}
\right)^N\\
&=2^{-r+3N} \Pi_{0 \leq i < j \leq 3} \left((z_i-z_j)(\z_i-\z_j)\right)^{-\frac{r}{8}} F(z_0,z_1,z_2,z_3)^{r-2N}
\left((z_0-z_1)(z_2-z_3)(\z_0-\z_1)(\z_2-\z_3)\right)^{\ft N}
\end{align*}
It is easy to derive the general formula from the above computations.
\end{proof}

By Lemma \ref{average}, Lemma \ref{average_func}
and Theorem \ref{deformation}, we have:
\begin{thm}
\label{deform_correlator}
Assume that there exists mutually orthogonal codewords $\al^1,\dots,\al^N \in G^\perp[2]$.
Then, the code CFT $F_G$ admits a current-current deformation parametrized by $\mathrm{O}(N,N)/ \mathrm{O}(N)\times \mathrm{O}(N)$.
Furthermore, for $\si \in \mathrm{O}(N,N)$
and $s^0,s^1,s^2,s^3 \in \{\pm \}^N$, the deformed four point function satisfies
\begin{align*}
\langle
&Y_\si(h(s^0)\cdot t_{\rreg},\uz_0)Y_\si(h(s^1)\cdot t_{\rreg},\uz_1)Y(h(s^1)\cdot t_{\rreg},\uz_2)Y_\si(h(s^3) \cdot t_{\rreg},\uz_3)\1
\rangle\\
&=2^{-r+3N} \delta_{p^0+p^1+p^2+p^3,0}F(z_0,z_1,z_2,z_3)^{r-2N} 
\Pi_{0 \leq i < j \leq 3} \left((z_i-z_j)(\z_i-\z_j)\right)^{\frac{1}{4}\left(p \si^{-1} (s^i,s^i) ,p \si^{-1} (s^j,s^j)\right)_l + \frac{1}{4}N -\frac{r}{8}}.
\end{align*}
\end{thm}



\section{Appendix}
Let $l,r \in \Z_{\geq 0}$.
In Appendix,
we will define the category of $(l,r)$-framed algebras and prove that
it is equivalent to the category of $(l,r)$-framed full vertex operator algebras.

The symmetric group $S_l$ (reps. $S_r$) naturally acts on the left component (resp. the right component) of $\Isr = \Is^l \times \Is^r$.
Similarly to the case of framed full vertex operator algebras,
a morphism from a $(l,r)$-framed algebra $S$ to $T$ is a pair of
a permutation $(g,\bar{g}) \in S_l \times S_r$ and a linear map $f:S \rightarrow T$
such that:
\begin{enumerate}
\item
For any $\mu \in \Isr$,
$f(S_\mu)=T_{(g,\bar{g})\mu}$.
\item
$f(1)=1$;
\item
$f(a\cdot b)=f(a) \cdot f(b)$.
\end{enumerate}
We denote the category of framed algebras by $\FA$.
For any framed full vertex operator algebra $F$,
by Lemma \ref{prevertex} and Proposition \ref{vertex}
$S_F$ is a framed algebra.
It is easy to show that this correspondence 
$$
S:\FF \rightarrow \FA,\;\; F \mapsto S_{F}
$$
gives a functor.
Conversely, for any framed algebra $S$,
by Theorem \ref{construction}, $F_S$ is a framed full vertex operator algebra.
Let $(f,(g,\bar{g})):S \rightarrow T$ be a morphism between framed algebras $S$ and $T$.
Then, we have a unique linear map $\tilde{f}:F_S \rightarrow F_T$
such that:
\begin{enumerate}
\item
the restriction of $\tilde{f}$ on $S$ coincides with $f$;
\item
$\tilde{f}$ is a  $\Vir^l\times \Vir^r$-module homomorphism.
\end{enumerate}
By using Lemma \ref{ward_induction},
we can show that $\tilde{f}$ is a full vertex algebra homomorphism.
Hence, the correspondence 
$$
F:\FA \rightarrow \FF,\;\; S \mapsto F_S
$$
also gives a functor.
Then, we have:
\begin{thm}
\label{equivalence}
The functor $F:\FA\rightarrow \FF$ gives an equivalence of categories.
The inverse functor is given by $S:\FF \rightarrow \FA$.
\end{thm}

\end{document}